\documentclass{amsart}

\usepackage[utf8]{inputenc}  
\usepackage[T1]{fontenc} 
\usepackage{amsmath,amssymb,amsthm} 
\usepackage[pagebackref=false]{hyperref} 
\usepackage{booktabs}
\usepackage{microtype} 
\usepackage{csquotes}
\usepackage[inline]{enumitem} 
    \setlist[enumerate]{label=\textup{(\roman*)}}
\usepackage{tikz}
\usepackage{nicematrix}
\usepackage[all,cmtip]{xy} 

\newenvironment{frcseries}{\fontfamily{frc}\selectfont}{}
\newcommand{\textfrc}[1]{{\frcseries#1}}
\newcommand{\mathfrc}[1]{\text{\footnotesize\rm\textfrc{#1}}}

\usetikzlibrary{arrows,positioning,cd,automata,backgrounds,calc,decorations.pathreplacing,decorations.markings,calligraphy}
\tikzset{cutting/.style={draw,fill=white,circle,minimum size=3pt,inner
    sep=0pt}}

\numberwithin{equation}{section}

\newcommand*{\nn}{\mathbb{N}}
\newcommand*{\zz}{\mathbb{Z}}
\newcommand*{\isom}{\cong}
\newcommand*{\from}{\colon}
\newcommand*{\emptyw}{\varepsilon}
\newcommand*{\green}[1]{\ensuremath{\mathrel{\mathcal{#1}}}}
\newcommand*{\pro}[1]{\widehat{#1}}
\newcommand*{\prov}[2]{{#2}^{\pv {#1}}}
\newcommand*{\bsigma}{{\boldsymbol{\sigma}}}
\newcommand*{\btau}{{\boldsymbol{\tau}}}
\newcommand*{\btheta}{{\boldsymbol{\theta}}}
\newcommand*{\bzeta}{{\boldsymbol{\zeta}}} 
\newcommand*{\shift}{T}
\newcommand*{\pv}[1]{\ensuremath{\mathsf{#1}}} \newcommand*{\loc}{{L}}
\newcommand*{\pvo}[1]{{\ensuremath{\overline{\mathsf{#1}}}}}
\newcommand*{\init}{\alpha} \newcommand*{\term}{\omega}
\newcommand*{\edg}{E} \newcommand*{\vertex}{V}
\newcommand*{\rstab}{\operatorname{Stab}}
\newcommand*{\catsig}{\ensuremath{\mathcal{C}}}

\newcommand*{\Om}[2]{\ensuremath{\overline{\Omega}_{#1}{\pv{#2}}}}
\newcommand*{\ltr}[1]{\mathtt{#1}}
\newcommand*{\catpro}{\mathbf{Pro}}
\renewcommand*{\varnothing}{\emptyset}
\newcommand*{\pli}{\mathfrc{h}}

\newcommand*{\card}{\operatorname{Card}}
\newcommand*{\fac}{\operatorname{fac}}
\newcommand*{\img}{\operatorname{Im}}
\newcommand*{\en}{\operatorname{End}}
\renewcommand*{\hom}{\operatorname{Hom}}
\newcommand*{\clos}[1]{\operatorname{Cl}_{\pv{#1}}}

\theoremstyle{plain}
\newtheorem{theorem}{Theorem}[section]
\newtheorem{lemma}[theorem]{Lemma}
\newtheorem{proposition}[theorem]{Proposition}
\newtheorem{corollary}[theorem]{Corollary}
\newtheorem{problem}[theorem]{Problem}

\theoremstyle{definition}
\newtheorem{definition}[theorem]{Definition}

\theoremstyle{remark} 
\newtheorem{remark}[theorem]{Remark}
\newtheorem{example}[theorem]{Example}

\begin{document}

\title[Profinite approach to S-adic shift spaces I: Saturating sequences.]{Profinite approach to S-adic shift spaces I: Saturating directive sequences.}  

\author[J. Almeida]{Jorge Almeida}
\address{CMUP, Dep.\ Matem\'atica, Faculdade de Ci\^encias, Universidade do Porto, Rua do Campo Alegre 687, 4169-007 Porto, Portugal}
\email{jalmeida@fc.up.pt}

\author[A. Costa]{Alfredo Costa}
\address{University of Coimbra, CMUC, Department of Mathematics, Largo D. Dinis, 3000-143, Coimbra, Portugal}
\email{amgc@mat.uc.pt}

\author[H. Goulet-Ouellet]{Herman Goulet-Ouellet}
\address{Universit\'e de Moncton, 60 Notre-Dame-du-Sacr\'e-Coeur Street, Moncton E1A 3E9, New-Brunswick, Canada}
\email{herman.goulet-ouellet@umoncton.ca}

\begin{abstract}
  This paper is the first in a series of three, about (relatively)
  free profinite semigroups and S-adic representations of minimal
  shift spaces. We associate to each primitive S-adic directive
  sequence $\bsigma$ a \emph{profinite image} in the free profinite
  semigroup over the alphabet of the induced minimal shift space. When
  this profinite image contains a $\green{J}$-maximal maximal subgroup
  of the free profinite semigroup (which, up to isomorphism, is called
  the \emph{Sch\"utzenberger group} of the shift space), we say that
  $\bsigma$ is \emph{saturating}. We show that if $\bsigma$ is
  recognizable, then it is saturating. Conversely, we use the notion
  of saturating sequence to obtain several sufficient conditions for
  $\bsigma$ to be recognizable: $\bsigma$ consists of pure encodings;
  or $\bsigma$ is eventually recognizable, saturating and consists of
  encodings; or $\bsigma$ is eventually recognizable, recurrent,
  bounded and consists of encodings. For the most part, we do not
  assume that $\bsigma$ has finite alphabet rank although we establish
  that this combinatorial property has important algebraic
  consequences, namely that the rank of the Sch\"utzenberger group is
  also finite, whose maximum possible value we also determine.
  We also show that for every  minimal shift space of finite topological rank,
the rank of its Sch\"utzenberger group  is a lower bound of the topological rank.
\end{abstract}

\date{\today}


\makeatletter
\@namedef{subjclassname@2020}{%
  \textup{2020} Mathematics Subject Classification}
\makeatother

\maketitle

\tableofcontents

\section{Introduction}

\label{s:intro}

This article is the first in a series of three papers linking
minimal shift spaces, via their S-adic representations,
with free profinite semigroups (cf.~\cite{Almeida&ACosta&Goulet-Ouellet:2024c,Almeida&ACosta&Goulet-Ouellet:2024d} for the ensuing two papers).
Finitely generated free profinite semigroups are completions of free semigroups by a natural metric. The elements of free profinite semigroups are called \emph{pseudowords}. In this first paper, we apply methods relying on a sort of ``algebraic combinatorics on pseudowords'' to obtain necessary and sufficient conditions for a primitive S-adic representation to be recognizable.
The main contribution of this paper is the introduction and exploration of the notion of \emph{saturating} directive sequence,
and its associated machinery, to obtain new results about recognizable directive sequences.

S-adic representations of minimal shift spaces are an important subject of symbolic
dynamics, that in the past few decades has received a lot of attention (as seen, for example, in the books~\cite{Fogg:2002,Durand&Perrin:2022}
and in the survey~\cite{Berthe&Delacroix:2014}). Symbolic dynamics has strong connections with the theory of automata and formal languages~\cite{Beal:1993, Lind&Marcus:1996, Fogg:2002, Lothaire:2001, Beal&Berstel&Eilers&Perrin:2021, Almeida&ACosta&Kyriakoglou&Perrin:2020b}.
Free profinite semigroups were involved in major advancements in that theory since the 1980s~\cite{Almeida:1994a, Rhodes&Steinberg:2009qt, Almeida&ACosta:2015hb,Straubing&Weil:2021}. Hence, it is not surprising there has been an emergence of direct links between
symbolic dynamics and free profinite semigroups. The first time that methods from symbolic dynamics were systematically employed in the theory of profinite semigroups was
in~\cite{Almeida:1999c}, where they were used to establish a
strong decidability property of the pseudovariety of all finite
$p$-groups. Shortly thereafter, the first author introduced a systematic connection between symbolic dynamics and free profinite semigroups, allowing him to associate to each irreducible shift space~$X$ a profinite group $G(X)$, naturally realized
as a maximal subgroup of the free profinite
semigroup over the alphabet of
$X$~\cite{Almeida:2003b,Almeida:2003a,Almeida:2003cshort,Almeida:2005c}.
The group $G(X)$, called the \emph{Sch\"utzenberger group of~$X$} since the paper~\cite{Almeida:2003a}, has dynamical significance:
it is a flow invariant~\cite{ACosta&Steinberg:2021}. 
If $X$ is sofic and non-periodic, then $G(X)$
is a free profinite group of rank $\aleph_0$~\cite{ACosta&Steinberg:2011}.
Besides the sofic case, the computation of $G(X)$ has only been made
for minimal shift spaces, a case in which it has also been shown to have geometric
significance involving Rauzy graphs \cite{Almeida&ACosta:2016b}. Such
calculations have been carried out mostly for substitutive shift spaces~\cite{Almeida:2005c,Almeida&ACosta:2013,Goulet-Ouellet:2022d,Goulet-Ouellet:2022c}, but not always~\cite{Almeida&ACosta:2016b}. The landscape of possibilities for $G(X)$ when $X$ is minimal seems rich, and remains largely unexplored.

In the series of three papers initiated here, we go beyond the substitutive case by systematically expanding to minimal S-adic shift spaces our study of the interplay between free profinite semigroups and symbolic dynamical systems, specially through the Sch\"utzenberger groups of the latter.
Fixing an S-adic representation for a shift space allows us to see it as a sort of limit of substitutive spaces, thus
suggesting a way to approach spaces that are non-substitutive by adapting what was done for the substitutive ones. This kind of approach is sketched in~\cite{Almeida:2003a} to investigate
the Sch\"utzenberger groups of Arnoux--Rauzy shift spaces.

In this article we focus on connections with the notion of \emph{recognizable} directive sequence.
Mossé's celebrated theorem, stating that every aperiodic primitive substitution
is recognizable \cite{Mosse:1992,Mosse:1996}, is crucial for the deduction, by the first two authors, of presentations for $G(X)$ when~$X$ is defined by a primitive substitution~\cite{Almeida&ACosta:2013}. Her theorem was extended and refined by several
authors~\cite{Bezuglyi&Kwiatkowski&Medynets:2009,Crabb&Duncan&McGregor:2010,Klouda&Starosta:2019}. This led to far-reaching generalizations by Berthé
et.~al.~\cite{Berthe&Steiner&Thuswaldner&Yassawi:2019} concerning primitive directive S-adic sequences, which, in conjunction
with earlier investigations of the group $G(X)$, motivated the work in the present paper. Further generalizations of Berthé et al.'s results appear in work by Béal et al.~\cite{Beal&Perrin&Restivo&Steiner:2023}.
Also testifying to their importance, we mention that recognizable directive sequences provide representations of S-adic shift spaces by Bratteli--Vershik systems~\cite[Theorem 6.5]{Berthe&Steiner&Thuswaldner&Yassawi:2019}. 

At this point, it is convenient to provide some technical context. An
S-adic directive sequence is a sequence $\bsigma=(\sigma_n)_{n\in\nn}$
of substitutions (i.e., homomorphisms $\sigma_n\from A_{n+1}^+\to
A_n^+$ between free semigroups) defining in a natural way a minimal
shift space $X(\bsigma)$ consisting of all biinfinite words over the
alphabet $A_0$ whose finite factors occur as factors of arbitrarily
far iterated images along the sequence $\bsigma$. We then also say
that $\bsigma$ is an S-adic representation of $X(\bsigma)$. Roughly
speaking, $\bsigma$ is recognizable when, denoting by $\bsigma^{(k)}$
the subsequence $(\sigma_n)_{n\geq k}$, every element of
$X(\bsigma^{(k)})$ has a unique ``de-substitution``, via $\sigma_k$,
as an element of $X(\bsigma^{(k+1)})$, for every $k\in\nn$. Quite
often, one needs to assume that $\bsigma$ is \emph{bounded}, meaning
that the sequence of cardinalities $\card(A_n)$ is bounded; or at
least that $\bsigma$ has \emph{finite alphabet rank}, meaning that
$\lim\inf\card(A_n)<\infty$. One of the most remarkable results
from~\cite{Berthe&Steiner&Thuswaldner&Yassawi:2019}, generalizing
Mossé's theorem, is that if $\bsigma=(\sigma_n)_{n\in\nn}$ is a
primitive directive sequence with finite alphabet rank, such that
$X(\bsigma)$ is aperiodic, then $\bsigma$ is \emph{eventually
  recognizable} (i.e., $\bsigma^{(k)}$ is recognizable for some
$k\in\nn$).

The central notion of this paper is that of \emph{saturating directive sequence}.
Briefly speaking,
a primitive directive sequence $\bsigma=(\sigma_n)_{n\in\nn}$ is saturating when a certain natural realization of $G(X(\bsigma))$ is contained in the intersection
of the images of the profinite extensions of the homomorphisms $\sigma_0\circ\cdots\circ\sigma_n$; we call this intersection the \emph{profinite image} of $\bsigma$.
The profinite image of $\bsigma$ is a group if $\bsigma$ is proper (Theorem~\ref{t:proper-is-group}); it is a simple semigroup if, and only if, all limit
words of $\bsigma$ belong
to $X(\bsigma)$ (Theorem~\ref{t:characterization-of-image-inside-Jsigma}).

The following theorem collects several of our main applications of saturation to recognizability
(cf.~Theorems~\ref{t:surjectivity}, \ref{t:a-sort-of-converse-of-surjectivity-theorem}, \ref{t:recurrent-encoding-made-implies-recognizable},
and Corollary~\ref{c:surjectivity-pure-encoding}). By saying that $\bsigma$ is an \emph{encoding} we mean that
$\sigma_n$ is an injective homomorphism for each $n\in\nn$, and by saying that it is \emph{pure} we mean that
it is an encoding such that, for each $n\in\nn$, the image of $\sigma_n$ is a pure code.
Moreover, we say that $\bsigma$ is \emph{recurrent} if for every $n\in\nn$ there exists $m>0$ such that $(\sigma_{0},\ldots, \sigma_{n}) = (\sigma_{m},\ldots, \sigma_{m+n})$.

\begin{theorem}
  \label{t:introduction-0}
  Let $\bsigma$ be an eventually recognizable primitive directive sequence. The following statements hold:
  \begin{enumerate}
  \item if $\bsigma$ is recognizable, then it is saturating;
    \item if $\bsigma$ is pure, then $\bsigma$ is recognizable;
  \item if $\bsigma$ is saturating and encoding, then $\bsigma$ is recognizable;
  \item if $\bsigma$ is recurrent, bounded, and encoding, then $\bsigma$ is recognizable.
  \end{enumerate}
\end{theorem}

Other sufficient conditions for recognizability of $\bsigma$ have been
obtained before. Berthé et al.~\cite[Theorem
4.6]{Berthe&Steiner&Thuswaldner&Yassawi:2019} showed that $\bsigma$ is
\emph{fully recognizable} (a property stronger than being
recognizable) if $X(\bsigma)$ is aperiodic and for each $n\in\nn$ the
homomorphism $\sigma_n\from A_{n+1}^+\to A_n^+$ satisfies one of the
following conditions: $\card(A_{n+1})=2$, the rank of the incidence
matrix of $\sigma_n$ is $\card(A_{n+1})$, or $\sigma_n$ is
rotationally conjugate to a left or right permutative homomorphism.
Bustos-Gajardo et al.~showed, again assuming aperiodicity of
$X(\bsigma)$, that $\bsigma$ is recognizable if each term $\sigma_n$
appears infinitely often in $\bsigma$ and is a constant-length
encoding, cf.~\cite[Lemma 3.4 and Theorem
3.6]{Bustos-Gajardo&Manibo&Yassawi:2023}.

Theorem~\ref{t:introduction-0} gives further links between symbolic
dynamics and free profinite semigroups. The latter are applied in the
proofs of all four statements included in the theorem, with the notion
of saturating sequence playing a key role in all of them. On the other hand, Theorem~\ref{t:introduction-0} is used to
obtain upper bounds for the rank of the profinite group $G(X)$, when
$X$ has finite alphabet rank (cf.~Corollary~\ref{c:upper-bound-saturating}).
We then deduce that the rank of $G(X)$ is a lower bound for the \emph{topological rank} of~$X$ (Theorem~\ref{t:finite-top-rank-implies-finitely-generated-Schutz}).
The determination of the topological rank of a minimal shift space is a difficult problem, frequently
approached with the help of the \emph{dimension group} of the space, as the rank
of that group is a lower bound of the topological rank~\cite{Durand&Perrin:2022}. In this context, it is worthy to point out
that, for example, if $X$ is the Prouhet-Thue Morse shift space, then the rank of $G(X)$ equals the topological rank of $X$, which is three,
while the dimension group of $X$ has rank two (Example~\ref{eg:ptm-comparing-ranks}). In fact, we are not aware of any example of a minimal shift space $X$ where the rank of $G(X)$ differs from the topological rank of~$X$. This state of affairs provides further motivation for future research
on Sch\"utzenberger groups of minimal shift spaces. In the two sequel papers~\cite{Almeida&ACosta&Goulet-Ouellet:2024c,Almeida&ACosta&Goulet-Ouellet:2024d} we obtain more information about these groups with the help of results from the present paper.

When delving into the proof of Theorem~\ref{t:introduction-0}, the
reader will notice our option to refine the concept of saturation by
considering free profinite semigroups relatively to pseudovarieties of
finite semigroups. A pseudovariety of finite semigroups is a class of
finite semigroups that is closed under taking finite products,
subsemigroups, and quotients. This type of class provides one of the
main frameworks for the study of finite semigroups and formal
languages, particularly via Eilenberg's Correspondence
Theorem~\cite{Eilenberg:1976}. This is enough as motivation to also
consider the image $G_{\pv V}(X)$ of $G(X)$ in the free pro-$\pv V$
semigroup over the alphabet of~$X$, when $\pv V$ is a pseudovariety of
finite semigroups; we say that $G_{\pv V}(X)$ is the \emph{\emph{\pv
    V}-Sch\"utzenberger group} of $X$. Note that
$\pv{V}$-Schützenberger groups are an extension of the original notion
of Schützenberger groups, in view of the equality $G(X)=G_{\pv S}(X)$,
where~$\pv S$ is the pseudovariety of all finite semigroups. Going
back to saturation, the corresponding refinement for the notion of
saturating directive sequence is that of \emph{$\pv V$-saturating}
directive sequence; the saturating sequences mentioned in
Theorem~\ref{t:introduction-0} are precisely the $\pv S$-saturating
sequences. Considering $\pv V$-saturating sequences, for $\pv V$ other
than $\pv S$, allows for more clarity in the proof of
Theorem~\ref{t:introduction-0} and enlarges its scope. It also
prepares the path to results about \emph{\pv V}-Sch\"utzenberger
groups in the subsequent
papers~\cite{Almeida&ACosta&Goulet-Ouellet:2024c,Almeida&ACosta&Goulet-Ouellet:2024d}.

We proceed by detailing how this paper is organized,
highlighting some of its content.
Preliminaries about symbolic dynamics and profinite semigroups are respectively given in the two sections following this introduction.
Immediately afterwards, we have a section dedicated to profinite categories. There, we improve Hunter's theorem
stating that the monoid of continuous endomorphisms of a finitely generated profinite semigroup is itself
a profinite monoid  for the pointwise topology \cite{Hunter:1983}: we extend it to any category of continuous homomorphisms between finitely many finitely generated profinite semigroups (cf.~Proposition~\ref{p:hom-is-a-profinite-category}). This improvement is necessary for Sections~\ref{sec:case-bound-direct} and~\ref{sec:models}, and for the proofs
of Theorem~\ref{t:recurrent-encoding-implies-saturating} and its closely related Theorem~\ref{t:recurrent-encoding-made-implies-recognizable}.
We also introduce free profinite categories and some of its properties, also needed for the same latter parts of the paper.

In Section~\ref{s:profinite-symbolic} we recapitulate existing results connecting minimal shift spaces with profinite semigroups, improving some of them
and establishing new ones. Part of the novelty comes from a more systematic consideration
in this study of all pseudovarieties of semigroups containing all finite local semilattices, and of the corresponding relatively free profinite semigroups.

In Section~\ref{sec:profinite-images} we introduce the profinite image
of an S-adic directive sequence~$\bsigma$, moreover establishing and
studying a natural inverse limit of the profinite images of the tails
$\bsigma^{(n)}$. In Section~\ref{sec:proper-case} we relate the
algebraic structure of the profinite image of $\bsigma$ with
combinatorial and dynamical aspects of $\bsigma$. In
Section~\ref{sec:case-bound-direct} we see that if the primitive
directive sequence $\bsigma$ is bounded, then the profinite image of
$\bsigma$ is the image of primitive continuous homomorphisms between
free profinite semigroups, obtained as cluster points of the sequence
of homomorphisms $\sigma_0\circ\sigma_1\circ\cdots \circ\sigma_n$.
Intuitively, this approximates even more the bounded case to the case
of substitutive shift spaces. Section~\ref{sec:models} further
develops the material of the preceding section by, among other things,
associating to each bounded primitive directive sequence a certain set
of continuous idempotent endomorphisms (of finitely generated
profinite semigroups), which we call \emph{kernel endomorphisms}. The
kernel endomorphisms play a key role in the ensuing
papers~\cite{Almeida&ACosta&Goulet-Ouellet:2024c,Almeida&ACosta&Goulet-Ouellet:2024d}.

Section~\ref{sec:satur-direct-sequ} contains the main results of the
paper. There, we introduce and study saturating directive sequences,
leading to the deduction of necessary or sufficient conditions,
summarized in Theorem~\ref{t:introduction-0}, for an S-adic directive
sequence to be recognizable.

Finally, in Section~\ref{sec:rank}, we restrict attention to primitive
directive sequences with finite alphabet rank. Applying several
results from earlier sections, we show the corresponding
Sch\"utzenberg groups have finite rank for which we determine sharp
upper bounds, depending on the pseudovariety of semigroups over which
they are considered. Moreover, we show that for every  minimal shift space of finite topological rank,
the rank of its Sch\"utzenberger group  is a lower bound of the topological rank.

\section{Symbolic dynamics}
\label{s:dynamics}

This section aims to provide some background on symbolic dynamics; the reader is referred to \cite{Fogg:2002} for a more in-depth
introduction, particularly in what concerns substitutions. 

\subsection{Basic notions}
Let $A$ be an alphabet, that is, a nonempty set whose elements we call
\emph{letters}. We denote by $A^*$ the set of all words over $A$,
including the empty word $\emptyw$, and we let $A^+ =
A^*\setminus\{\emptyw\}$; under the operation of word concatenation,
$A^*$ is the free monoid on $A$, and $A^+$ is the free semigroup. The
length of a word $w\in A^*$ is denoted $|w|$, while the number of
occurrences of a letter $a\in A$ in $w$ is denoted $|w|_a$. Counting
from the left starting at $0$, the letter in position
$i\in\{0,\ldots,|w|-1\}$ is denoted $w[i]$. For $0\leq i\leq j\leq
|w|$, we let $w[i,j) = w[i]\cdots w[j-1]$. Note that $w[i,i) =
\emptyw$. We denote by $\fac(w)$ the set of \emph{nonempty} factors of
$w$, that is
\begin{equation*}
    \fac(w) = \{ w[i,j) : 0\le i< j\leq |w|\}.
\end{equation*}
Factors of the form $w[0,j)$ are further called \emph{prefixes}, while those of the form $w[i,|w|)$ are called \emph{suffixes}.
We make the choice of excluding the empty word from $\fac(w)$ because it will often be more convenient to work with free semigroups rather than free monoids.

Let $A^\zz$ be the set of two-sided infinite words over $A$. Given $x\in A^\zz$ and
$i\in\zz$, we let $x[i]$ be the letter of $x$ on position $i$.
If $i,j\in\zz$ are such that $i\leq j$,
we may consider the word $x[i,j)=x[i]\cdots x[j-1]$. Observe that $x{[}i,i{)}=\varepsilon$. The set
\begin{equation*}
\fac(x) = \{ x[i,j) : i< j \}  
\end{equation*}
is the set of \emph{nonempty} factors of $x$. \textit{Mutatis mutandis}, we make similar definitions for \emph{right infinite words} and \emph{left infinite words}, that is elements of $A^\nn$ and $A^{\zz_{-}}$, where $\nn$  and $\zz_{-}$  respectively stand for the set of nonnegative and the set of negative integers.
For $x\in A^{\zz_{-}}$ and $y\in A^\nn$, we denote by $x\cdot y$
the element $z$ of $A^\zz$ such that $z[i]=x[i]$ if $i<0$ and $z[i]=y[i]$ if $i\geq 0$.

At this point, we assume that $A$ is finite and we endow it with the
discrete topology, and $A^\zz$ with the corresponding product
topology. The \emph{shift map} is the homeomorphism $\shift\from A^\zz
\to A^\zz$ defined by $\shift(x) = (x[i+1])_{i\in\zz}$. A \emph{shift
  space} over the alphabet $A$ is a nonempty closed subset $X$ of
$A^\zz$ that satisfies $\shift(X)=X$. Note that the pair $(X,\shift)$ is a
topological dynamical system, and so one may apply to shift spaces terminology
from the theory of dynamical systems, such as that of
\emph{topological conjugacy}, which is natural the notion of isomorphism for dynamical
systems.

We focus primarily on shift spaces that are \emph{minimal} (for the
inclusion order). An infinite word $x\in A^\zz$ is \emph{periodic} if it has a finite $\shift$-orbit, and aperiodic otherwise; a shift space is called \emph{periodic} if it is the orbit of a periodic infinite word, and \emph{aperiodic} if it contains no periodic shift space.

The \emph{language} of a subset $X\subseteq A^\zz$ is the subset of $A^+$ defined by
\begin{equation*}
    L(X) = \bigcup_{x\in X}\fac(x).
\end{equation*}
It is well known that for two shift spaces $X, Y\subseteq A^\zz$, we have $L(X)\subseteq L(Y)$ if and only if $X\subseteq
Y$. The language of a shift space $X$ is both \emph{factorial} ($\fac(w)\subseteq L(X)$ for every $w\in L(X)$) and \emph{extendable} (if $w\in L(X)$, then $awb\in
L(X)$ for some $a,b\in A$); conversely, every
nonempty, factorial and extendable language is the language of a unique shift
space. Minimal shift spaces have the following  simple
characterization in terms of their languages. A language $L\subseteq A^+$ is called \emph{uniformly recurrent} if it
is factorial, extendable, and for every $u\in L$, there exists
$n\in\nn$ such that $u\in\fac(v)$ for every $v\in L$ with $|v|\geq n$. Then, a shift space $X$ is minimal if and
only if the language $L(X)$ is uniformly recurrent.

Consider a semigroup homomorphism  $\sigma\from A^+\to B^+$.
For each $x\in A^\zz$, the element $\sigma(x)$  of $B^\zz$ is defined by the equality
 \begin{equation*}
 \sigma(x)=\cdots\sigma(x[-2])\sigma(x[-1])\cdot\sigma(x[0])\sigma(x[1])\sigma(x[2])\cdots .
 \end{equation*}

\subsection{S-adic representations}
\label{ss:s-adic}

A common way of defining shift spaces is to use so-called S-adic representations, which we proceed to introduce.

Let $\bsigma=(\sigma_n)_{n\in\nn}$ be a sequence of homomorphisms of free semigroups $\sigma_n\from A_{n+1}^+\to A_n^+$, where $A_n$ is a finite alphabet for every $n\in\nn$.
We say that $\bsigma$ a \emph{directive sequence}. 
 The \emph{alphabet rank} of $\bsigma$
 is the limit
 $\liminf_{n\to\infty}\card(A_n)$
where $\card(S)$ denotes the cardinal of a set $S$.
For such a directive sequence $\bsigma$ and natural numbers $n\le m$, let $\sigma_{n,m}$ be the homomorphism $A_m^+\to A_n^+$ given by the composition
\begin{equation*}
  \sigma_{n,m}=\sigma_{n}\circ\cdots\circ\sigma_{m-1},
\end{equation*}
with the convention that $\sigma_{n,n}$ is the identity on $A_n^+$.
Consider the factorial language
\begin{equation*}
    L(\bsigma) = \bigcup_{n\geq 0}\bigcup_{a\in A_n} \fac(\sigma_{0,n}(a)).
  \end{equation*}
Let $X(\bsigma)$ be the set of elements $x$ of $A_0^\zz$ such that
$\fac(x)\subseteq L(\bsigma)$. The set $X(\bsigma)$ is a shift space when it is nonempty. We say that a shift space $X$ is \emph{represented} by the directive sequence $\bsigma$, or that $\bsigma$ is an \emph{S-adic representation of $X$}, when $X=X(\bsigma)$.

\begin{remark}
  \label{r:when-Xsigma-is-a-shift-space}
  One has $X(\bsigma)\neq \varnothing$ precisely when $L(\bsigma)$ is infinite,
which happens if and only if $\lim_{n\to\infty}\max_{a\in A_{n+1}}|\sigma_{0,n}(a)|=\infty$ (in some publications, this limit is part of the definition of directive sequence, e.g.~\cite{Berthe&CecchiBernales&Durand&Leroy:2021}).
The inclusion $L(X(\bsigma))\subseteq L(\bsigma)$ clearly holds, but it may be strict, even
if $X(\bsigma)\neq\emptyset$ (cf.~\cite[Example 1.4.5]{Durand&Perrin:2022}).

\end{remark}

\begin{remark}
  \label{r:terminology-directive-sequence}
  In the book of Durand and Perrin~\cite{Durand&Perrin:2022} the terminology \emph{directive sequence} is reserved for
   sequences $\bsigma$ such that $L(\bsigma)=L(X(\bsigma))$. On the other hand, our usage is adopted in many other relevant publications
    (cf.~\cite{Durand&Leroy&Richomme_2013,Berthe&Delacroix:2014,Berthe&CecchiBernales&Durand&Leroy:2021,Berthe&Steiner&Thuswaldner&Yassawi:2019}).
\end{remark}

Let $k\in\nn$. We denote by $\bsigma^{(k)}$
the \emph{tail sequence} given by $\bsigma^{(k)}=(\sigma_{n+k})_{n\in\nn}$.
A~proof of the following fact is found in~\cite[Lemma~4.2]{Berthe&Steiner&Thuswaldner&Yassawi:2019}.

\begin{lemma}
  \label{l:changing-levels-shift}
  For every $m,n\in\nn$ such that $m\geq n$,
  the shift space $X(\bsigma^{(n)})$ is the smallest one containing the set $\sigma_{n,m}(X(\bsigma^{(m)}))$.
\end{lemma}

Let $\varphi$ be a \emph{substitution} over the alphabet $A$, by which we mean an
endomorphism of $A^+$. In the special case where $\sigma_n=\varphi$
for all~$n\in\nn$, we denote $L(\bsigma)$ and $X(\bsigma)$
respectively by $L(\varphi)$ and $X(\varphi)$. Assuming moreover that
$X(\varphi)\neq \varnothing$, we say that $X(\varphi)$ is a
\emph{substitutive} shift space. We mention that in some sources
the equality $L(\varphi)=L(X(\varphi))$
is included in the definition of substitution (e.g.~\cite{Durand&Perrin:2022}).

When studying minimal shift spaces, it is often useful to focus on S-adic representations subject to special conditions, some of which we introduce next. We start with conditions on homomorphisms. 
\begin{definition}
  \label{d:properties-of-homomorphisms}
    A homomorphism $\varphi\from A^+\to B^+$ is called: 
    \begin{enumerate}
   \item \emph{expansive} if $|\varphi(a)|\geq 2$ for every $a\in A$;
    \item\label{item:def-positive}
      \emph{positive} if $\varphi$ is expansive and $B\subseteq\fac(\varphi(a))$ for every $a\in A$;
        \item \emph{circular} if it is injective and $uv,
          vu\in\varphi(A^+)$ implies $u, v\in\varphi(A^+)$ for every
          $u, v\in B^+$; 
        \item \emph{left proper} if there is a letter $b\in B$
          such that $\varphi(a)\in bB^*$ for every $a\in A$;
          \item \emph{right proper} if there is a letter $b\in B$
          such that $\varphi(a)\in B^*b$ for every $a\in A$;
        \item \emph{proper} if it is both  right proper and left proper. 
    \end{enumerate}
  \end{definition}

  In turn, when we say that a directive sequence $\bsigma = (\sigma_n)_{n\in\nn}$ is \emph{circular}, \emph{right proper}, \emph{left proper}, \emph{proper}, or \emph{positive}, we mean that $\sigma_n$ has that property for every $n\in\nn$.
  We also say that $\bsigma$ is \emph{encoding} if $\sigma_n$ is injective for every $n\in\nn$.

A slightly more subtle notion is that of primitive directive sequence. 
\begin{definition}
  \label{d:primitive-sequences}
A directive sequence $\bsigma = (\sigma_n)_{n\in\nn}$ is \emph{primitive} if, for every $n\in\nn$, there exists $m>n$ such that $\sigma_{n,m}$ is positive. 
An endomorphism $\varphi\from A^+\to A^+$ is \emph{primitive} if there is $k\geq 1$ such that $\varphi^k$ is positive.
\end{definition}

Clearly every tail of primitive directive sequence is itself primitive.
In some papers, for instance \cite{Berthe&Delacroix:2014,Leroy:2014}, primitive directive sequences are called instead \emph{weakly primitive}.
The following theorem is well known within the community studying minimal shift spaces and their S-adic representations. A proof can be found
in  \cite[Section 6.4.2]{Durand&Perrin:2022}.

\begin{theorem}\label{t:characterization-minimal-shifts}
  Let $X$ be a shift space. The following conditions are equivalent:
  \begin{enumerate}
  \item $X$ is a minimal shift space;\label{item:characterization-minimal-shifts-1}
  \item $X = X(\bsigma)$ for some primitive directive sequence $\bsigma$;\label{item:characterization-minimal-shifts-2}
  \item $X = X(\bsigma)$ for some proper, primitive and circular directive sequence $\bsigma$.\label{item:characterization-minimal-shifts-3}
  \end{enumerate}
  Moreover, if $\bsigma$ is a primitive directive sequence, then the equality $L(X(\bsigma))=L(\bsigma)$ holds.
\end{theorem}

\begin{remark}\label{r:primitive-implies-directive}
  In \cite[Proposition~6.4.5]{Durand&Perrin:2022} it is stated that
  if $\bsigma$ is a primitive directive sequence, then
  the equality $L(X(\bsigma))=L(\bsigma)$ holds under the extra assumption that the sequence is \emph{without bottleneck}, that is, $\card(A_n)\geq 2$ for all $n\in\nn$. We avoid this extra assumption in Theorem~\ref{t:characterization-minimal-shifts}
  because we force positive homomorphisms to be expansive also when the alphabet in the image has only one letter.
\end{remark}

\begin{remark}\label{r:guide-for-the-proof}
  Theorem~\ref{t:characterization-minimal-shifts} is essentially Proposition~6.4.5 from the book of Durand and Perrin~\cite{Durand&Perrin:2022},
  with two notable differences. First, our statement includes periodic shift spaces
  because we allow bottleneck (cf.~Remark~\ref{r:primitive-implies-directive}).
  Second, in the statement provided in the book there is no explicit reference to circular homomorphisms;
  but the proof found there gives what we write here,
  since the pertinent homomorphisms are encodings by \emph{return words}, well known to be circular encodings~(cf.~\cite[Lemma 17]{Durand&Host&Skau:1999}). In our companion paper~\cite{Almeida&ACosta&Goulet-Ouellet:2024c} one finds a more detailed discussion about the
  representation by a proper, primitive and circular directive sequence that follows from that proof.
\end{remark}

A useful operation on directive sequences is that of contraction, which consists in grouping consecutive homomorphisms in the sequence. More precisely, a \emph{contraction} (also called a \emph{telescoping} in many sources) of a
sequence of homomorphisms $\bsigma=(\sigma_n)_{n\in\nn}$ is a sequence of the
form $\btau=(\sigma_{n_k,n_{k+1}})_{k\in\nn}$, for some strictly
increasing sequence $(n_k)_{k\in\nn}$ of natural numbers such that
$n_0=0$. Note that, if $\bsigma$ is primitive, then $\bsigma$ has a
contraction which is positive; moreover, every contraction of $\bsigma$ is
primitive.
As seen next, under a very mild condition\footnote{This mild condition appears to be implicit in several sources where it is stated that taking a contraction does not change the shift being represented by the directive sequence (e.g.~\cite[Section~5.2]{Berthe&Steiner&Thuswaldner&Yassawi:2019} and~\cite[Section 6.4.1]{Durand&Perrin:2022})}, satisfied by primitive directive sequences,
the shift space $X(\bsigma)$ remains unchanged when passing to a contraction. The reader should bear this fact in mind.

\begin{lemma}
  \label{l:contraction-does-not-change-the-language}
  Let $\bsigma=(\sigma_n)_{n\in\nn}$ be a directive
  sequence with a contraction $\btau=(\sigma_{n_k,n_{k+1}})_{k\in\nn}$.
  Suppose that $A_n\subseteq\fac(\sigma_n(A_{n+1}))$ for every $n\in\nn$.
  Then, the equalities $L(\bsigma^{(n_k)})=L(\btau^{(k)})$ and $X(\bsigma^{(n_k)})=X(\btau^{(k)})$ hold for every $k\in\nn$.
\end{lemma}

Lemma~\ref{l:contraction-does-not-change-the-language}, whose proof is an easy exercise, does not hold if
we drop some inclusion $A_n\subseteq\fac(\sigma_n(A_{n+1}))$ (cf.~Exercise 1.27, and its solution, in the book~\cite{Durand&Perrin:2022}).

\subsection{Recognizability}

We proceed to give the necessary background on the important notion of \emph{recognizable} directive sequence, following the monograph~\cite{Durand&Perrin:2022} and the paper~\cite{Berthe&Steiner&Thuswaldner&Yassawi:2019}.  

Let $\sigma\from A^+\to B^+$ be a homomorphism, where $A$ and $B$ are finite alphabets. A \emph{$\sigma$-representation} of a point $y\in B^\zz$ is a pair $(k,x)$, where $k\in\nn$ and $x\in A^\zz$, satisfying $\shift^k\sigma(x)=y$. We say that it is \emph{centered} if, additionally, $k<|\sigma(x[0])|$. 
\begin{definition}[Dynamical recognizability]
  \label{d:dynamical-recognizability}
    Given $X\subseteq A^\zz$, we say that $\sigma$ is \emph{(dynamically) recognizable in $X$} if every $y\in B^\zz$ has at most one centered $\sigma$-representation $(k,x)$ with $x\in X$.
\end{definition}

In case $X=A^\zz$, we say instead that $\sigma$ is \emph{fully
recognizable}. Full recognizability has the following characterization (cf.~\cite[Proposition~1.4.32]{Durand&Perrin:2022}).

\begin{proposition}
  \label{p:circular-are-fully-recognizable}
  A homomorphism is fully recognizable if and only if it is circular.
\end{proposition}

We say that a directive sequence $\bsigma=(\sigma_n)_{n\in\nn}$ is
\emph{recognizable} if the homomorphism $\sigma_n$ is recognizable in
$X(\bsigma^{(n+1)})$ for every $n\in\nn$; and \emph{eventually
  recognizable} if this holds only for all but finitely many
$n\in\nn$. One should bear in mind the following remarkable result of
Berthé et~al.~\cite[cf.~Theorem
5.2]{Berthe&Steiner&Thuswaldner&Yassawi:2019}.
  
\begin{theorem}
  \label{t:aperiodic-implies-eventually-recognizable}
  Let $\bsigma$ be a primitive directive sequence with finite alphabet
  rank. If $X(\bsigma)$ is aperiodic, then $\bsigma$ is eventually
  recognizable.
\end{theorem}
  
There is also a pointwise version of recognizability. Fix $x\in A^\zz$ and a homomorphism $\sigma\from A^+\to B^+$; define the set of \emph{$\sigma$-cutting points of $x$} by:
\begin{equation*}
    C_\sigma(x) = \{ -|\sigma(x[i,0))| : i<0\}\cup\{0\}\cup \{ |\sigma(x[0,i))| : i> 0\}.
\end{equation*}
The following definition was introduced in a seminal paper by Mossé~\cite{Mosse:1992}.
\begin{definition}[Mossé's recognizability]\label{def:Mosse}
  \label{d:Mosse-recognizability}
  Let $x\in A^\zz$ and write $y=\sigma(x)$. We say that $\sigma$ is
  \emph{recognizable for $x$ in Mossé's sense} when, for some positive
  integer $\ell$ (called the \emph{constant of recognizability}), the
  following holds for every $m\in C_\sigma(x)$ and $n\in\zz$:
    \begin{equation*}
        y[m-\ell,m+\ell)=y[n-\ell,n+\ell) \implies n\in C_\sigma(x).
    \end{equation*}
\end{definition}

Under mild conditions, dynamical recognizability implies Mossé's
recognizability.
\begin{proposition}[{\cite[Theorem~2.5(1)]{Berthe&Steiner&Thuswaldner&Yassawi:2019}}]
  \label{p:mosse-sense}
    Let $\sigma\from A^+\to B^+$ be a homomorphism, $X\subseteq A^\zz$ be a shift space and $x\in X$ be such that $L(X) = \fac(x)$. If $\sigma$ is recognizable in $X$, then it is recognizable in Mossé's sense for $x$.
\end{proposition}
In particular, if $X$ is a minimal shift space, then $\sigma$ is recognizable in Mossé's sense for every $x\in X$.

\section{Profinite semigroups}
\label{ss:pre-semigroups}
We move on to review some elements of
semigroup theory, with a focus on profinite semigroups. We follow the definition of a semigroup as being a \emph{nonempty} set endowed with an associative binary operation (in some sources, such as the book of Rhodes and Steinberg~\cite{Rhodes&Steinberg:2009qt}, the empty set is considered to be a semigroup).

\subsection{Green's relations}
\label{ss:green-relations}

We briefly recall a few standard facts about Green's relations; a
thorough account may be found in any book covering basic semigroup
theory, for
instance~\cite{Clifford&Preston:1961,Lallement:1979,Howie:1995}.

Let $S$ be a semigroup, and $S^1$ be the smallest monoid containing
$S$ (obtained by adjoining to $S$ an identity element, generically
denoted 1, if needed). For $s,t\in S$, write:
\begin{itemize}
\item $s\le_{\green{R}}t$ (or say that $t$ is a \emph{prefix} of~$s$)
  when $sS^1\subseteq tS^1$;
\item $s\le_{\green{L}}t$ (or say that $t$ is a \emph{suffix} of~$s$)
  when $S^1s\subseteq S^1t$;
    \item $s\le_{\green{H}}t$ when $sS^1\subseteq tS^1$ and $S^1s\subseteq S^1t$;
    \item $s\le_{\green{J}}t$ (or say that $t$ is a \emph{factor}
      of~$s$) when $S^1sS^1\subseteq S^1tS^1$.
\end{itemize}
These are quasi-orders known as \emph{Green's quasi-orders}. They
induce four equivalence relations, respectively denoted $\green{R}$,
$\green{L}$, $\green{H}$ and $\green{J}$, called \emph{Green's
  equivalences}. By a classical theorem of Green, the maximal
subgroups (maximal for inclusion) of $S$ are precisely the
$\green{H}$-classes of its idempotent elements. We may write $H_s$ for
the $\green{H}$-class of $s$ and similarly for other Green's
equivalences. For any Green's relation $\green{K}\in
\{\green{R},\green{L},\green{H},\green{K}\}$, we may write
$\mathcal{K}_S$ instead of $\green{K}$, whenever we want to emphasize
that we are considering the relation $\green{K}$ in the semigroup $S$;
this may be needed when reasoning with different semigroups at the
same time.

In this paper, we deal mostly with \emph{compact semigroups}:
semigroups endowed with a compact topology for which the
multiplication is continuous (we include the Hausdorff property in the
definition of compactness). Note that finite semigroups equipped with
the discrete topology are compact semigroups. In compact semigroups,
all of Green's relations (quasi-orders and equivalences) are closed;
in particular, so are the equivalence classes of Green's equivalences. When $S$ is a compact semigroup which is not a monoid, then $S^1$ is viewed as a compact semigroup by considering the topological sum of $S$ and of the space $\{1\}$.

A useful property of compact semigroups is that they are
\emph{stable}, that is, the following implications hold for all
elements $s$ and $t$:
 \begin{align*}
   (s\le_{\green{R}} t\ \text{and}\ s\green{J}t) &\implies s\green{R} t,
   \\
   (s\le_{\green{L}} t\ \text{and}\ s\green{J}t) &\implies s\green{L} t.
\end{align*}
Stable semigroups $S$ enjoy several useful properties:
\begin{itemize}
\item Two elements $s$ and $t$ are $\green{J}$-equivalent if and only
  if there is $u$ such that $s\green{R}u\green{L}t$, if and only if
  there is $v$ such that $s\green{L}v\green{R}t$.
\item A $\green{J}$-class $J$ contains an idempotent if and only if
  each of its $\green{L}$-classes contains an idempotent; the same
  holds for $\green{R}$-classes. This is also equivalent to every
  element $S$ of $J$ being \emph{regular}, which means that $s\in
  sSs$. Whenever $J$ satisfies these
  equivalent conditions, we 
  call it a \emph{regular $\green{J}$-class}. Its maximal subgroups
  are then isomorphic to one another, continuously so in the compact
  case.
\item The intersection of every $\green{R}$-class with every
  $\green{L}$-class contained in the same $\green{J}$-class is an
  $\green{H}$-class.
\end{itemize}

A semigroup is called \emph{simple} if $\green{J}$ is the universal
relation.\footnote{Note that the notion of (ideal) simple semigroup is
  unrelated with the classical notion of (congruence) simple group: as
  a semigroup, every group is simple.} The reader unfamiliar with
semigroup theory is cautioned that in the literature one finds also
the \emph{completely simple semigroups}, which are the stable simple
semigroups.

A subset $F$ of a semigroup $S$ is said to be \emph{factorial} if it
is an upset for the quasi-order $\le_{\green{J}}$, that is, it is
closed under taking factors.

\subsection{Pseudovarieties of semigroups}
\label{ss:pseudovarieties}

A \emph{pseudovariety} of semigroups is a class of finite semigroups closed under taking subsemigroups, homomorphic images, and finite direct products. Examples include:
\begin{itemize}
    \item the class $\pv S$ of all finite semigroups;
    \item the class $\pv G$ of all finite groups;
    \item the \emph{trivial pseudovariety} $\pv I$ (with only one-element semigroups);
    \item the class $\pv A$ of all finite \emph{aperiodic semigroups} (semigroups whose subgroups are trivial);
    \item the class $\pv{Sl}$ of all finite \emph{semilattices} (commutative semigroups whose elements are idempotent);
    \item the class $\pv{CS}$ of all finite \emph{simple semigroups}, that is, semigroups
      where the relation $\green{J}$ is universal;
    \item the class $\pv N$ of all finite \emph{nilpotent semigroups} (a semigroup is nilpotent if it has a zero $0$ and $S^k=\{0\}$ for some $k\geq 1$).
  
\end{itemize}

There are several operators of interest on pseudovarieties. For this
paper and the ensuing companions~\cite{Almeida&ACosta&Goulet-Ouellet:2024c,Almeida&ACosta&Goulet-Ouellet:2024d}, the following are relevant.
\begin{itemize}
\item If $\pv{H}$ is a pseudovariety consisting of finite groups, then the
  class $\pvo{H}$ of all finite semigroups whose subgroups belong to
  $\pv{H}$ is a pseudovariety. Note that $\pvo I=\pv A$ and $\pvo G=\pv S$.
\item Given a pseudovariety of semigroups \pv{V}, the class $\pv{\loc V}$ of all finite semigroups~$S$ such that $eSe\in\pv{V}$ for all
  idempotents $e\in S$ is also a pseudovariety, called the \emph{local
    of \pv{V}}. (In this paper we need to consider the pseudovarieties $\pv{\loc{I}}$
  and $\pv{\loc{Sl}}$.)
\item For two pseudovarieties of semigroups \pv V and \pv W, their
  \emph{semidirect product} $\pv V*\pv W$ is the smallest pseudovariety containing all semidirect products of the form $S*R$
  with $S\in\pv V$ and $R\in\pv W$.
\end{itemize}

\subsection{Relatively free profinite semigroups}
\label{ss:rel-free}

This subsection serves to introduce profinite semigroups, and in
particular relatively free profinite semigroups. For more details on
this topic, see~\cite{Almeida:1994a,Rhodes&Steinberg:2009qt} and the
shorter \cite{Almeida:2003cshort}.

By an \emph{inverse system of functions} we mean a triple
\begin{displaymath}
  \bigl((X_i)_{i\in I}; (\varphi_{i,j})_{i,j\in I; i\leq j},I\bigr)
\end{displaymath}
where $(I,{\le})$ is a directed set, the $X_i$ are sets and each
$\varphi_{i,j}$ is a function from $X_j$ to~$X_i$ such that
$\varphi_{i,j}\circ\varphi_{j,k}=\varphi_{i,k}$ whenever $i\leq j\leq k$, with $\varphi_{i,i}$ being the identity on $X_i$. The
\emph{inverse limit} of such a system is the set
\begin{displaymath}
  \varprojlim X_i
  = \Bigl\{x\in\prod_{i\in I}X_i:
  \forall i,j\in I\ \bigl(i\leq j \implies \varphi_{i,j}(x_j)=x_i\bigr)\Bigr\},
\end{displaymath}
where $x_k$ denotes the $k$-component of~$x$. The inverse system is said to be \emph{surjective} if all the functions
$\varphi_{i,j}$ are surjective. The restricted component
projection $\varprojlim X_i\to X_j$ is called the \emph{natural
  $j$-projection}. In case each $X_i$ has the additional structure of being a topological space, an algebra or a small category, we
assume that each $\varphi_{i,j}$ is a morphism in the corresponding
category. Then the inverse limit stays in the same category viewed,
respectively, as a subspace, subalgebra or subcategory, of the product
$\prod_{i\in I}X_i$, sometimes with the exception of the case where $\varprojlim X_i=\emptyset$ (as, for example, we do not allow the empty set to be a semigroup).  If the inverse system is surjective and the $X_i$ are compact spaces, then the natural
projections are surjective, cf.~\cite[Corollary~3.2.15]{Engelking:1989}.

Let $\pv V$ be a pseudovariety of finite semigroups. In this paper, we
always consider finite semigroups to be equipped with the discrete
topology. A \emph{pro-$\pv V$ semigroup} is a topological semigroup
$S$ which is is isomorphic as a topological semigroup to an inverse
limit of members of $\pv{V}$. Equivalently, $S$ is compact and
\emph{residually $\pv{V}$}, in the sense that any two distinct
elements $x,y\in S$ take distinct values under some continuous
homomorphism $\varphi\from S\to R$ where $R\in\pv{V}$. In particular,
members of $\pv{V}$ are pro-\pv{V}; we also say that a pro-$\pv {S}$
semigroup is \emph{profinite}. Other such specialized terminology will
be introduced as needed.

For each pseudovariety of semigroups $\pv V$, the category of pro-\pv{V} semigroups has free objects, called \emph{free pro-\pv{V} semigroups}~\cite[Subsection 3.2]{Almeida:2003cshort}. The free pro-\pv{V} semigroup over a set $A$ is denoted $\Om AV$. 
It comes equipped with a mapping $\iota_\pv{V}\from A\to\Om AV$ such that the following universal property holds: for every mapping $f\from A\to
S$ with $S$ a pro-$\pv V$ semigroup, there exists a unique continuous
homomorphism $\prov Vf\from \Om AV\to S$ such that the following
diagram commutes:
\begin{equation*}
    \xymatrix{
        A \ar[r]^{\iota_\pv{V}} \ar[dr]_f & \Om AV \ar[d]^{\prov Vf}\\
        & S.
    }
\end{equation*}
Semigroups of the form $\Om AV$ for some pseudovariety of semigroups $\pv V$ are called \emph{relatively free}.

  If the alphabet $A$ is finite, then topology of $\Om AV$ is metrizable, for every pseudovariety $\pv V$~\cite[Subsection 3.4]{Almeida:2003cshort}.
  A metric generating the topology of $\Om AV$ is
  the following: for $u,v\in\Om AV$ such that $u\neq v$,
  their distance, denoted $d(u,v)$, is given by the equality
  $d(u,v)=2^{-r(u,v)}$
  where $r(u,v)$ is the smallest possible cardinal for
  a semigroup $S$ from $\pv V$ for which there is a continuous homomorphism $\varphi\from \Om AV\to S$ satisfying $\varphi(u)\neq\varphi(v)$.
  
  On the other hand,  $\Om AS$
  is not metrizable if $A$ is infinite (cf.~\cite{Almeida&Steinberg:2008}).

Let $\pv{V},\pv{W}$ be two pseudovarieties of semigroups with $\pv
W\subseteq \pv V$. For a given set~$A$, the universal property of $\Om
AV$ applied to the mapping $\iota_\pv{W}\from A\to\Om AW$ gives a
continuous onto homomorphism $p_{\pv{V},\pv{W}}\from\Om AV\to \Om AW$
such that the following diagram commutes:
\begin{displaymath}
  \xymatrix{
        A
        \ar[r]^{\iota_\pv{V}}
        \ar[dr]_{\iota_\pv{W}} &
        \Om AV
        \ar[d]^{p_{\pv{V},\pv{W}}}\\
        & \Om AW
      }
\end{displaymath}
We call $p_{\pv{V},\pv{W}}$ the \emph{natural projection} of $\Om AV$ onto $\Om AW$.

The mapping $\iota_\pv{V}$ extends uniquely to a homomorphism
$\iota_\pv{V}^+\from A^+\to\Om AV$. Whenever, $\pv V$ contains the
pseudovariety $\pv N$ of finite nilpotent semigroups, or the
pseudovariety $\pv G$ of finite groups, $\iota_\pv{V}^+$ is injective.
In particular, for such a pseudovariety $\pv V$, we may identify $A^+$
with the (dense) subspace $\iota_\pv{V}^+(A^+)\subseteq\Om AV$
whenever convenient. We denote by $\clos V(L)$ the closure in $\Om AV$
of a language $L\subseteq A^+$ viewed as a subset of $\Om AV$,
whenever \pv V is a pseudovariety of semigroups containing \pv{N} or
\pv G. In this context, the elements of $\Om AV$ may be seen as
generalizations of words, for which reason they are called
\emph{pseudowords}.

Consider a homomorphism $\varphi\from A^+\to B^+$ of free semigroups.
It follows from the universal property of free pro-$\pv V$ semigroups
that for every homomorphism $\varphi\from A^+\to B^+$ there is a
unique continuous homomorphism $\prov V\varphi\from \Om AV\to\Om BV$
such that the following diagram commutes
\begin{displaymath}
  \xymatrix{
    A^+
    \ar[d]^\varphi
    \ar[r]^{\hat\iota_{\pv V,A}^+}
    &
    \Om AV
    \ar[d]^{\varphi^{\pv V}}
    \\
    B^+
    \ar[r]^{\hat\iota_{\pv V,B}^+}
    &
    \Om BV
  }
\end{displaymath}
In case $\pv V$ contains \pv{N} or \pv{G}, these mappings are
injective and we say that $\prov V\varphi$ is the \emph{pro-$\pv V$
  extension} of $\varphi$. By the uniqueness of the pro-$\pv V$
extension of homomorphisms between free semigroups, the correspondence
$\varphi\mapsto \varphi_{\pv V}$ is functorial; in other words, $\prov
V{(\varphi\circ \psi)}=\prov V\varphi\circ\prov V\psi$ whenever
$\varphi$ and $\psi$ are composable homomorphisms of free semigroups,
and the pro-$\pv V$ extension of the identity on $A^+$ is the identity
on $\Om AV$.

Recall that a \emph{Stone space} is a topological space that is both
compact and totally disconnected. Note that closed subspaces of Stone
spaces are also Stone spaces. 
Furthermore, the categories of Stone spaces and Boolean algebras are connected by \emph{Stone duality}~\cite{Burris&Sankappanavar:1981}.
A theorem of Numakura states that a topological semigroup is profinite if and only if it is topologically a Stone space~\cite[Theorem~1]{Numakura:1957}.

For every pseudovariety of semigroups $\pv V$, a language $L\subseteq A^+$ is said to be \emph{\pv V-recognizable} if
there are a semigroup $S\in\pv V$ and a homomorphism $\varphi\from
A^+\to S$ such that $L=\varphi^{-1}(\varphi(L))$. The \emph{syntactic semigroup} of the language $L$ is the
quotient of $A^+$ by the least congruence saturating $L$. We have the following alternative characterization of
the notion of \pv V-recognizable language: $L$ is \pv V-recognizable if and only if the syntactic semigroup of $L$
belongs to $\pv V$.  The following basic
result gives a topological characterization of the same notion, which amounts to the fact that the topological space \Om AV is the Stone
dual of the Boolean algebra of all \pv V-recognizable subsets of~$A^+$.

\begin{theorem}[{\cite[Theorem~3.6.1]{Almeida:1994a}}]
  \label{t:V-recognizability}
  Let \pv V be a pseudovariety of semigroups containing \pv N. Then a
  language $L\subseteq A^+$ over a finite alphabet $A$ is \pv
  V-recognizable if and only if $\clos V(L)$ is open.
\end{theorem}

Endow the set $\nn_+$ of positive integers with the semigroup operation of addition.
For this structure, the length mapping
$\ell\from A^+\to\nn_+$,
defined by $\ell(u)=|u|$
is a semigroup homomorphism. We extend the length homomorphism to pseudowords
in the following way.
Let $\nn_+\cup\{\infty\}$ be the Alexandroff compactification of the
discrete space $\nn_+$,
and extend the addition operation on $\nn_+$ to $\nn_+\cup\{\infty\}$
by making $\infty$ an absorbing element of $\nn_+\cup\{\infty\}$.
In this way, $\nn_+\cup\{\infty\}$ is a pro-$\pv N$ semigroup (it is in fact a free pro-\pv{N} semigroup on the single generator~1).
Therefore, provided \pv V contains $\pv{N}$, the length homomorphism $\ell\from A^+\to \nn_+$ extends uniquely to a continuous homomorphism
$\prov V\ell\from\Om AV\to\nn_+\cup\{\infty\}$. We use the notation $|u|$ for $\prov V\ell(u)$, for every $u\in\Om AV$. An element $u\in\Om AV$ has \emph{infinite length} if $|u|=\infty$, and \emph{finite length} otherwise.  Clearly, infinite-length pseudowords form a closed ideal of~$\Om AV$, a fact included in the next proposition, whose complete proof can be found in~\cite[Section 3]{Almeida&ACosta&Goulet-Ouellet:2024a}, which extends earlier results for the case when $A$ is finite;
see e.g.~\cite{Almeida:1994a}.

\begin{proposition}
  \label{p:A+-as-filter}
  Let $A$ be an arbitrary alphabet and let \pv V be a pseudovariety of
  semigroups containing \pv N. Then the following hold:
  \begin{enumerate}
    \item the elements of~$A^+$ are isolated points in~$\Om AV$;
      \label{i:A+-as-filter-0}
    \item the set $\{u\in\Om AV:|u|=\infty\}$ is an ideal of the free
      pro-\pv{V} semigroup $\Om AV$;
      \label{i:A+-as-filter-1}
    \item the set $\Om AV\setminus A^+$ is an ideal of the free
      pro-\pv{V} semigroup $\Om AV$.
      \label{i:A+-as-filter-2}
  \end{enumerate}
\end{proposition}

\begin{remark}
  Since, whenever $\pv V\supseteq\pv N$, the elements of~$A^+$ are
  isolated points in~$\Om AV$, the equality $\clos V(L)\cap A^+=L$
  holds for every language $L\subseteq A^+$.
\end{remark}

\begin{remark}
  In case $A$ is finite, the equality $\Om AV\setminus A^+=\{u\in\Om
  AV:|u|=\infty\}$ holds whenever $\pv V\supseteq\pv N$, since there
  are only finitely many words of~$A^+$ of a given length. The
  equality no longer holds if~$A$ is an infinite alphabet: in that
  case, the topological closure of~$A$ in~$\Om AV$, being compact,
  contains some element not in~$A$, and any such element has length 1
  by continuity of the length homomorphism.
\end{remark}

Let us suppose that $\pv V$ contains the pseudovariety~$\pv{\loc I}$,
bearing in mind that $\pv{\loc I}$ contains~$\pv N$. In that case, for
every nonnegative integer $n$, every pseudoword $w\in(\Om AV)^1$ such
that $|w|\geq n$ has a unique prefix of length of $n$, denoted
$w{[0},n{)}$, and a unique suffix of length $n$, denoted
$w{[}-n,-1{]}$ (for further details, we refer to the discussion
in~\cite[Section 6]{Almeida&ACosta&Goulet-Ouellet:2024a}). We also
denote $w[0,k+1)[-1,-1]$ by $w[k]$, which means that $w[0,n) =
w[0]\cdots w[n-1]$ is the unique factorization of $w[0,n)$ into
pseudowords of length 1.

In the case of the pseudovariety $\pv S$, we have the following
property, which amounts to saying that in an equality of pseudowords
we may cancel equal finite-length prefixes, or suffixes.
  
\begin{proposition}
  \label{p:letter-super-cancelativity}
  Let $A$ be any alphabet. If $x,y,u,v\in(\Om AS)^1$ are pseudowords
  such that $xu=yv$ or $ux=vy$, and $|x|=|y|\in\nn$, then $x=y$ and
  $u=v$.
\end{proposition}

Let $w$ be an infinite-length pseudoword on $\Om AS$, and let $x$ be
its prefix of length~$n$, where $n\in\nn$. By
Proposition~\ref{p:letter-super-cancelativity} there is a unique
pseudoword $u\in\Om AS$ such that $w=xu$. We may denote $u$ by
$x^{-1}w$, and sometimes, alternatively, by $w^{(n)}$.

 \subsection{Codes}
\label{sec:codes}

Let $C$ be a subset of $A^+$. Recall that $C$ is a \emph{code} if the
subsemigroup of $A^+$ generated by $C$ is free with basis $C$. The
code $C$ is called \emph{pure} if it is closed under extraction of
roots, that is, if for every $u\in A^+$ and integer $n\geq 1$, the
following implication holds:
\begin{equation*}
  u^n\in C^+\implies u\in C^+.
\end{equation*}

It turns out that a finite code $C$ over the alphabet $A$ is pure if
and only if the syntactic semigroup of $C^+$ belongs to the
pseudovariety $\pv{A}$ of all finite aperiodic semigroups
(cf.~\cite[Theorem 3.1]{Restivo:1973}; see also~\cite[Chapter 7,
Exercise 8]{Lallement:1979}). Hence, pure codes are often called
\emph{aperiodic} codes.

The property of being closed under root extraction carries through for
pro-\pv{V} closures of pure codes, in the following sense.

\begin{lemma}
  \label{l:profinite-root-extraction}
  Let $C$ be a finite pure code over a finite alphabet $A$, $\pv{V}$ a
  pseudovariety containing $\pv{A}$, and $u$ an element of $\Om AV$.
  If $\clos V(C^+)\cap \clos V(u^+)\neq\varnothing$, then $u$ belongs
  to $\clos V(C^+)$.
\end{lemma}

\begin{proof}
  Take $x\in \clos V(C^+)\cap \clos V(u^+)$. Let $(u_i)_{i\in\nn}$ be
  a sequence of finite words such that $u = \lim u_i$ and
  $(u^{n_i})_{i\in\nn}$ be a sequence of powers of $u$ converging to
  $x$ such that $n_i$ is a positive integer for every $i\in\nn$. We
  claim that $(u_i^{n_i})_{i\in\nn}$ converges to $x$. Let
  $\varphi\from \Om AV\to S$ be an arbitrary continuous homomorphism
  where $S\in\pv{V}$; because $\Om AV$ is residually $\pv{V}$, we are
  reduced to showing that there exists $j\in\nn$ such that
  $\varphi(u_i^{n_i})=\varphi(x)$ for every $i\geq j$. Since $S$ is
  discrete, there exists $i_1\in \nn$ such that $\varphi(u_i) =
  \varphi(u)$ for every $i\geq i_1$. Likewise, there exists $i_2\in
  \nn$ such that $\varphi(u^{n_i}) = \varphi(x)$ for all $i\geq i_2$.
  Then, for all $i\geq \max\{i_1, i_2\}$, we find that
  \begin{equation*}
    \varphi(u_i^{n_i}) = \varphi(u_i)^{n_i}
    = \varphi(u)^{n_i} = \varphi(u^{n_i}) = \varphi(x).
  \end{equation*}
  This concludes the proof of the claim.
  
  Since $\pv{A}\subseteq\pv{V}$ and the syntactic semigroup of $C^+$
  is in $\pv{A}$, it follows that $\clos V(C^+)$ is clopen. Therefore,
  there exists $j\in\nn$ such that $u_i^{n_i}\in \clos V(C^+)$ for all
  $i\geq j$.  As $\clos V(C^+)\cap A^+=C^+$, this means that
  $u_i^{n_i}\in C^+$ for all $i\geq j$.
  By purity,  it follows that $u_i\in C^+$ for all $i\geq j$,
  whence $u = \lim u_i\in\clos V(C^+)$.
\end{proof}

Using this lemma, we deduce the following key property of pure codes.

\begin{proposition}
  \label{p:surjective-image-pure-code}
  Let $C$ be a finite pure code over a finite alphabet $A$ and $\pv V$
  be a pseudovariety of finite semigroups containing $\pv A$. For
  every subgroup $H\subseteq\Om AV$, the following implication holds:
  \begin{equation*}
    H\cap\clos{V}(C^+)\neq\varnothing \implies H\subseteq \clos{V}(C^+).
  \end{equation*}
\end{proposition}

For the proof of this proposition, we use the following: in a
profinite finite semigroup~$S$, given~$s\in S$, the sequence
$(s^{n!})_{n\in\nn}$ converges to an idempotent, denoted~$s^\omega$,
which is the unique idempotent in the closed subsemigroup of $S$
generated by~$s$ \cite[cf.~Proposition
3.9.2]{Almeida&ACosta&Kyriakoglou&Perrin:2020b}. At some point in the
paper we also use the notation $s^{\omega-1}$ for the inverse of
$s^{\omega+1}=s^\omega\cdot s$ in the maximal subgroup of $S$
containing the idempotent $s^\omega$.

\begin{proof}[Proof of Proposition~\ref{p:surjective-image-pure-code}]
  Let $u\in H$. Take $h\in H\cap\clos V(C^+)$. Observe that
  $u^\omega=h^\omega$, as $H$ is a subgroup. Since $\clos V(C^+)$ is a
  closed semigroup, it follows that $u^\omega\in \clos V(C^+)\cap\clos
  V(u^+)$. This yields $u\in\clos V(C^+)$ by
  Lemma~\ref{l:profinite-root-extraction}.
\end{proof}

Let $\pv V$ be a pseudovariety of finite semigroups and $C\subseteq
A^+$ be a code. If the syntactic semigroup of $C^+$ is in~$\pv V$,
then we say that $C$ is a \emph{$\pv V$-code}. In particular the
finite $\pv{A}$-codes are exactly the finite pure codes.

An injective homomorphism $\sigma\from A^+\to B^+$ is called an
\emph{encoding}; equivalently, $\sigma$ is injective on $A$ and
$\sigma(A)$ is a code. The encoding $\sigma$ is \emph{pure} if
$\sigma(A)$ is a pure code. Note that a circular homomorphism is a
pure encoding, but the converse fails~\cite[Example
7.1.4]{Berstel&Perrin&Reutenauer:2010} for the Prouhet-Thue-Morse
substitution
\begin{equation*}
  \tau\from \ltr{a}\mapsto \ltr{ab},\ \ltr{b}\mapsto \ltr{ba}.
\end{equation*}

We say that an encoding $\sigma\from A^+\to B^+$
is a \emph{$\pv V$-encoding} if $\sigma(A)$ is a $\pv V$-code. 

The following theorem may be attributed to Margolis, Sapir and
Weil~\cite{Margolis&Sapir&Weil:1995}. Since they did not explicitly
state the theorem in this form, we give a short proof for the sake of
completeness.
 
\begin{theorem}\label{t:sufficient-conditions-for-injectivity}
  Let $\sigma\from A^+\to B^+$ be a homomorphism and
  $\pv H,\pv K$ be pseudovarieties of groups.
  Suppose that:
  \begin{enumerate}
  \item $\pv H*\pv K\subseteq \pv
    H$;\label{item:sufficient-conditions-for-injectivity-1}
  \item $\sigma$ is a $\pvo
    K$-encoding.\label{item:sufficient-conditions-for-injectivity-2}
  \end{enumerate}
  Then the pro-$\pvo H$ extension
   $\prov {\pvo H}\sigma\from \Om A{\pvo H}\to \Om B{\pvo H}$
  is injective.
\end{theorem}

\begin{proof}
  We refer the reader to \cite{Margolis&Sapir&Weil:1995} for the
  definitions of \emph{sagittal semigroup} and of \emph{unambiguous
    product of semigroups}, which are used in this proof. By
  \cite[Corollary~1 of Theorem~4.9]{Weil:1986} (which is based on
  \cite[Theorem 3]{LeRest&LeRest:1980a}), we know that the sagittal
  semigroup of $\sigma(A)$ is in $\pvo K$, and by \cite[Proposition
  2.1]{Margolis&Sapir&Weil:1995}, if $\sigma$ is injective then
  the extension $\prov {\pvo H}\sigma$ is injective provided the unambiguous
  product of every semigroup of $\pvo H$ with the sagittal semigroup
  of $\sigma(A)$ is still in~$\pvo H$. By \cite[Lemma
  1.3]{Margolis&Sapir&Weil:1995}, such unambiguous product indeed
  belongs to $\pvo H$ whenever $\pv H*\pv K\subseteq\pv H$.
\end{proof}

When $\pv K$ is the trivial pseudovariety, then
condition~\ref{item:sufficient-conditions-for-injectivity-1} in the
statement of the theorem holds trivially,
while~\ref{item:sufficient-conditions-for-injectivity-2} means that
$\sigma$ is a pure encoding. Thus, if $\sigma$ is pure, then $\prov
{\pvo H}\sigma$ is injective for every pseudovariety of groups \pv H.

In Theorem~\ref{t:sufficient-conditions-for-injectivity}, condition
\ref{item:sufficient-conditions-for-injectivity-1} cannot be omitted.
We illustrate this with the following example.

\begin{example}
  \label{e:induced-barH-homomorphism-not-injective-for-H-encoding}
  Let $\pv{H}$ be a nontrivial locally finite pseudovariety of groups
  (\emph{locally finite} means that all finitely generated pro-\pv{V}
  semigroups are finite). Let $A = \{\ltr{a}\}$ be a one-letter
  alphabet and $n$ be the order of $\Om AH$. The free pro-\pvo{H}
  semigroup $\Om A{\pvo{H}}$ consists of all powers $\ltr{a}^k$ with
  $k$ a positive integer, which are all distinct powers, together with
  a group of order $n$. Consider the homomorphism $\sigma\from A^+\to
  A^+$ defined by $\sigma(\ltr{a}) = \ltr{a}^n$. The syntactic
  semigroup of $\sigma(A^+)$ is $\zz/n\zz$, so $\sigma$ is an
  $\pv{H}$-encoding, hence also an $\pvo{H}$-encoding. The unique
  idempotent $e$ of $\Om A{\pvo{H}}$ satisfies $\prov {\pvo
    H}\sigma(e) = \prov {\pvo H}\sigma(e\ltr{a})$, whereas $e\neq
  e\ltr{a}$; hence, $\prov {\pvo H}\sigma\from\Om A{\pvo{H}}\to\Om
  A{\pvo{H}}$ is not injective. Note however that $\pv{H}*\pv{H}$ is
  not contained in $\pv{H}$.
\end{example}

\subsection{Finitely generated profinite semigroups}
\label{sec:finitely-generated}

A profinite semigroup $S$ is said to be \emph{$n$-generated} if it has
a subset of cardinality at most $n$ which generates a dense
subsemigroup. We also say that $S$ is \emph{finitely generated} if it
is $n$-generated for some $n\in\nn$. This subsection collects a number
of useful facts on finitely generated profinite semigroups

Given two profinite semigroups $S$ and $R$, let $\hom(S,R)$ be the set
of continuous semigroup homomorphisms $S\to R$. The monoid of
continuous endomorphisms of a profinite semigroup $S$ is denoted
$\en(S)$.

Hunter proved that, when $S$ is finitely generated, the monoid
$\en(S)$ is a profinite semigroup for the pointwise
topology~\cite[Proposition~1]{Hunter:1983}. This was rediscovered by
the first author, who moreover observed the equality between the
pointwise and compact-open
topologies~\cite[Proposition~4.13]{Almeida:2003cshort}. The next
result is a generalization of this to general hom-sets $\hom(S,R)$.

\begin{theorem}\label{t:hunter}
  Let $S$ and $R$ be finitely generated profinite semigroups. Then,
  the compact-open topology on $\hom(S,R)$ agrees with the
  pointwise topology. Under this topology, $\hom(S,R)$ is a Stone space.
\end{theorem}

Before the proof, we need to set up some notation. Let $X$ and $Y$ be
two topological spaces and let $F$ be a set of functions from $X$ to
$Y$. Given a compact subset $K$ of~$X$ and an open subset $O$ of $Y$,
we denote $[K,O]_F$ the set of all $f\in F$ such that $f(K)\subseteq
O$; these sets are a subbasis for the \emph{compact-open} topology on
$F$. When $K$ runs only over singleton subsets of $X$, we obtain a
subbasis for the pointwise topology instead.

\begin{proof}[Proof of Theorem~\ref{t:hunter}]
  Let $U$ be the topological coproduct of $S$, $R$, and $\{0\}$ and
  extend the multiplications of $S$ and $R$ by declaring all other
  products in $U$ to be~0. In this way, $U$ becomes a finitely
  generated profinite semigroup. We extend each element $\varphi$ of
  $\hom(S,R)$ to a continuous endomorphism $\xi(\varphi)$ of~$U$ by
  mapping $R\cup\{0\}$ to~$0$. Note that $\xi$ is an injective mapping.

  For a compact subset $K\subseteq S$ and an open subset $O\subseteq
  R$, we have
  \begin{equation*}
    \xi([K,O]_{\hom(S,R)})=[K,O]_{\en(U)}\cap\img(\xi).
  \end{equation*}
  This shows that $\xi$ is a topological embedding, both when the
  compact-open and pointwise topologies are considered in both the
  domain and range of~$\xi$. Since the two topologies coincide on
  $\en(U)$ as observed above, it follows that the two topologies also
  coincide on~$\hom(S,R)$. Since the image of $\xi$ has complement the open
  union $\bigcup_{s\in S}[\{s\},S\cup\{0\}]$ and $\en(U)$ is a Stone by \cite[Proposition~1]{Hunter:1983}, it follows that the image of $\xi$ is also a Stone space. 
  Since $\xi$ is a topological embedding, we conclude that $\hom(S,R)$ is a Stone space.
\end{proof}

For each pair of profinite semigroups $S$ and $R$ the \emph{evaluation mapping} is the mapping
$\hom(S,R)\times S\to R$ sending each pair $(\varphi,s)$ to $\varphi(s)$. 
We need the following corollary of Theorem~\ref{t:hunter} for several of our proofs.

\begin{corollary}
  \label{c:evaluation-mapping-is-continuous}
  Let $S$ and $R$ be finitely generated profinite semigroups.
  Consider in $\hom(S,R)$ the pointwise topology.
  Then the evaluation mapping $\hom(S,R)\times S\to R$ is continuous.
\end{corollary}

\begin{proof}
  The evaluation mapping is continuous  under the compact-open
  topology of $\hom(S,R)$~\cite[Corollary~X.3.1]{Bourbaki:1998GTb}, which agrees
  with the pointwise topology by Theorem~\ref{t:hunter}.
\end{proof}

\section{Profinite categories}
\label{sec:profinite-cats}

In this paper, a \emph{graph} is a structure $\Gamma$ which consists
of two disjoint sets $\vertex(\Gamma)$ and $\edg(\Gamma)$, called the
\emph{vertex set} and the \emph{edge set}, together with two
\emph{adjacency mappings} $\init_\Gamma, \term_\Gamma\from
\vertex(\Gamma)\to \edg(\Gamma)$, called the \emph{domain mapping}
and \emph{range mapping} respectively. The set of \emph{composable edges}, also called \emph{consecutive edges}, is the subset of $\edg(\Gamma)\times \edg(\Gamma)$ given by
\begin{equation*}
  D(\Gamma) = \{ (u,v)\in \edg(\Gamma)\times \edg(\Gamma) : \init_\Gamma(u)=\term_\Gamma(v)\}.
\end{equation*}
When the graph $\Gamma$ is clear from the context, we may simply write $V$,
$E$, $D$, $\init$ and $\term$. We say that $u$ is an edge \emph{from $\alpha(u)$ to $\omega(u)$}.
An edge $u$ from a vertex $q$ to the same vertex $q$ is called a \emph{loop} at $q$.

Every small category $C$ is a graph: the set $V$
is the set of objects of $C$, the set $E$ is the set of morphisms of $C$,
and the adjacency mappings $\init,\term\from E\to V$ send a morphism to its domain and codomain respectively.
The loops of $C$ are the endomorphisms of objects of $C$.

The \emph{consolidation} of a small category $C$ is the semigroup
  $C_{cd}=E(C)\uplus\{0\}$ such that $0$ is an element not in $E(C)$,
  which is a zero of $C_{cd}$, with the semigroup operation on $C_{cd}$  being
  the following natural extension of the composition on $C$:
  \begin{equation*}
    fg=
    \begin{cases}
      f\circ g&\text{ if $(f,g)$ is a pair of composable edges of $C$},\\
      0&\text{ otherwise}.
    \end{cases}
  \end{equation*}
Green's relations on the consolidation of
$C$ restrict on $E(C)$ to relations which are called \emph{Green's relations} of~$C$.

A \emph{graph homomorphism} $\Gamma\to\Delta$ is a mapping
$\varphi\from\vertex(\Gamma)\cup
\edg(\Gamma)\to\vertex(\Delta)\cup\edg(\Delta)$ which maps vertices
to vertices, edges to edges, and satisfies the following equations for
every $u\in\edg(\Gamma)$:
\begin{equation*}
  \init_\Delta(\varphi(u)) = \varphi(\init_\Gamma(u)),\quad
  \term_\Delta(\varphi(u)) = \varphi(\term_\Gamma(u)).
\end{equation*}
Note that, under these conditions, $(\varphi(u),\varphi(v))\in
D(\Delta)$ for all $(u,v)\in D(\Gamma)$. When $\Gamma$ and $\Delta$
are small categories, we may say that $\varphi$ is a \emph{category
  homomorphism} when $\varphi$ is a functor, which means that $\varphi$
is a graph homomorphism such that $\varphi(uv) = \varphi(u)\varphi(v)$
whenever $(u,v)\in D(\Gamma)$ and $\varphi(1_q)=1_{\varphi(q)}$ whenever $q\in V(\Gamma)$, where $1_p$ stands for the local identity on object $p$.

For a graph $\Gamma$, a \emph{path} is a word $u\in \edg(\Gamma)^+$
such that $(u[i],u[i+1])\in D_\Gamma$ for all $0\leq i<|u|-1$ (see Figure~\ref{fig:path-u}). To each vertex $q\in V(\Gamma)$ we associate
an empty path~$1_q$. In this way, we form the free category over $\Gamma$, denoted $\Gamma^*$, given by the following data:
one has $V(\Gamma^*)=V(\Gamma)$,
the set $E(\Gamma^*)$ is the set of all paths (including empty paths) in $\Gamma$,
the relations $\alpha(u)=\alpha(u[n-1])$ and $\omega(u)=\omega(u[0])$ hold for every nonempty path $u$ of length $n$,
the empty path $1_q$ is the local identity at $q$ for every $q\in V(\Gamma)$, and the composition
of two consecutive paths $u,v$ is the path $uv$.

\begin{figure}[ht]
  \centering
\begin{equation*}
  \xymatrix@C=1.1 cm{
    q_0&q_1\ar[l]_{u[0]}&q_2\ar[l]_{u[1]}&\cdots\ar[l]&q_{n-1}\ar[l]&q_n\ar[l]_(.4){u[n-1]}
  }
\end{equation*}  
  \caption{Path $u$ of length $n$, seen as a composition of edges $u[i]$ from $q_{i+1}$ to $q_{i}$, for $0\leq i<n$}
  \label{fig:path-u}
\end{figure}

A \emph{topological graph} is a graph $\Gamma$ with topologies on $V$
and $E$ such that the incidence mappings $\init$, $\term$ are
continuous; a \emph{topological category} is a small category, as an
algebraic structure, with topologies on $V$ and $E$ making continuous the
incidence mappings and the category operations (i.e., the composition mapping and the mapping $q\in V(C)\mapsto 1_q\in E(C)$). We say that
a topological graph or category is \emph{compact} when both $V$ and
$E$ are compact spaces (recall that we include the Hausdorff property
in the definition of compactness).

A graph or category is called \emph{finite-vertex} when $V$ is finite,
and \emph{finite} when both $V$ and $E$ are finite. A \emph{profinite
  graph} is an inverse limit of finite discrete graphs; a \emph{profinite
category} is likewise an inverse limit of finite categories. When
equipped with continuous category homomorphisms (i.e., both vertex and
edge components of the homomorphism are continuous), profinite
categories form a category in the usual sense of category theory.

Let $\catpro$ denote the category of profinite semigroups, where
morphisms are continuous semigroup homomorphisms.
For each set $\mathcal F$ of profinite semigroups, denote by
$\catpro[\mathcal{F}]$ the full subcategory of
$\catpro$ whose objects are the elements of $\mathcal F$.
Let us say that the category $\catpro[\mathcal{F}]$ \emph{is equipped with the pointwise topology} if
it is equipped with the following topological structure:
\begin{enumerate}
\item $\mathcal{F}$ has the discrete topology;
\item each hom-set $\hom(S,R)$, with $S,R\in \mathcal{F}$, is endowed with the pointwise topology (which agrees with the compact-open topology when $S$ and $R$ are finitely generated, by Theorem~\ref{t:hunter});
\item the set of morphisms of $\catpro[\mathcal{F}]$ is endowed with the coproduct
  topology of the spaces of the form $\hom(S,R)$, with $S,R\in \mathcal{F}$.
\end{enumerate}

\begin{proposition}\label{p:hom-is-a-profinite-category}
  Let $\mathcal{F}$ be a set of finitely generated profinite semigroups.
  When equipped with the pointwise topology, the category $\catpro[\mathcal{F}]$ is a topological category.
  If moreover $\mathcal{F}$ is finite, then $\catpro[\mathcal{F}]$ is a profinite category.
\end{proposition}

\begin{proof}
  Since the topology of the set of objects of $\catpro[\mathcal{F}]$ is discrete, the mapping
  $S\mapsto 1_S$, with domain $\mathcal{F}$, is clearly continuous. 
  For every $S,R\in \mathcal{F}$, the set $\hom(S,R)$ has the compact-open topology, by Theorem~\ref{t:hunter}.
  The composition is continuous for compact-open topologies (see~\cite[Proposition~X.3.9]{Bourbaki:1998GTb}). Therefore, $\catpro[\mathcal F]$ is indeed a topological category.

  A coproduct of finitely many Stone spaces is itself a Stone space.
  Hence, when $\mathcal{F}$ is finite, the morphisms of $\catpro[\mathcal F]$ form a Stone space by Theorem~\ref{t:hunter}.
  Peter Jones showed that a finite-vertex topological category whose space of morphisms is a Stone space\footnote{Peter Jones used the term \emph{Boolean category} to refer to a topological category whose underlying topological space is a Stone space.}
  is in fact a profinite category (cf.~\cite[Theorem 4.1]{Jones:1996}).
  Therefore, if $\mathcal{F}$ is finite, then $\catpro[\mathcal{F}]$  is profinite.
\end{proof}

Let $\pv{Cat}$ be the class of all finite categories. Given a finite-vertex graph
$\Gamma$, we let $\Om \Gamma{Cat}$ denote the free profinite category
on $\Gamma$; a description of $\Om\Gamma{Cat}$ may be found in a paper
by the first two authors~\cite[Section~3.4]{Almeida&ACosta:2007a}; it can be found in earlier papers such as~\cite{Jones:1996,Almeida&Weil:1996}. The free profinite category $\Om \Gamma{Cat}$ is equipped
with an inclusion mapping $\iota\from\Gamma\to\Om \Gamma{Cat}$ with
the following universal property: for every profinite category
$\Sigma$ and graph homomorphism $\varphi\from\Gamma\to \Sigma$, there
exists a unique continuous category homomorphism
$\pro\varphi\from\Om\Gamma{Cat}\to \Sigma$ such that
$\pro\varphi\circ\iota = \varphi$.

The canonical mapping $\iota$ extends to an inclusion mapping 
$\iota^*\from\Gamma^*\to\Om\Gamma{Cat}$ which identifies $\Gamma^*$
with a dense discrete subcategory of $\Om\Gamma{Cat}$; this is similar
to how $A^+$ is identified to a dense discrete subset of $\Om AV$ when
$\pv V$ is a pseudovariety of semigroups containing $\pv{N}$. The edges of $\Om\Gamma{Cat}$ are called
\emph{pseudopaths}.

Every (profinite) monoid is viewed as a (profinite) category with only one object. Therefore, given a finite-vertex graph $\Gamma$, we
may consider the continuous category homomorphism
\begin{equation*}
\chi_\Gamma\from\Om\Gamma{Cat}\to(\Om{\edg(\Gamma)}S)^1
\end{equation*}
which collapses all vertices and maps each $b\in\edg(\Gamma)$ to the
corresponding generator of $\Om{\edg(\Gamma)}S$.
The next proposition goes back to \cite{Almeida:1996c} and is discussed in full generality
in~\cite[Proposition 3.19(ii) and Remark 6.13]{Almeida&ACosta&Goulet-Ouellet:2024a}.

\begin{proposition}
  \label{p:list-of-properties-of-chi}
  Let $\Gamma$ be a finite-vertex graph. The
  following properties hold:
  \begin{enumerate}
  \item one has $\chi(w)=1$ if and only if $w$ is a local identity;\label{item:list-of-properties-of-chi-1-and-half}
  \item the restriction of $\chi_\Gamma$ to the edges of $\Om\Gamma{Cat}$ that are not local identities is injective.\label{item:list-of-properties-of-chi-2}
  \end{enumerate}
\end{proposition}

Following Proposition~\ref{p:list-of-properties-of-chi}, we may view the pseudopaths of $\Om{\Gamma}{Cat}$ that are not local identities
as elements of $\Om{\edg(\Gamma)}S$: that is, if $w$ is a pseudopath of $\Om{\Gamma}{Cat}$ that is not a local identity, we may identify
$w$ with $\chi(w)$. We already apply this convention in the following statement (for a proof, see~\cite[Proposition 3.19(iii)]{Almeida&ACosta&Goulet-Ouellet:2024a}).

\begin{proposition}
  \label{p:factoriality-of-pseudopaths}
  Let $\Gamma$ be a finite-vertex graph.
  If $w$ is a pseudopath of $\Om{\Gamma}{Cat}$ and $u,v\in\Om {\edg(\Gamma)}S$ are pseudowords
  such that the equality $w=uv$ holds in $\Om{\edg(\Gamma)}S$, then $u$ and $v$ are consecutive pseudopaths of $\Om{\Gamma}{Cat}$, and the equality $w=uv$ also holds in $\Om{\Gamma}{Cat}$.
\end{proposition}

\begin{remark}
  \label{r:extension-pseudowords-to-pseudopaths}
  Propositions~\ref{p:list-of-properties-of-chi} and \ref{p:factoriality-of-pseudopaths} allow us to immediately extend to pseudopaths several
definitions, properties and notation that we already introduced for pseudowords.
For example, we consider the length of a pseudopath $w$ as just being the length of $w$ seen as a pseudoword;
we know that every pseudopath $w$ of infinite length has a unique finite prefix in $\Om{\Gamma}{Cat}$ of length $n$, denoted $w{[}0,n{)}$;
we may consider the pseudopath $w^{(n)}=(w{[}0,n{)})^{-1}w$; etc.
\end{remark}

From hereon, we apply liberally the generalizations from pseudowords to pseudopaths mentioned in Remark~\ref{r:extension-pseudowords-to-pseudopaths}.

\begin{definition}[Prefix accessible pseudopaths]
Let $\Gamma$ be an arbitrary graph.
A \emph{right-infinite} path of $\Gamma$ is an element $w$ of $\edg(\Gamma)^\nn$
such that $w{[}0,n{)}$ is a path of $\Gamma$, for every $n\in\nn$.
Let $w$ be a right-infinite path of a finite-vertex graph $\Gamma$.
A cluster point in $\Om\Gamma{Cat}$ of the sequence $(w[0,n))_n$
is said to be \emph{prefix accessible} by $w$.
\end{definition}

Let $A$ be an arbitrary alphabet. A right-infinite word $w\in A^\nn$
is said to be \emph{recurrent} if for every $n\in\nn$, there is some
$m>n$ such that $w[0,n)=w[m,m+n)$. Equivalently, $w\in A^\nn$ is
recurrent when every finite factor of $w$ occurs infinitely often in
$w$. For a proof of the following proposition, see~\cite[Corollary
6.14]{Almeida&ACosta&Goulet-Ouellet:2024a}.
  
\begin{proposition}
  \label{p:idempotents-in-Pw-Cat} Let $w$ be a right-infinite path
over a finite-vertex graph $\Gamma$. Then $w$ is recurrent if and only
if there is an idempotent pseudopath in $\Om\Gamma{Cat}$  that is prefix accessible by $w$.
\end{proposition}

Let $C$ be a small category. For an edge $x$ of $C$,
the \emph{right stabilizer} in $C$ is the set
\begin{equation*}
  \rstab_C(x)=\{y\in \edg: xy=x\},
\end{equation*}
which we denote simply by $\rstab(x)$ when $C$ is clear from context.
Note that $\rstab(x)$ is a submonoid of
the monoid of loops at $\init(x)$.
Moreover, $\rstab(x)$ is a profinite semigroup when $C$ is a profinite category. Since we view monoids
as one-vertex categories, the
definition of right stabilizer applies to a semigroup $S$ as well, by
considering the monoid $S^1$.

When $S$ is a profinite semigroup, there is a unique $\green{J}$-class
which has all elements of $S$ as factors. This $\green{J}$-class is frequently called the \emph{kernel of $S$}.

\begin{theorem}\label{t:right-stabilizers-are-R-trivial-bands} Let
$\Gamma$ be an arbitrary finite-vertex graph. For every pseudopath
$x\in\Om\Gamma{Cat}$, the kernel of\/ $\rstab(x)$ is a left-zero semigroup.
\end{theorem}

A proof of Theorem~\ref{t:right-stabilizers-are-R-trivial-bands} may be found in~\cite[Corollary 7.7]{Almeida&ACosta&Goulet-Ouellet:2024a},
and in the same paper we find a discussion about other similar results
going back to work of Rhodes and Steinberg~\cite{Rhodes&Steinberg:2001}.

The following characterization of the right stabilizer of an edge of $\Om\Gamma{Cat}$, extracted from~\cite[Corollary 7.8]{Almeida&ACosta&Goulet-Ouellet:2024a},
is used in the proof of Proposition~\ref{p:isomorphism-with-inverse-limit}. We adopt throughout the paper the topological definition of net and subnet
given by Willard~\cite[Definition 11.2]{Willard:1970}.

\begin{theorem}\label{t:characterization-of-L-minimal-right-stabilizers}
  Let $x$ be a prefix accessible pseudopath of $\Om\Gamma{Cat}$,
  with $\Gamma$ being a finite-vertex graph.
  Then an edge $y$ of $\Om\Gamma{Cat}$ belongs to the kernel of\/ $\rstab(x)$
  if and only if there is a net $(x_i)_{i\in I}$ of finite-length prefixes of $x$ such that
  $x_i\to x$ and $x_i^{-1}x\to y$.
\end{theorem}

\section{Free profinite semigroups and symbolic dynamics}
\label{s:profinite-symbolic}

The first author established a natural bijection associating to each
minimal shift space $X$ of $A^\zz$ a regular $\green{J}$-class of $\Om
AS$~\cite{Almeida:2005c}, whenever $A$ is a finite alphabet. This
subsection aims to provide sufficient background on this mapping,
which is at the core of the present paper. When checking the
literature, the reader may notice that the bijection is frequently
established in the larger realm of irreducible shift spaces
(cf.~\cite{Almeida:2003cshort}), but here we only deal with the case
of minimal shift spaces.

For a proof of the next proposition,
see~\cite[Lemma~2.3]{Almeida:2005c} or
\cite[Section~3]{ACosta&Steinberg:2011}. In the first of these
sources, only the pseudovariety $\pv S$ is explicitly mentioned, but
the arguments extend to all pseudovarieties containing $\pv N$.

\begin{proposition}
  \label{p:J-class-of-minimal-shift-space} For every pseudovariety of
semigroups \pv V containing \pv N and every minimal shift space
$X\subseteq A^\zz$, the set $\clos V(L(X))\setminus L(X)$ is contained
in a regular \green{J}-class of \Om AV.
\end{proposition}

For each minimal shift space $X\subseteq A^\zz$ and pseudovariety \pv
V containing \pv N, we denote by $J_\pv{V}(X)$ the $\green{J}$-class
of $\Om AV$ containing $\clos V(L(X))\setminus L(X)$. Since
$J_\pv{V}(X)$ is a regular $\green{J}$-class, it contains maximal
subgroups of $\Om AV$, and all these maximal subgroups are isomorphic
profinite groups; we denote by $G_{\pv V}(X)$ a profinite group
representing their isomorphism class. We say that $G_{\pv V}(X)$ is
the \emph{\pv V-Sch\"utzenberger group of $X$}.  

The \pv S-Sch\"utzenberger group of $X$ is a topological conjugacy invariant~\cite{ACosta:2006}.
In fact, it is a flow invariant (flow equivalence is an important relation between shift spaces that is strictly coarser
than topological conjugacy~\cite[Section 13.6]{Lind&Marcus:1996}) with
the same holding for many other pseudovarieties, as seen in the next
theorem, which is a special case of~\cite[Corollary~6.14]{ACosta&Steinberg:2021}.

\begin{theorem}
  \label{t:invariance-under-flow-equivalence}
  Let $\pv H$ be a pseudovariety of groups. If $X$ and $Y$ are flow equivalent minimal
shift spaces, then $G_{\pvo H}(X)$ and $G_{\pvo H}(Y)$ are isomorphic
profinite groups.
\end{theorem}

Recall that if $\bsigma$ is a primitive directive sequence, then
$X(\bsigma)$ is a minimal shift space
(Theorem~\ref{t:characterization-minimal-shifts}). We denote the
$\mathcal{J}$-class $J_{\pv V}(X(\bsigma))$ and the profinite group
$G_{\pv V}(X(\bsigma))$ respectively by $J_{\pv V}(\bsigma)$ and
$G_{\pv V}(\bsigma)$. In case $\varphi$ is a primitive substitution,
we also write $J_{\pv V}(\varphi)$ and $G_{\pv V}(\varphi)$ instead
of, respectively, $J_{\pv V}(X(\varphi))$ and $G_{\pv V}(X(\varphi))$.

For a pseudoword $w\in\Om AV$, we
denote by $\fac(w)$ the set of all words $u\in A^+$ such that $u$ is a
factor of $w$, assuming $\pv N\subseteq \pv V$ so that $A^+$
embeds in $\Om AV$.

\begin{proposition}\label{p:mirage} Let~$X$ be a minimal shift space
of $A^\zz$ and\/ $\pv V$ be a pseudovariety of semigroups containing
$\pv {\loc Sl}$.
Every infinite-length factor of an element of $J_{\pv V}(X)$ also
belongs to $J_{\pv V}(X)$. More precisely, for every infinite-length pseudoword $w\in \Om AV$, we have
$w\in J_{\pv V}(X)$ if and only if $\fac(w)\subseteq L(X)$.
\end{proposition}

This proposition is from~\cite[Lemma~2.3]{Almeida:2005c};
alternatively, it is found in~\cite[Theorem 6.3]{Almeida&ACosta:2007a} with  a very different proof.
In the first of these two references only the case $\pv V=\pv S$ is explicitly mentioned, but the arguments hold whenever $\pv V\supseteq\pv {\loc Sl}$.

\begin{remark}
  \label{r:why-we-need-LSL} The hypothesis in
Proposition~\ref{p:mirage} that $\pv V$ contains $\pv {\loc Sl}$ is
necessary to guarantee that $\clos V(A^*uA^*)$ is open when $u\in
A^+$, a property crucially used in the proof. The reason why $\clos
V(A^*uA^*)$ is then open is that $A^*uA^*$ is an \pv{\loc
Sl}-recognizable language \cite[Theorem~5.2.1]{Pin:1986;bk}, which
entails the desired topological property by
Theorem~\ref{t:V-recognizability}. Proposition~\ref{p:mirage} fails
for example when
$\pv V=\pv {\loc I}$; indeed, if $X\subseteq A^\zz$ is
any minimal shift space, then $e=eue$ for every idempotent $e\in\Om
A{\loc I}$ and word $u\in A^+$, entailing $\fac(e)= A^+$.
\end{remark}

\begin{corollary}
  \label{c:mirage}
  Let $X$ be a minimal shift space of $A^\zz$. If\/ $\pv V$ and $\pv W$ are
  pseudovarieties of semigroups such that $\pv{\loc Sl}\subseteq \pv
  W\subseteq\pv V$, then $p_{\pv V,\pv W}^{-1}(J_{\pv W}(X))=J_{\pv
    V}(X)$.
\end{corollary}
 
\begin{proof}
  Take $u\in J_{\pv W}(X)$. Since $p_{\pv V,\pv W}$ is onto, we may
  consider $\hat u\in \Om AV$ such that $u=p_{\pv V,\pv W}(\hat u)$.
  Because $\pv V$ and $\pv W$ contain $\pv {\loc Sl}$, we know that
  $\fac(u)=\fac(\hat u)$ (see Remark~\ref{r:why-we-need-LSL}). Since both $u$ and $\hat u$ are infinite
  pseudowords, it follows from Proposition~\ref{p:mirage} that $\hat
  u\in J_{\pv V}(X)$.
\end{proof}

The property stated in the next proposition is new in its full generality. The special case of the so called pseudovarieties \emph{closed under concatenation}
is treated in the last section of the
paper~\cite{Almeida&ACosta:2007a}. We point out that the proof of the proposition uses the property, first shown independently in the
papers~\cite{Almeida&ACosta:2007a,Henckell&Rhodes&Steinberg:2010b}, that
$\clos{S}(L)$ is factorial in $\Om AS$ whenever $L$ is a factorial subset of $A^+$.

\begin{proposition}
  \label{p:infinite-closure-J-class} Let $X$ be a minimal shift space
of $A^\zz$. If the pseudovariety of semigroups \pv V contains
$\pv{\loc Sl}$, then the equality
  \begin{equation*} J_{\pv V}(X)=\clos V(L(X))\setminus A^+
  \end{equation*} holds.
\end{proposition}
  
  \begin{proof} Recall that $\clos{V}(L(X))\setminus A^+\subseteq
J_{\pv V}(X)$ by definition of $J_{\pv V}(X)$. Conversely, let $u\in
J_{\pv V}(X)$. By Corollary~\ref{c:mirage}, there is $\hat u\in J_{\pv
  S}(X)$ such that $u=p_{\pv S,\pv V}(\hat u)$. Since
$L(X)$ is a factorial subset of $A^+$,  the topological closure of $\clos{S}(L(X))$ is
factorial in \Om AS, by~\cite[Proposition 2.4]{Almeida&ACosta:2007a}.
Because $J_{\pv S}(X)$ intersects $\clos{S}(L(X))$, it follows that
$\hat u\in \clos{S}(L(X))$. As $p_{\pv S,\pv V}$ restricts to the
identity on $A^+$, and by continuity of $p_{\pv S,\pv V}$, we conclude
that $u\in \clos{V}(L(X))$.
  \end{proof}

  \begin{remark}
    \label{r:closure-not-factorial-for-non-minimal-shift-space} By
Proposition~\ref{p:infinite-closure-J-class}, if $X$ is a minimal
shift space of $A^{\zz}$ and $\pv V$ is a pseudovariety of semigroups
containing $\pv{\loc Sl}$, then $\clos{V}(L(X))$ is factorial in $\Om
AV$, as the finite factors of elements of $J_{\pv V}(X)$ belong to
$L(X)$ by Proposition~\ref{p:mirage}. But one may have a pseudovariety $\pv V$
containing $\pv{\loc Sl}$ and a shift space $X$ not minimal such that $\clos{V}(L(X))$ is not a factorial
subset of $\Om AV$ (cf.~\cite[Example 3.4]{ACosta&Steinberg:2021}).
  \end{remark}

Recall that if \pv V contains the pseudovariety
\pv{\loc I}, then every infinite-length pseudoword $w\in \Om AV\setminus A^+$
has a well-defined right infinite prefix $\overrightarrow{w}\in A^\nn$
and left-infinite suffix $\overleftarrow{x}\in A^{\zz_{-}}$ (see
Section~\ref{ss:rel-free}). Consider the mapping $\pli\from\Om AV\setminus A^+\to
A^\zz$ defined by $\pli(x) = \overleftarrow{x}{\cdot}\overrightarrow{x}$.
The next result shows that this mapping characterizes the
\green{H}-classes of $J_\pv{V}(X)$; it was originally proved by the
first author~\cite[Theorem~3.3]{Almeida:2003a} (see
also~\cite[Lemma~6.6]{Almeida&ACosta:2007a}).

\begin{lemma}
  \label{l:characterization-of-Green-H-in-J-class-of-minimal-shift}
Let $X$ be a minimal shift space and $\pv V$ be a pseudovariety of
semigroups containing~\pv{\loc I}. Then, for every $u,v\in
J_\pv{V}(X)$, the equality $\pli(u) = \pli(v)$ holds if and only if
$u\green{H} v$.
More precisely,
for every $u,v\in J_\pv{V}(X)$ we have
$\overrightarrow{u}=\overrightarrow{v}$ if and only if $u\green{R} v$,
and
$\overleftarrow{u}=\overleftarrow{v}$ if and only if $u\green{L} v$.
\end{lemma}

It follows that the mapping $\bar{\pli}(u/{\green{H}}) = \pli(u)$, with $u\in J_\pv{V}(X)$, is
well defined. For an element $x\in A^\zz$, let
\begin{equation*} x(-\infty,0) = \cdots x[-2]x[-1]\in A^{\zz_{-}},\quad
x[0,\infty) = x[0]x[1]\cdots\in A^\nn.
\end{equation*}
Lemma~\ref{l:characterization-of-Green-H-in-J-class-of-minimal-shift}
says in particular that the mapping $\bar{\pli}$ is a bijection
between the $\green{H}$-classes of $J_\pv{V}(X)$ and the following
set:
\begin{equation*} \{ y(-\infty,0)\cdot x[0,\infty) : x, y\in X\}.
\end{equation*}

The next result locates the maximal subgroups in $J_\pv{V}(X)$.

\begin{proposition}[{\cite[Lemma~5.3]{Almeida&ACosta:2012}}]
  \label{p:parametrization-of-H-classes} Let $X$ be a minimal shift
space and $\pv V$ be a pseudovariety of semigroups containing~\pv{\loc
Sl}. An $\green{H}$-class $H$ of $J_\pv{V}(X)$ contains an idempotent
if and only if $\bar{\pli}(H)\in X$. Moreover, every element of~$X$ is of
the form $\pli(e)$ for a unique idempotent $e$ of~$J_{\pv
V}(X)$.
\end{proposition}

According to the following corollary, the shape of the
$\green{J}$-class $J_\pv{V}(X)$ is independent of the pseudovariety
$\pv{V}$, provided $\pv{\loc Sl}\subseteq\pv{V}$.

\begin{corollary}\label{c:projecting-the-J-class-of-X}
  Let $\pv V$ and $\pv W$ be pseudovarieties of semigroups such that
  $\pv{\loc Sl}\subseteq \pv W\subseteq\pv V$. Let $X$ be a minimal
  shift space of $A^\zz$. The following properties hold:
    \begin{enumerate}
    \item If $H$ is an \green{H}-class of $J_{\pv V}(X)$, then the set
$p_{\pv V,\pv W}(H)$ is an \green{H}-class of $J_{\pv W}(X)$.
    \item If $K$ is an \green{H}-class of $J_{\pv W}(X)$, then the set
$p_{\pv V,\pv W}^{-1}(K)$ is an \green{H}-class of $J_{\pv V}(X)$.
    \end{enumerate}
\end{corollary}

\begin{proof} Note that $\pli(p_{\pv V,\pv W}(u))=\pli(u)$ for every
$u\in\Om AV$. The corollary now follows immediately from
Corollary~\ref{c:mirage}.
\end{proof}

We apply again Proposition~\ref{p:parametrization-of-H-classes} to show the following lemma.

\begin{lemma}
   \label{l:cancelative-property-in-JX}
   Let $X$ be a minimal shift
   space of $A^\zz$ and $\pv V$ be a pseudovariety of semigroups
   containing~\pv{\loc Sl}. Let $u,v\in\Om AV$ and $x,y\in A^*$ be such that $|x|=|y|$.
\begin{enumerate}
\item If $xu\in J_{\pv V}(X)$ and $xu=yv$, then we have $x=y$ and $u=v$.\label{item::cancelative-property-in-JX-1}
\item If $ux\in J_{\pv V}(X)$ and $ux=vy$, then we have  $x=y$ and $u=v$.\label{item::cancelative-property-in-JX-2}
\end{enumerate}
\end{lemma}

\begin{proof}
  Suppose that $xu=yv\in J_{\pv V}(X)$. As $\pv V$ contains $\pv{\loc I}$, every pseudoword of $\Om AV$ of length at least $n$
  has a unique prefix and a unique suffix of length $n$, whenever $n\in\nn$. In particular, we have $x=y$.
  Since $x$ has finite length and $xu$ has infinite length, both $u,v$ and have infinite length,
  whence $u,v\in J_{\pv V}(X)$ by Proposition~\ref{p:mirage}.
  As $\Om AV$ is stable, it follows that $u\green{L}xu=xv\green{L}v$.
  Also because $x$ has finite length, we have
  \begin{equation*}
    x\cdot\overrightarrow{u}=\overrightarrow{xu}=\overrightarrow{xv}=x\cdot \overrightarrow{v},
  \end{equation*}
  thus $\overrightarrow{u}=\overrightarrow{v}$.
  We deduce from Lemma~\ref{l:characterization-of-Green-H-in-J-class-of-minimal-shift} that $u\green{H}v$.
  
  By Green's Lemma, the mapping $H_u\to H_{xu}$
  sending each element $w$ in the $\green{H}$-class $H_u$ to $xw$ is a bijection (see~\cite[Lemma~A.3.1]{Rhodes&Steinberg:2009qt}).
  In particular, since $u,v\in H_u$, it follows from
  the equality $xu=xv$ that $u=v$.
  This shows~\ref{item::cancelative-property-in-JX-1},
  and the proof of~\ref{item::cancelative-property-in-JX-2}
  follows by symmetric arguments.
\end{proof}

\begin{remark}
  \label{r:weak-generalization-letter-super-cancelativity}
  When $\pv V=\pv S$, Lemma~\ref{l:cancelative-property-in-JX} is a special case
  of Proposition~\ref{p:letter-super-cancelativity}.
  While Proposition~\ref{p:letter-super-cancelativity}
  still holds if we replace $\pv S$ by many other pseudovarieties $\pv V$~\cite[Proposition 6.4]{Almeida&ACosta&Goulet-Ouellet:2024a}, it does not hold for all
  $\pv V$ containing~$\pv{\loc Sl}$ (cf.~\cite[Proposition 6.2]{Almeida&ACosta&Costa&Zeitoun:2019}).
\end{remark}

The following proposition is used in the proof of Theorem~\ref{t:a-sort-of-converse-of-surjectivity-theorem}.

 \begin{proposition}
   \label{p:shift-of-idempotents}
   Let $X$ be a minimal shift
space and $\pv V$ be a pseudovariety of semigroups containing~\pv{\loc Sl}. Let $e,f$ be idempotents in $J_{\pv V}(X)$.
Let  $n$ be a positive integer.
The following conditions are equivalent:
\begin{enumerate}
\item the equality $\pli(e)=\shift^n(\pli(f))$ holds;\label{item:shift-of-idempotents-1}
\item one has $pe=fp$ for some word $p$ of length $n$;\label{item:shift-of-idempotents-2}
\item one has $pe\green{H}fp$ for some word $p$ of length $n$.\label{item:shift-of-idempotents-3}
\end{enumerate}
Moreover, if $p$ is a word of length $n$
such that $pe\green{H}fp$, then $pe$ and $fp$ belong to $J_{\pv V}(X)$, and the equalities $p=f{[}0,n{)}=e{[}-n,-1{]}$
and $pe=fp$ hold.
\end{proposition}

 \begin{proof}
  $ \ref{item:shift-of-idempotents-2}
  \Rightarrow
  \ref{item:shift-of-idempotents-3}$ This implication is trivial. 
   
   $\ref{item:shift-of-idempotents-3}
   \Rightarrow
   \ref{item:shift-of-idempotents-1}$
   From $pe\green{H}fp$ we get, on one hand, the equalities $\overrightarrow{pe}=\overrightarrow{fp}=\overrightarrow{f}$, whence
   \begin{equation}
     \label{eq:shift-of-idempotents-1}
     e{[}0,\infty{)}=f{[}n,\infty{)};
   \end{equation}
   and, on the other hand, the equalities $\overleftarrow{e}=\overleftarrow{pe}=\overleftarrow{fp}$,
   thus
   \begin{equation}
     \label{eq:shift-of-idempotents-2}
    e{(}-\infty,-1{]}=f{(}-\infty,n-1{]}.
   \end{equation}
   Combining~\eqref{eq:shift-of-idempotents-1} and~\eqref{eq:shift-of-idempotents-2}, we obtain $ \pli(e)=\shift^n(\pli(f))$.
   
   $\ref{item:shift-of-idempotents-1}
   \Rightarrow
   \ref{item:shift-of-idempotents-2}$
   Assuming that $\pli(e)=\shift^n(\pli(f))$, we have $e{[}-n,-1{]}=f{[}0,n{)}$. Set $p=e{[}-n,-1{]}$, and consider the factorization $f=pt$.
     By Proposition~\ref{p:mirage}, the infinite-length pseudoword $t$ belongs to $J_{\pv V}(X)$, and so
     $f\green{L}t$ because profinite semigroups are stable. As $f$ is idempotent, we then have $t=tf$.

   Consider the infinite-length pseudoword $g=tp$. Since $f=f^2=ptpt=pgt$, we know that $g\in J_{\pv V}(X)$ by Proposition~\ref{p:mirage}.
   Note that $pg=fp$. Hence,  it suffices to show that $g=e$ to
   conclude the proof of the implication
   $\ref{item:shift-of-idempotents-1}
   \Rightarrow
   \ref{item:shift-of-idempotents-2}$.
   We first check that $g$ is idempotent: indeed, as $pt=f$ and $t=tf$, we have $g^2=tptp=tfp=tp=g$.
   Then, it follows from the equality $pg=fp$ and the already
   established implication
   $\ref{item:shift-of-idempotents-3}
   \Rightarrow
   \ref{item:shift-of-idempotents-1}$
   (with $g$ playing the role of $e$ in that implication) that
   \begin{equation*}
     \pli(g)=\shift^n(\pli(f))=\pli(e).
   \end{equation*}
   This implies $g\green{H}e$ by Lemma~\ref{l:characterization-of-Green-H-in-J-class-of-minimal-shift},
   which means that $g=e$ as $g$ and $e$ are idempotents.

   We have therefore established the chain of equivalences
   $\ref{item:shift-of-idempotents-1}
   \Leftrightarrow
   \ref{item:shift-of-idempotents-2}
   \Leftrightarrow
   \ref{item:shift-of-idempotents-3}$.
   It remains to justify the last sentence in the proposition.
   Suppose that $pe\green{H}fp$ for a word $p$ of length $n$. Then
   in fact we have $pe=fp$, as we already proved the implication
   $\ref{item:shift-of-idempotents-3}
   \Rightarrow
   \ref{item:shift-of-idempotents-2}$.
   Note that $f$ is a prefix of $pe$, whence $p=f{[}0,n{)}$
   by the unicity of the prefix of length $n$ in any infinite-length pseudoword of $\Om AV$.
   Similarly, $p$ is the suffix $e{[}-n,-1{]}$ of $e$.
   Let $t$ be such that $e=tp$. Then $e=e^2=tpe$. Since $pe$ is an infinite-length factor of $e$, it follows from Proposition~\ref{p:mirage}
   that $pe\in J_{\pv V}(X)$. Similarly, we have $fp\in J_{\pv V}(X)$.
 \end{proof}
 
\section{Profinite images of directive sequences}
\label{sec:profinite-images}

In this section, we consider a directive
sequence $\bsigma = (\sigma_n)_{n\in\nn}$, with $\sigma_n\from A_{n+1}^+\to A_n^+$,
and a pseudovariety  of semigroups $\pv V$ containing $\pv N$.

Recall that $\prov V\sigma_n\from\Om {A_{n+1}}V\to\Om{A_n}V$
and $\prov V\sigma_{m,n}\from \Om {A_{n}}V\to\Om{A_m}V$
are the unique continuous homomorphisms extending
$\sigma_n\from A_{n+1}^+\to A_n^+$ and
$\sigma_{m,n}\from A_{n}^+\to A_m^+$, respectively,
and that $\prov V\sigma_{m,n}=\prov V\sigma_m\circ\cdots\circ\prov V\sigma_{n-1}$
(cf.~Subsection~\ref{ss:rel-free}).

\begin{definition}
  \label{d:V-image}
  The \emph{$\pv{V}$-image} of $\bsigma$, denoted $\img_\pv{V}(\bsigma)$, is the intersection
  \begin{equation*}
    \bigcap_{n\in\nn} \img(\prov V\sigma_{0,n}).
  \end{equation*}
   By a \emph{profinite image} of $\bsigma$ we mean a
 set of the form $\img_{\pv V}(\bsigma)$ for some pseudovariety~$\pv V$.
\end{definition}

\begin{remark}
  \label{r:V-image}
  Since the sequence of sets $\img(\sigma_{0,n}^{\pv V})$ is a chain for the reverse inclusion,
  we have 
  \begin{equation*}
    \img_\pv{V}(\bsigma)=\bigcap_{k\in\nn} \img(\prov V\sigma_{0,n_k})
  \end{equation*}
  for every strictly increasing sequence $(n_k)_{k\in\nn}$ of nonnegative integers.
\end{remark}

\begin{remark}
  \label{r:V-image-is-a-semigroup}
  The set $\img_\pv{V}(\bsigma)$ is a closed subsemigroup of $\Om {A_0}V$; indeed, $\img_\pv{V}(\bsigma)$ is a nonempty compact space by the finite intersection property of compact spaces.
\end{remark}

We next register that a contraction does not change the $\pv V$-image,
which is an immediate consequence of Remark~\ref{r:V-image}.

 \begin{lemma}
   \label{l:V-image-contraction}
   If $\btau$ is a contraction of $\bsigma$, then
   $\img_\pv{V}(\btau)=\img_\pv{V}(\bsigma)$.
 \end{lemma}

We proceed to establish an elementary technical lemma which will be used several times.
\begin{lemma}
    \label{l:cluster}
    Let $(I_n)_{n\in\nn}$ be a sequence of subsets of a compact metric space $M$. Let $C$ be the set of cluster points of sequences $(x_n)_{n\in\nn}$ such that $x_n\in I_n$ for all $n\in\nn$. 
    \begin{enumerate}
        \item The set $C$ is closed.\label{i:cluster-closed}
        \item If $(I_n)_{n\in\nn}$ is a descending chain, then $C = \bigcap_{n\in\nn}\overline I_n$.\label{i:cluster-chain}
    \end{enumerate}
\end{lemma}

\begin{proof}
    \ref{i:cluster-closed} Let $(x_n)_{n\in\nn}$ be a sequence of elements of $C$
    converging to an element $x\in M$.
    Up to taking a subsequence, we may assume that $d(x,x_k)<\frac{1}{2k}$ for every $k\geq 1$.
    We recursively build a strictly increasing sequence $(n_k)_{k\geq 1}$ of positive integers,
    together with a sequence $(y_k)_{k\geq 1}$ of elements such that $y_k\in I_{n_k}$, as follows:
    \begin{itemize}
        \item $n_0=0$ and $y_0$ is any element of $I_0$;
        \item if $k>0$, then $n_k\in\nn$ and $y_k\in I_{n_k}$ are chosen such that $n_k>n_{k-1}$ and $d(x_{k},y_k)<\frac{1}{2k}$.
            Such $n_k$ and $y_k$ must exist by the definition of the set $C$, to which $x_{k}$ belongs.
    \end{itemize}
    We then have $d(x,y_k)\leq d(x,x_{k})+d(x_{k},y_k)<\frac{1}{2k}+\frac{1}{2k}=\frac{1}{k}$ for every $k\geq 1$.
    It follows that $\lim_k y_k=x$, whence $x\in C$. 
    This proves that $C$ is closed in $M$.

    \ref{i:cluster-chain} Let $I = \bigcap_{n\in\nn}\overline I_n$.
    We first establish the inclusion $C\subseteq I$.   
    Let $x\in C$. Take a
    strictly increasing sequence $(n_k)_{k\in\nn}$
    of positive integers and a sequence $(x_k)_{k\in\nn}$ of elements of $M$
    converging to $x$ such that $x_k\in I_{n_k}$ for each $k\in\nn$.
    Fix $r\in\nn$. If $k\geq r$, then $n_k\geq r$ and so the inclusion $I_{n_k}\subseteq I_r$ holds.
    Hence, $x=\lim_{k\geq r} x_k$ is in the closed subspace $\overline{I_r}$.
    Since $r$ is arbitrary, this shows that $x\in I$.

    Finally we prove the inclusion $I\subseteq C$.
    Let $x\in I$.
    Let $k$ be a positive integer. Since $x$ is in the topological closure of $I_k$,
    there is $y_k\in I_k$ such that $d(x,y_k)<\frac{1}{k}$.
    It follows that $\lim_k y_k = x$, which shows that $x\in C$.
\end{proof}

Applying Lemma~\ref{l:cluster}\ref{i:cluster-chain} to the sequences of subsets
\[
  (\img(\prov V\sigma_{0,n}))_{n\in\nn},\quad (\img(\sigma_{0,n}))_{n\in\nn}
\]
yields the following, which we state for convenience.
 
\begin{lemma}
  \label{l:V-image-as-a-set-of-cluster-points}
  Consider the following sets:
  \begin{enumerate}
  \item the set $C$ of cluster points, in the space $\Om {A_0}V$, 
    of sequences $(w_n)_{n\in\nn}$
    of pseudowords
   such that $w_n\in \img(\prov V\sigma_{0,n})$ for every $n\in\nn$;
 \item the set $D$ of cluster points, in the space $\Om {A_0}V$, 
   of sequences $(w_n)_{n\in\nn}$ of words
   such that $w_n\in \img(\sigma_{0,n})$ for every $n\in\nn$.
  \end{enumerate}
   Then the equalities $\img_\pv{V}(\bsigma)=C=D$ hold.
 \end{lemma}

There is a natural relationship between profinite images of $\bsigma$ relative
to comparable pseudovarieties.

\begin{proposition}\label{p:projecting-profinite-image}
  Let $\pv V$ and $\pv W$ be pseudovarieties of semigroups such that
  $\pv {N}\subseteq\pv{W}\subseteq\pv{V}$. The following equality holds:
  \begin{equation*}
    \img_\pv{W}(\bsigma) = p_{\pv{V},\pv{W}}(\img_\pv{V}(\bsigma)).
  \end{equation*}
\end{proposition}

\begin{proof}
  As $\prov W\sigma_{0,n}\circ p_{\pv{V},\pv{W}}=p_{\pv{V},\pv{W}}\circ\prov V\sigma_{0,n}$, we clearly have
  $
    \img(\prov W\sigma_{0,n})=p_{\pv{V},\pv{W}}(\img(\prov V\sigma_{0,n}))
  $
  and so we immediately obtain
  \begin{equation*}
    p_{\pv{V},\pv{W}}(\img_\pv{V}(\bsigma))
    \subseteq\bigcap_{n\in\nn} p_{\pv{V},\pv{W}}(\img(\prov V\sigma_{0,n}))=\img_\pv{W}(\bsigma).
  \end{equation*}
  
  Conversely, let $w\in \img_\pv{W}(\bsigma)$. For each $n\in\nn$, we may take
  $v_n\in\img(\prov V\sigma_{0,n})$ such that $w=p_{\pv{V},\pv{W}}(v_n)$.
  Let $v$ be a cluster point of the sequence $(v_n)_n$. By continuity, we get $w=p_{\pv{V},\pv{W}}(v)$.
  On the other hand, we have $v\in \img_\pv{V}(\bsigma)$ by Lemma~\ref{l:V-image-as-a-set-of-cluster-points}.
  This establishes the inclusion $\img_\pv{W}(\bsigma)\subseteq p_{\pv{V},\pv{W}}(\img_\pv{V}(\bsigma))$, thus concluding the proof.
\end{proof}

Denote by $\Lambda_{\pv V}(\bsigma)$
the set of pseudowords
of $\Om {A_0}V$ that are cluster points of
some sequence $(w_n)_{n\in\nn}$ such that $w_n\in \sigma_{0,n}(A_n)$ for every $n\in\nn$.
The following is a simple application of Lemma~\ref{l:cluster}.

\begin{lemma}
  \label{p:lambda-is-closed}
  The set $\Lambda_{\pv V}(\bsigma)$ is a closed subspace of\/ $\img_\pv{V}(\bsigma)$.
\end{lemma}

\begin{proof}
  The inclusion $\Lambda_{\pv V}(\bsigma) \subseteq
  \img_\pv{V}(\bsigma)$ follows from the instance of
  Lemma~\ref{l:cluster}\ref{i:cluster-chain} presented in
  Lemma~\ref{l:V-image-as-a-set-of-cluster-points}. The fact that
  $\Lambda_{\pv V}(\bsigma)$ is closed follows from
  Lemma~\ref{l:cluster}\ref{i:cluster-closed}.
\end{proof}

We next establish some properties of the set $\Lambda_{\pv V}(\bsigma)$ in the case on which we focus: the case where $\bsigma$ is primitive.

\begin{proposition}
  \label{p:properties-of-lambda}
  Let $\bsigma$ be a primitive directive sequence.
  The following properties hold:
  \begin{enumerate}
  \item $\Lambda_{\pv V}(\bsigma)$ generates a dense subsemigroup of
    $\img_{\pv V}(\bsigma)$;\label{item:properties-of-lambda-0}
  \item $\Lambda_{\pv V}(\bsigma)\subseteq J_{\pv V}(\bsigma)\cap
    \img_{\pv V}(\bsigma)$;\label{item:properties-of-lambda-1}
  \item $\img_{\pv V}(\bsigma)=\Lambda_{\pv V}(\bsigma)\cdot \img_{\pv
      V}(\bsigma)=\img_{\pv V}(\bsigma)\cdot\Lambda_{\pv
      V}(\bsigma)$;\label{item:properties-of-lambda-2}
  \item $\Lambda_{\pv V}(\bsigma)$ is contained in a regular
    $\green{J}$-class of the semigroup $\img_{\pv
      V}(\bsigma)$.\label{item:properties-of-lambda-3}
  \end{enumerate}
\end{proposition}

\begin{proof}
  \ref{item:properties-of-lambda-0} Note that profinite semigroups
  embed as topological semigroups into products of finite semigroups.
  Hence, it suffices to show that, for every continuous homomorphism
  $\varphi:\Om{A_0}V\to S$ into a finite semigroup $S$, and every
  element $w\in\img_{\pv V}(\bsigma)$, there is a product $u$ of
  finitely many elements of~$\Lambda_{\pv V}(\bsigma)$ such that
  $\varphi(w)=\varphi(u)$. Let $(w_k)_{k\in\mathbb{N}}$ be a sequence
  of words $w_k\in A_{n_k}^+$ such that $w=\lim\sigma_{0,n_k}(w_k)$.
  For each $k\in\mathbb{N}$, let $u_k\in A_{n_k}^+$ be a word of
  minimum length such that
  $\varphi\bigl(\sigma_{0,n_k}(w_k)\bigr)=\varphi\bigl(\sigma_{0,n_k}(u_k)\bigr)$.
  By the minimality assumption on $u_k$, the values under
  $\varphi\circ\sigma_{0,n_k}$ of the prefixes of~$u_k$ must be
  distinct, so that $|u_k|\le|S|$. By taking subsequences, we may well
  assume that all $u_k$ have the same length $\ell$. For each
  $k\in\mathbb{N}$, write $u_k=a_{k,1}\cdots a_{k,\ell}$ with the
  $a_{k,i}\in A_{n_k}^+$. By compactness, up to further taking
  subsequences, we may assume that each of the sequences
  $\bigl(\sigma_{0,n_k}(a_{k,i})\bigr)_k$ converges to the element
  $a_i$ of $\Lambda_{\pv V}(\bsigma)$. Then $u=a_1\cdots a_\ell$ has
  the required property as, for all sufficiently large $k$, the
  following equalities hold:
  \begin{align*}
    \varphi(w)
    &=\varphi\bigl(\sigma_{0,n_k}(w_k)\bigr)
    =\varphi\bigl(\sigma_{0,n_k}(u_k)\bigr)\\
    &=\varphi\bigl(\sigma_{0,n_k}(a_{k,1})\bigr)
      \cdots\varphi\bigl(\sigma_{0,n_k}(a_{k,\ell})\bigr)
    =\varphi(a_1\cdots a_\ell)
    =\varphi(u).
  \end{align*}

  \ref{item:properties-of-lambda-1} Since the set $\sigma_{0,n}(A_n)$
  is contained in $L(\bsigma)$ for every $n\in\nn$, we clearly have
  $\Lambda_{\pv V}(\bsigma)\subseteq \clos V(L(\bsigma))$. Moreover,
  the fact that $\bsigma$ is primitive also ensures that
  $\lim_{n\to\infty}\min\{|\sigma_{0,n}(a)|:a\in A_n\}=\infty$.
  Therefore, and in view of Lemma~\ref{p:lambda-is-closed}, we have
  indeed $\Lambda_{\pv V}(\bsigma)\subseteq J_{\pv V}(\bsigma)\cap
  \img_{\pv V}(\bsigma)$.
  
  \ref{item:properties-of-lambda-2} We show the equality $\img_{\pv
    V}(\bsigma)=\Lambda_{\pv V}(\bsigma)\cdot \img_{\pv V}(\bsigma)$.

  The inclusion $\Lambda_{\pv V}(\bsigma)\cdot \img_{\pv
    V}(\bsigma)\subseteq \img_{\pv V}(\bsigma)$ clearly holds as
  $\Lambda_{\pv V}(\bsigma)\subseteq \img_{\pv V}(\bsigma)$ and
  $\img_{\pv V}(\bsigma)$ is a semigroup. Conversely, let $w\in
  \img_{\pv V}(\bsigma)$. Then, by
  Lemma~\ref{l:V-image-as-a-set-of-cluster-points}, we have
  $w=\lim\sigma_{0,n_k}(u_k)$ for some strictly increasing sequence
  $(n_k)_{k\in\nn}$ of positive integers and a sequence
  $(u_k)_{k\in\nn}$ such that $u_k\in (A_{n_k})^+$ for every
  $k\in\nn$. Since $\bsigma$ is primitive, for each $k\in\nn$ we may
  choose some $r(k)\in\nn$ such that the word
  $w_k=\sigma_{n_k,n_{r(k)}}(u_{r(k)})$ has length at least two.
  Moreover, we may build the sequence $(r(k))_{k\in\nn}$ so that it is
  strictly increasing. For such a sequence, we have
  $\lim_{k\to\infty}\sigma_{0,n_k}(w_k)
  =\lim_{k\to\infty}\sigma_{0,n_{r(k)}}(u_{r(k)})=w$.
  
  For each $k\in\nn$, since $|w_k|\geq 2$, there are $a_k\in A_{n_k}$
  and $s_k\in (A_{n_k})^+$ such that $u_k=a_ks_k$. Let $(a,s)$ be an
  accumulation point in $\Om {A_{0}}V\times \Om {A_{0}}V$ of the
  sequence $(\sigma_{0,n_k}(a_k),\sigma_{0,n_k}(s_k))_{k}$. Since
  $\lim \sigma_{0,n_k}(a_k)\sigma_{0,n_k}(s_k)=\lim
  \sigma_{0,n_k}(u_k)=w$, we have $w=as$ by continuity of the
  multiplication. Note also that $(a,s)\in \Lambda_{\pv
    V}(\bsigma)\times \img_{\pv V}(\bsigma)$ by the definition of $
  \Lambda_{\pv V}(\bsigma)$ and by
  Lemma~\ref{l:V-image-as-a-set-of-cluster-points}. Therefore, we have
  $w\in \Lambda_{\pv V}(\bsigma)\cdot \img_{\pv V}(\bsigma)$. This
  concludes the proof of the equality $\img_{\pv
    V}(\bsigma)=\Lambda_{\pv V}(\bsigma)\cdot \img_{\pv V}(\bsigma)$.
  The proof of the equality $\img_{\pv V}(\bsigma)= \img_{\pv
    V}(\bsigma)\cdot \Lambda_{\pv V}(\bsigma)$ is entirely similar.

  \ref{item:properties-of-lambda-3} Let $a,b\in\Lambda_{\pv
    V}(\bsigma)$. Then there are strictly increasing sequences
  $(n_k)_{k\in\nn}$ and $(m_k)_{k\in\nn}$ of positive integers such
  that
  \begin{equation*}
    a=\lim\sigma_{0,n_k}(a_k)\qquad\text{and}\qquad b=\lim\sigma_{0,m_k}(b_k)
  \end{equation*}
  for some sequences $(a_k)_{k\in\nn}$ and $(b_k)_{k\in\nn}$ for which
  we have $a_k\in A_{n_k}$ and $b_k\in B_{m_k}$ for every $k\in\nn$.
  Since $\bsigma$ is primitive, for each $k\in \nn$ we may choose some
  $r(k)>k$ such that $n_{r(k)}>m_k$ and
  $\fac(\sigma_{m_k,n_{r(k)}}(A_{n_{r(k)}}))\supseteq A_{m_k}$.
  Moreover, the sequence $(r(k))_{k\in\nn}$ may be chosen to be
  strictly increasing. Going on with such a choice, we have, for each
  $k\in\nn$, a factorization
  $$\sigma_{m_k,n_{r(k)}}(a_{n_{r(k)}})=p_kb_ks_k$$
  with $p_k,s_k\in (A_{n_k})^*$, and with at least one of the words
  $p_k,s_k$ being nonempty. By compactness, we may extract from the
  sequence $(\sigma_{0,m_k}(p_k),\sigma_{0,m_k}(s_k))_{k\in\nn}$ a
  subsequence
  $(\sigma_{0,m_{k_i}}(p_{k_i}),\sigma_{0,m_{k_i}}(s_{k_i}))_{i\in\nn}$
  converging in $(\Om {A_0}V)^1\times (\Om {A_0}V)^1$ to some pair
  $(p,s)$. We then have
  \begin{equation}\label{eq:properties-of-lambda-1}
    a=\lim_{i\in\nn}\sigma_{0,n_{r(k_i)}}(a_{n_{r(k_i)}})
    =\lim_{i\in\nn}\Bigl(\sigma_{0,m_{k_i}}(p_{k_i})\cdot
    \sigma_{0,m_{k_i}}(b_{k_i})\cdot
    \sigma_{0,m_{k_i}}(s_{k_i})\Bigr)
    =pbs.
  \end{equation}
  Note that $p,s\in \img_{\pv V}(\bsigma)\cup\{\emptyw\}$ by
  Lemma~\ref{l:V-image-as-a-set-of-cluster-points}. This shows that
  $a\leq_{\green{J}}b$ in $\img_{\pv V}(\bsigma)$. Since $a,b$ are
  arbitrary elements of $\Lambda_{\pv V}(\bsigma)$, we then get
  $a\green{J}b$ in $\img_{\pv V}(\bsigma)$. Going back
  to~\eqref{eq:properties-of-lambda-1}, and taking $a=b$, we get
  $a=pas$. Since $ps=\lim \sigma_{0,m_k}(p_ks_k)$ and
  $p_ks_k\neq\emptyw$ for every $k\in\nn$, at least one of the
  pseudowords $p,s$ is not the empty word. Without loss of generality,
  assume that $p\neq\emptyw$. From $a=pas$, we obtain $a=p^kas^k$ for
  every $k\in\nn$, whence $a=p^\omega as^\omega$, which in turn yields
  $a\leq_{\green{J}}p^\omega$ in $\img_{\pv V}(\bsigma)$. On the other
  hand, for some $c\in \Lambda_{\pv V}(\bsigma)$ we have
  $p^\omega\leq_{\green{J}}c$ in $\img_{\pv V}(\bsigma)$, by the
  already shown item~\ref{item:properties-of-lambda-2}. But we already
  proved that all elements of $\Lambda_{\pv V}(\bsigma)$ are
  $\green{J}$-equivalent in $\img_{\pv V}(\bsigma)$. Joining all
  pieces, we see that $p^\omega\green{J}a$ in $\img_{\pv V}(\bsigma)$.
  This shows that $a$ is regular in $\img_{\pv V}(\bsigma)$,
  concluding the proof that $\Lambda_{\pv V}(\bsigma)$ is contained in
  a regular $\green{J}$-class of the profinite semigroup $\img_{\pv
    V}(\bsigma)$.
\end{proof}

The next theorem will play a key role in Section~\ref{sec:satur-direct-sequ}.

\begin{theorem}
  \label{t:Im-intersection-J}
  Let $\bsigma$ be a primitive directive sequence. The set $J_{\pv
    V}(\bsigma)\cap\img_{\pv V}(\bsigma)$ is a regular
  $\green{J}$-class of the semigroup\/ $\img_{\pv V}(\bsigma)$.
\end{theorem}

\begin{proof}
  By Proposition~\ref{p:properties-of-lambda}, the set $\Lambda_{\pv
    V}(\bsigma)$ is contained in a regular $\green{J}$-class $J$ of
  $\img_{\pv V}(\bsigma)$. We also know by
  Proposition~\ref{p:properties-of-lambda} that $\Lambda_{\pv
    V}(\bsigma)\subseteq J_{\pv V}(\bsigma)$, and so we already know
  that $J\subseteq J_{\pv V}(\bsigma)\cap\img_{\pv V}(\bsigma)$.

  Conversely, let $u$ be an element of $J_{\pv
    V}(\bsigma)\cap\img_{\pv V}(\bsigma)$. By
  Proposition~\ref{p:properties-of-lambda}, there are idempotents
  $e\in J$ and $f\in J$ such that $u=eu=uf$. In particular, it
  suffices to show that $e\in u\img_{\pv V}(\bsigma)$ to get $u\in J$.

  As $u$ and $e$ belong to the same $\green{J}$-class of the stable
  semigroup $\Om {A_0}V$, the equality $u=eu$ yields the existence of
  some pseudoword $x$ such that $ux=e$. We may assume that $x=fxe$,
  because $e$ and $f$ are idempotents and $u=uf$. Under such
  assumption, the pseudowords $x$ and $e$ belong the same
  $\green{L}$-class of $\Om {A_0}V$. Therefore, the equality
  \begin{equation*}
    uH_x=H_e
  \end{equation*}
  holds by Green's Lemma
  (cf.~\cite[Lemma~A.3.1]{Rhodes&Steinberg:2009qt}), where $H_s$
  denotes the $\green{H}$-class in $\Om {A_0}V$ of the pseudoword $s$.

  From the equalities $x=fxe$ and $e=ux$, and from $e\green{J}_{\Om
    {A_0}V}f$, we obtain $f\green{R}_{\Om {A_0}V}x\green{L}_{\Om
    {A_0}V}e$, by stability of $\Om {A_0}V$; on the other hand, since
  $e,f\in J$, there is $w\in \img_{\pv V}(\bsigma)$ such that
  $f\green{R}_{\img_{\pv V}(\bsigma)}w\green{L}_{\img_{\pv
      V}(\bsigma)}e$. This implies that $w\in H_x$, thus $uw\in H_e$.
  Since $H_e$ is a profinite group with identity $e$, it follows that
  $(uw)^{\omega}=e$. That is, for the pseudoword $z=w(uw)^{\omega-1}$,
  we have $e=uz$. Since $u,w\in \img_{\pv V}(\bsigma)$, the pseudoword
  $z$ belongs to $\img_{\pv V}(\bsigma)$. This shows that indeed $e\in
  u\img_{\pv V}(\bsigma)$, which, as already noted, yields $u\in J$.
  This establishes the inclusion $J_{\pv V}(\bsigma)\cap\img_{\pv
    V}(\bsigma)\subseteq J$.
\end{proof}

\begin{corollary}
  \label{c:Im-intersection-J}
  Let $\bsigma$ be a primitive directive sequence. The inclusion
  $\img_{\pv V}(\bsigma)\subseteq J_{\pv V}(\bsigma)$ holds if and
  only if the profinite semigroup $\img_{\pv V}(\bsigma)$ is simple.
\end{corollary}

\begin{proof}
  By Theorem~\ref{t:Im-intersection-J}, the intersection $J_{\pv
    V}(\bsigma)\cap\img_{\pv V}(\bsigma)$ is a $\green{J}$-class of
  the semigroup\/ $\img_{\pv V}(\bsigma)$. Therefore, the
  $\green{J}$-relation of $\img_{\pv V}(\bsigma)$ is universal in
  $\img_{\pv V}(\bsigma)$ if and only if the inclusion $\img_{\pv
    V}(\bsigma)\subseteq J_{\pv V}(\bsigma)$ holds.
\end{proof}

In the next example we see that $\img_{\pv V}(\bsigma)$ may not be
contained in $J_{\pv V}(\bsigma)$.

\begin{example}
  \label{e:example-that-img-may-not-be-in-J}
  Let $\sigma$ be the primitive substitution on the alphabet $A=\{\ltr a,\ltr
  b,\ltr c\}$ defined by
  \begin{equation*}
    \sigma\from
    \ltr a\mapsto \ltr{ac},\
    \ltr b\mapsto \ltr{bcb},\
    \ltr c\mapsto \ltr{ba},
  \end{equation*}
  and consider the constant directive sequence
  $\bsigma=(\sigma,\sigma,\ldots)$.
  Let $w$ be any cluster point in $\Om A{\pv S}$
  of the sequence $\sigma^{2n}(\ltr a)$.
  Note that $w\in\Lambda_{\pv S}(\bsigma)$,
  and so we have  $w\in J_{\pv S}(\bsigma)\cap \img_{\pv S}(\bsigma)$.
  Since $\sigma^2(\ltr a)=\ltr{acba}$, the pseudoword $w$ starts and
  ends with $\ltr a$, thus $\ltr a^2$ is a factor of $w^2$. 
  On the other hand, since $\ltr a^2$ is not a factor of any of the
  words $\sigma^{2n}(\ltr a)$, we also know that 
  $\ltr a^2$ is not a factor of $w$.
  Hence, the element $w^2$ of $\img_{\pv S}(\bsigma)$ does not
  belong to~$J_{\pv S}(\bsigma)$.
\end{example}

We proceed to see how the $\pv V$-images of the tails of a directive
sequence are related with each other. The proof of the next lemma uses
only standard arguments about compact metric spaces.
 
\begin{lemma}
  \label{l:preservation-of-lambda}
  The equality
  $\Lambda_{\pv V}(\bsigma)=\prov V\sigma_{0,n}(\Lambda_{\pv V}(\bsigma^{(n)}))$
  holds for all $n\in\nn$.
\end{lemma}

\begin{proof}
  Let $a\in \Lambda_{\pv V}(\bsigma^{(n)})$. Then $a$ is a cluster
  point in $\Om {A_n}V$ of a sequence $(\sigma_{n,k}(a_k))_{k> n}$
  such that $a_k\in A_k$ for every $k>n$. As
  $\sigma_{0,k}(a_k)=\sigma_{0,n}(\sigma_{n,k}(a_k))$, it then follows
  by continuity of $\prov V\sigma_{0,n}$ that $\prov V\sigma_{0,n}(a)$
  is a cluster point of the sequence $(\sigma_{0,k}(a_k))_{k>n}$.
  Hence $\prov V\sigma_{0,n}(a)\in \Lambda_{\pv V}(\bsigma)$, thus
  showing the inclusion $\prov V\sigma_{0,n}(\Lambda_{\pv
    V}(\bsigma^{(n)}))\subseteq \Lambda_{\pv V}(\bsigma)$.

  Conversely, let $a\in \Lambda_{\pv V}(\bsigma)$. We may pick a
  strictly increasing sequence $(m_r)_{r\in\nn}$ of integers greater
  than $n$ and a sequence $(a_r)_{r\in\nn}$ such that $a_r\in
  A_{m_r}$, for every $r\in\nn$, and $a=\lim\sigma_{0,m_r}(a_r)$. By
  compactness of $\Om {A_n}V$, the sequence
  $(\sigma_{n,m_r}(a_r))_{r\in\nn}$ has a subsequence
  $(\sigma_{n,m_{r_s}}(a_{r_s}))_{s\in\nn}$ converging in $\Om {A_n}V$
  to a pseudoword $b$. Note that $b\in \Lambda_{\pv
    V}(\bsigma^{(n)})$. By continuity of $\prov V\sigma_{0,n}$, we
  have
  \begin{equation*}
    \prov V\sigma_{0,n}(b)
    =\lim_{s\to\infty}\prov V\sigma_{0,n}(\sigma_{n,m_{r_s}}(a_{r_s}))
    =\lim_{s\to\infty}\sigma_{0,m_{r_s}}(a_{r_s})=a.
  \end{equation*}
  This establishes the inclusion
  $\Lambda_{\pv V}(\bsigma)\subseteq
  \prov V\sigma_{0,n}(\Lambda_{\pv V}(\bsigma^{(n)}))$, finishing the proof.
\end{proof}

Combining Lemma~\ref{l:preservation-of-lambda} with
Proposition~\ref{p:properties-of-lambda}\ref{item:properties-of-lambda-0},
we obtain the following corollary.

\begin{corollary}
  \label{c:V-image-sent-to-image}
  The equality $\img_\pv{V}(\bsigma)=\prov
  V\sigma_{0,n}\bigl(\img_\pv{V}(\sigma^{(n)})\bigr)$ holds for all
  $n\in\nn$.
\end{corollary}

Applying Corollary~\ref{c:V-image-sent-to-image} to the tails of a
directive sequence $\bsigma$ allows us to consider the following
inverse system of onto continuous (restricted) homomorphisms of
pro-$\pv V$ semigroups:
\begin{equation*}
  \mathcal F=\{\prov V\sigma_{n,m}\from \img_{\pv V}(\bsigma^{(m)})\to
  \img_{\pv V}(\bsigma^{(n)})\mid m,n\in\nn,\,m\geq n\}.
\end{equation*}
Similarly, Lemmas~\ref{l:preservation-of-lambda}
and~\ref{p:lambda-is-closed} yield the following inverse system of
onto continuous functions between compact spaces:
\begin{equation*}
  \mathcal G=\{\prov V\sigma_{n,m}\from \Lambda_{\pv
    V}(\bsigma^{(m)})\to
  \Lambda_{\pv V}(\bsigma^{(n)})\mid m,n\in\nn,\,m\geq n\}.
\end{equation*}
We denote the inverse limits $\varprojlim\mathcal F$ and
$\varprojlim\mathcal G$ respectively by $\img_\pv{V}^\infty(\bsigma)$
and $\Lambda_{\pv V}^\infty(\bsigma)$. By compactness, these sets are
nonempty~\cite[Theorem 3.2.13]{Engelking:1989}, with
$\img_\pv{V}^\infty(\bsigma)$ is a pro-$\pv V$ semigroup and
$\Lambda_{\pv V}^\infty(\bsigma)$ is a closed subspace of
$\img_\pv{V}^\infty(\bsigma)$. The corresponding projections
$\img_\pv{V}^\infty(\bsigma)\to \img_\pv{V}(\bsigma^{(n)})$ are onto
continuous homomorphisms~\cite[Theorem 3.2.15]{Engelking:1989}, which
we denote by~$\prov V\sigma_{n,\infty}$, for every $n\in\nn$. Bear
also in mind that $\prov
V\sigma_{n,\infty}\bigl(\Lambda_{\pv{V}}^\infty(\bsigma)\bigr)
=\Lambda_{\pv V}(\bsigma^{(n)})$.

The next proposition is deduced from
Proposition~\ref{p:properties-of-lambda} with routine arguments.

\begin{proposition}
  \label{p:at-the-limit-properties-of-lambda}
  Let $\bsigma$ be a primitive directive sequence.
  The following properties hold:
  \begin{enumerate}
  \item $\Lambda_{\pv V}^\infty(\bsigma)$ generates a dense
    subsemigroup of
    $\img_{\pv{V}}^\infty(\bsigma)$;\label{item:at-the-limit-properties-of-lambda-0}
  \item the set $\Lambda_{\pv V}^\infty(\bsigma)$ is contained in a regular
    $\green{J}$-class of $\img_{\pv V}^\infty(\bsigma)$;\label{item:at-the-limit-properties-of-lambda-2}
  \item $\img_{\pv V}^\infty(\bsigma)=\Lambda_{\pv V}^\infty(\bsigma)\cdot \img_{\pv V}^\infty(\bsigma)=\img_{\pv V}^\infty(\bsigma)\cdot\Lambda_{\pv V}^\infty(\bsigma)$.\label{item:at-the-limit-properties-of-lambda-1}
  \end{enumerate}
\end{proposition}

\begin{proof}
  \ref{item:at-the-limit-properties-of-lambda-0} As $\prov
  V\sigma_{n,\infty}\bigl(\Lambda_{\pv{V}}^\infty(\bsigma)\bigr)
  =\Lambda_{\pv V}(\bsigma^{(n)})$, the subsemigroup of
  $\img_{\pv V}^\infty(\bsigma)$ generated by
  $\Lambda_{\pv{V}}^\infty(\bsigma)$ is mapped by $\prov
  V\sigma_{n,\infty}$ to the subsemigroup of $\img_{\pv
    V}(\bsigma^{(n)})$ generated by $\Lambda_{\pv V}(\bsigma^{(n)})$.
  Since, for every $n\in\nn$, the latter is a dense subsemigroup by
  Proposition~\ref{p:properties-of-lambda}~\ref{item:properties-of-lambda-0},
  so is the former.
  
  \ref{item:at-the-limit-properties-of-lambda-2} It is folklore, whose
  proof is an easy exercise, the fact that in an inverse limit
  $S=\lim_{i\in I}S_i$ of compact semigroups, and for all elements
  $s=(s_i)_{i\in I}$ and $t=(t_i)_{i\in I}$ of $S$, one has
  $s\green{J}t$ if and only if $s_i\green{J}t_i$ for every $i\in I$;
  and that $s$ is regular if and only if $s_i$ is regular for every
  $i\in I$ (e.g., cf.~\cite[Propositions 9.1 and
  9.3]{Rhodes&Steinberg:2001} or~\cite[Corollary
  5.6.2]{Almeida:1994a}). With this on hand, the second item follows
  immediately from Proposition~\ref{p:properties-of-lambda}.

  \ref{item:at-the-limit-properties-of-lambda-1} The inclusion of the
  two products in the set $\img_{\pv V}^\infty(\bsigma)$ follows from
  the fact that this set is a subsemigroup of~$\Om{A_0}V$. It follows
  from \ref{item:at-the-limit-properties-of-lambda-0} and compactness
  that every element of $\img_{\pv V}^\infty(\bsigma)$ belongs to
  $\Lambda_{\pv V}^\infty(\bsigma)\bigl(\img_{\pv
    V}^\infty(\bsigma)\bigr)^1$. Since $\Lambda_{\pv
    V}^\infty(\bsigma)\subseteq\Lambda_{\pv
    V}^\infty(\bsigma)\img_{\pv V}^\infty(\bsigma)$
  by~\ref{item:at-the-limit-properties-of-lambda-2}, we obtain the
  inclusion $\img_{\pv V}^\infty(\bsigma)\subseteq\Lambda_{\pv
    V}^\infty(\bsigma)\img_{\pv V}^\infty(\bsigma)$. This establishes
  the equality $\img_{\pv V}^\infty(\bsigma) =\Lambda_{\pv
    V}^\infty(\bsigma)\cdot \img_{\pv V}^\infty(\bsigma)$. The
  equality $\img_{\pv V}^\infty(\bsigma) =\img_{\pv
    V}^\infty(\bsigma)\cdot\Lambda_{\pv V}^\infty(\bsigma)$ follows by
  symmetric arguments.
\end{proof}

Denote by $J_{\pv V}^\infty(\bsigma)$ the regular $\green{J}$-class
of $\img_{\pv V}^\infty(\bsigma)$ containing the set $\Lambda_{\pv V}^\infty(\bsigma)$.

\begin{corollary}
 \label{c:at-the-limit-properties-of-lambda}
 Let $\bsigma$ be a primitive directive sequence. Then the following hold,
 for every $n,m\in\nn$, with $n\leq m$:
 \begin{enumerate}
 \item $\prov V\sigma_{n,\infty}(J_{\pv V}^\infty(\bsigma))\subseteq J_{\pv V}(\bsigma^{(n)})\cap\img_{\pv V}(\bsigma^{(n)})$.\label{item:corollary-at-the-limit-properties-of-lambda-1}
 \item $\prov V\sigma_{n,m}\Bigl(J_{\pv V}(\bsigma^{(m)})\cap\img_{\pv V}(\bsigma^{(m)})\Bigr)\subseteq J_{\pv V}(\bsigma^{(n)})\cap\img_{\pv V}(\bsigma^{(n)})$.\label{item:corollary-at-the-limit-properties-of-lambda-2}
 \end{enumerate}
\end{corollary}

\begin{proof}
  Recall that  $\prov V\sigma_{n,\infty}(\img_{\pv V}^\infty(\bsigma))=\img_{\pv V}(\bsigma^{(n)})$
  and $\prov V\sigma_{n,\infty}(\Lambda_{\pv V}^\infty(\bsigma))=\Lambda_{\pv V}(\bsigma^{(n)})$.
  On the other hand, respectively by the definition of $J_{\pv V}^\infty(\bsigma)$
  and by Proposition~\ref{p:properties-of-lambda},
  we have the inclusions $\Lambda_{\pv V}^\infty(\bsigma)\subseteq J_{\pv V}^\infty(\bsigma)$
  and $\Lambda_{\pv V}(\bsigma^{(n)})\subseteq J_{\pv V}(\bsigma^{(n)})$. Since any homomorphism sends $\green{J}$-classes into $\green{J}$-classes, this establishes
  item~\ref{item:corollary-at-the-limit-properties-of-lambda-1}.

  We also have $\prov V\sigma_{n,m}(\img_{\pv V}(\bsigma^{(m)}))=\img_{\pv V}(\bsigma^{(n)})$
  and $\prov V\sigma_{n,m}(\Lambda_{\pv V}(\bsigma)^{(m)})=\Lambda_{\pv V}(\bsigma^{(n)})$ and, by Proposition~\ref{p:properties-of-lambda},
  also $\Lambda_{\pv V}(\bsigma^{(k)})\subseteq J_{\pv V}(\bsigma^{(k)})\cap \img_{\pv V}(\bsigma^{(k)})$
  for every $k\in\nn$. Again because homomorphisms map $\green{J}$-classes into $\green{J}$-classes,
  this shows~\ref{item:corollary-at-the-limit-properties-of-lambda-2}.
\end{proof}

\begin{corollary}
  \label{c:chain-of-idempotents-in-the-inverse-limit}
  Let $\bsigma$ be a primitive directive sequence.
  There is a sequence $(e_n)_{n\in\nn}$ of idempotent pseudowords
  satisfying $e_n\in J_{\pv V}(\bsigma^{(n)})\cap\img_{\pv V}(\bsigma^{(n)})$
  and $e_n=\prov V\sigma_{n,m}(e_m)$
  for every $n,m\in\nn$ such that $n\leq m$.
\end{corollary}

\begin{proof}
  We may take an idempotent $e$ in the regular $\green{J}$-class
  $J_{\pv V}^\infty(\bsigma)$ and consider, for each $k\in\nn$, the idempotent $e_k=\prov V\sigma_{k,\infty}(e)$.
  By the definition of the inverse limit $\img_{\pv V}^\infty(\bsigma)$,
  this immediately yields the equality $e_n=\prov V\sigma_{n,m}(e_m)$
  for every $n,m\in\nn$ such that $n\leq m$.
  The remaining of the statement follows from Corollary~\ref{c:at-the-limit-properties-of-lambda}~\ref{item:corollary-at-the-limit-properties-of-lambda-2}.
\end{proof}

\section{Simplicity of profinite images of directive sequences}
\label{sec:proper-case}

Consider a directive sequence $\bsigma=(\sigma_n)_{n\in\nn}$,
with $\sigma_n\from A_{n+1}^+\to A_n^+$ for each $n\in\nn$. In this section, we investigate more systematically necessary and sufficient conditions
for $\img_{\pv V}(\bsigma)$ to be a simple profinite semigroup (cf.~Corollary~\ref{c:Im-intersection-J}). It turns out that being left or right proper is such a sufficient condition (cf.~Theorems~\ref{t:right-proper-in-an-R-class} and~\ref{t:proper-is-group}).

A \emph{limit word} of $\bsigma$ is an element of $\bigcap_{n\geq 1}\sigma_{0,n}(A_n^\zz)$. For a discussion about the significance of this notion,
see the introductory paragraphs of~\cite[Section~4]{Berthe&Steiner&Thuswaldner&Yassawi:2019}.

\begin{theorem}
  \label{t:characterization-of-image-inside-Jsigma}
  Let $\bsigma$ be a primitive directive sequence.
  Let $\pv V$ be a pseudovariety of semigroups containing $\pv{\loc Sl}$.
  The following statements are equivalent:
  \begin{enumerate}
  \item the profinite semigroup $\img_{\pv V}(\bsigma)$ is
    simple;\label{i:characterization-of-image-inside-Jsigma-1}
  \item the inclusion $\img_{\pv V}(\bsigma)\subseteq J_{\pv
      V}(\bsigma)$
    holds;\label{i:characterization-of-image-inside-Jsigma-2}
  \item all limit words of $\bsigma$ belong to $X(\bsigma)$.\label{i:characterization-of-image-inside-Jsigma-3} 
  \end{enumerate}
\end{theorem}

\begin{proof}
  $\ref{i:characterization-of-image-inside-Jsigma-1} \Leftrightarrow
  \ref{i:characterization-of-image-inside-Jsigma-2}$ This equivalence
  holds by Corollary~\ref{c:Im-intersection-J}.
   
  $\ref{i:characterization-of-image-inside-Jsigma-2} \Rightarrow
  \ref{i:characterization-of-image-inside-Jsigma-3}$ Suppose that
  $\img_{\pv V}(\bsigma)\subseteq J_{\pv V}(\bsigma)$. Let $x$ be a
  limit word of $\bsigma$. Take $k\in\nn$. We want to show that
  $x[-k,k]\in L(\bsigma)$. For each $n\in\nn$, take $x_n\in A_n^\zz$
  such that $x=\sigma_{0,n}(x_n)$. Let $w$ be a cluster point in $\Om
  {A_0}V$ of the sequence $(\sigma_{0,n}(x_n[-1,0]))_{n\in\nn}$. Since
  $\bsigma$ is primitive, the word $x[-k,k]$ is a factor of
  $\sigma_{0,n}(x_n[-1,0])$ for every sufficiently large~$n$.
  Therefore, $x[-k,k]$ is also a factor of~$w$. Note that $w\in
  \img_{\pv V}(\bsigma)$, by
  Lemma~\ref{l:V-image-as-a-set-of-cluster-points}. By the assumption
  that $\img_{\pv V}(\bsigma)\subseteq J_{\pv V}(\bsigma)$, all
  finite-length factors of $w$ belong to $L(\bsigma)$, by
  Proposition~\ref{p:mirage}. In particular, we have $x[-k,k]\in
  L(\bsigma)$, for every $k\in\nn$. This means that $x\in X(\bsigma)$.

  $\ref{i:characterization-of-image-inside-Jsigma-3} \Rightarrow
  \ref{i:characterization-of-image-inside-Jsigma-2}$ Suppose that all
  limit words of $\bsigma$ belong to $X(\bsigma)$. Let $u\in \img_{\pv
    V}(\bsigma)$. By Lemma~\ref{l:V-image-as-a-set-of-cluster-points},
  we know that there is a strictly increasing sequence
  $(n_k)_{k\in\nn}$ of positive integers and a sequence
  $(u_k)_{k\in\nn}$ of words, with $u_k\in (A_{n_k})^+$, such that
  $\sigma_{0,n_k}(u_k)\to u$. In particular, the pseudoword $u$ has
  infinite length.
   
  Let $v$ be a finite-length factor of $u$. We claim that $v\in
  L(\bsigma)$. Note that the set $(\Om AV)^1v(\Om AV)^1$ is clopen, as
  $\pv V$ contains $\pv{\loc Sl}$. Hence, taking subsequences, we may
  suppose that $v$ is a factor of $\sigma_{0,n_k}(u_k)$ for every
  $k\in\nn$. Since $\bsigma$ is primitive, we may further assume that
  all words in $\sigma_{0,n_k}(A_{n_k})$ have length greater than that
  of $v$. Therefore, for each $k\in\nn$, we may take letters
  $a_k,b_k\in A_{n_k}$ such that $v$ is a factor of
  $\sigma_{0,n_k}(a_kb_k)$. If $v$ is a factor of
  $\sigma_{0,n_k}(a_k)$ or of $\sigma_{0,n_k}(b_k)$ for some $k$, then
  $v\in L(\bsigma)$ and the claim is proved. Therefore, we may as well
  assume that for every $k\in\nn$ there is a factorization $v=s_kp_k$
  such that $s_k$ is a nonempty suffix of $\sigma_{0,n_k}(a_k)$ and
  $p_k$ is a nonempty prefix of $\sigma_{0,n_k}(b_k)$. In fact, again
  by taking subsequences, we are reduced to the case where $(s_k,p_k)$
  has constant value $(s,p)$. Repeating the process of taking
  subsequences, we may as well suppose that the sequence
  $(\sigma_{0,n_k}(a_k),\sigma_{0,n_k}(b_k))_{k\geq 1}$ converges in
  $\Om {A_0}V\times \Om {A_0}V$ to some pair $(\alpha,\beta)$ of
  elements of $\Lambda_{\pv V}(\bsigma)$. Note that $s$ is a
  finite-length suffix of $\alpha$ and $p$ is a finite-length prefix
  of $\beta$. Consider the element $x$ of $A_0^{\zz}$ such that, for
  every positive integer $n$, the words $x{[-n,-1]}$ and
  $x{{[}0,n{)}}$ are respectively the suffix of length of $n$ of
  $\alpha$ and the prefix of length $n$ of $\beta$. Then we have
  $v=x{{[}-|s|,|p|{)}}$. Therefore, to show that $v\in L(\bsigma)$, it
  suffices to show that $x$ is a limit word of $\bsigma$.
   
  Let $r\in\nn$. As $\alpha,\beta\in \img_{\pv V}(\bsigma)$, we may
  take infinite-length pseudowords $\alpha',\beta'\in \Om {A_r}V$ such
  that $\alpha=\prov V\sigma_{0,r}(\alpha')$ and $\beta=\prov
  V\sigma_{0,r}(\beta')$. Let $y$ be the element of $A_r^{\zz}$ such
  that, for every positive integer $n$, the words $y{[-n,-1]}$ and
  $y{{[}0,n{)}}$ are respectively the suffix of length of $n$ of
  $\alpha'$ and the prefix of length $n$ of $\beta'$. Then
  $\sigma_{0,r}(y{[-n,-1]})$ and $\sigma_{0,r}(y{{[}0,n{)}})$ are
  respectively a suffix of length at least $n$ of $\alpha$ and a
  prefix of length at least $n$ of $\beta$. As this holds for very
  $n\in\nn$, it follows that $x=\sigma_{0,r}(y)$. Since $r$ is an
  arbitrary element of $\nn$, we conclude that
  $x\in\bigcap_{r\in\zz}\sigma_{0,r}(A_r^\zz)$, that is to say, that
  $x$ is a limit word of $\bsigma$. By assumption, we therefore have
  $x\in X(\bsigma)$, establishing the claim that $v\in L(\bsigma)$.

  Since $v$ is an arbitrary finite-length factor of $u$, by
  Proposition~\ref{p:mirage} we deduce that $u\in J_{\pv V}(\bsigma)$.
  This establishes the inclusion $\img_{\pv V}(\bsigma)\subseteq
  J_{\pv V}(\bsigma)$.
\end{proof}
 
\begin{remark}
  \label{r:downward-inheritance-Im-in-J}
  In view of item~\ref{i:characterization-of-image-inside-Jsigma-3} in
  Theorem~\ref{t:characterization-of-image-inside-Jsigma}, the choice
  of $\pv{V}$ plays no role in the statement of the theorem. More
  precisely, one has $ \img_{\pv S}(\bsigma)\subseteq J_{\pv
    S}(\bsigma)$ if and only if $\img_\pv V(\bsigma)\subseteq J_\pv
  V(\bsigma)$, when $\pv V$ is a pseudovariety of semigroups
  containing~$\pv{\loc Sl}$. This equivalence also follows directly
  from Corollary~\ref{c:mirage} and
  Proposition~\ref{p:projecting-profinite-image}.
\end{remark}
 
Denote by $\fac_n(w)$ the set of factors of length $n$ of a pseudoword~$w$, for each~$n\in\nn$. 

\begin{definition}
  \label{d:stable-directive-sequences}
  We say that the directive sequence $\bsigma$ is \emph{stable} if for every $n\in\nn$
  and every $a,b\in A_{n+1}$, the inclusion $\fac_2(\sigma_n(ab))\subseteq L(\bsigma^{(n)})$ holds.
\end{definition}

\begin{remark}
  \label{r:stable-contractions}
  Thanks to Lemma~\ref{l:contraction-does-not-change-the-language}, we
  know that if $\bsigma$ is a directive sequence such that
  $A_n\subseteq \fac(\sigma_n(A_{n+1}))$ for every $n\in\nn$ (which
  happens if $\bsigma$ is primitive), then the following equivalence
  holds: the directive sequence $\bsigma$ has a stable contraction if
  and only if there exists a strictly increasing sequence
  $(n_k)_{k\in\nn}$ such that $n_0=0$ and, for all $k\geq 1$, and all
  $u\in A_{n_{k+1}}^2$ (equivalently, all $u\in A_{n_{k+1}}^+$), every
  factor of length two of $\sigma_{n_{k},n_{k+1}}(u)$ belongs to
  $L(\bsigma^{(n_k)})$.
\end{remark}

\begin{example}
  \label{eg:Prouhet-Thue-Morse-stable}
  Set $A=\{\ltr a,\ltr b\}$.
  Consider the Prouhet-Thue-Morse substitution
  \begin{equation*}
    \tau\from \ltr{a}\mapsto \ltr{ab},\ \ltr{b}\mapsto \ltr{ba}.
  \end{equation*}
  Then we have $\tau(A^2)\subseteq L(\tau)$, and
  so the constant primitive directive sequence $\btau=(\tau,\tau,\tau,\ldots)$ is stable.
\end{example}

Our next step is to see how the property of having a stable contraction is preserved by contractions.

\begin{proposition}
  \label{p:contraction-of-contraction-stable}
  Let $\bsigma$ be a directive sequence
  such that $A_n\subseteq \fac(\sigma_n(A_{n+1}))$ for every $n\in\nn$.
  Assume that $\bsigma$ has a stable contraction. Then every contraction
  of $\bsigma$ has a stable contraction.
\end{proposition}

\begin{proof}
  Let $\btau$ be a stable contraction of $\bsigma$,
  and let $\btheta$ be a contraction of $\bsigma$.
  Take increasing sequences $(n_k)_{n\in\nn}$
  and $(m_k)_{k\in\nn}$
  of integers, with $n_0=m_0=0$,
  such that $\btau=(\sigma_{n_k,n_{k+1}})_{k\in\nn}$
  and $\btheta=(\sigma_{m_k,m_{k+1}})_{k\in\nn}$.
  The sequence $(m_k)_{k\in\nn}$
  has a subsequence $(m_{k_i})_{i\in\nn}$
  such that, for every $i\in\nn$, we there is $j(i)\in\nn$
  satisfying $m_{k_i}<n_{j(i)}< n_{j(i)+1}<m_{k_{i+1}}$.
  For such a choice of a subsequence,
  we have therefore the factorization
  \begin{equation*}
    \sigma_{m_{k_i},m_{k_{i+1}}}=\sigma_{m_{k_i},n_{j(i)}}\circ \sigma_{n_{j(i)},n_{j(i)+1}}\circ \sigma_{n_{j(i)+1},m_{k_{i+1}}}
\end{equation*}
for every $i\in\nn$. Set $\zeta_i=\sigma_{m_{k_i},m_{k_i+1}}$. Note that $\bzeta=(\zeta_i)_{i\in\nn}$
is a contraction of $\btheta$. We claim that $\bzeta$ is stable.
Fix $i\in\nn$, and let $a,b\in A_{1+m_{k_i+1}}$. We want to show that $\fac_2(\zeta_i(ab))\subseteq L(\bzeta^{(i)})$.
To do that, observe that
\begin{equation*}
  \zeta_i(ab)=\sigma_{m_{k_i},n_{j(i)}}(\sigma_{n_{j(i)},n_{j(i)+1}}(w))
\end{equation*}
for some word $w$.
Since $\sigma_{n_{j(i)},n_{j(i)+1}}=\tau_{j(i)}$
and $\btau$ is stable, we know that
\begin{equation*}
  \fac_2(\tau_{j(i)}(w))\subseteq L(\btau^{(j(i))}).
\end{equation*}
Bear in mind that, by Lemma~\ref{l:contraction-does-not-change-the-language}, the equality  $L(\btau^{(j(i))})=L(\bsigma^{(n_{j(i)})})$ holds.
Since $\zeta_i(ab)=\sigma_{m_{k_i},n_{j(i)}}(\tau_{j(i)}(w))$,
every factor of length two of $\zeta_i(ab)$
is a factor of $\sigma_{m_{k_i},n_{j(i)}}(u)$
for some factor $u$ of length two of $\tau_{j(i)}(w)$. We already saw that $u\in L(\bsigma^{(n_{j(i)})})$,
and so, by Lemma~\ref{l:contraction-does-not-change-the-language}, we have
\begin{equation*}
\sigma_{m_{k_i},n_{j(i)}}(L(\bsigma^{(n_{j(i)})})\subseteq L(\bsigma^{(m_{k_i})})=L(\btheta^{(k_i)})=L(\bzeta^{(i)}). 
\end{equation*}
This shows that $\fac_2(\zeta_i(ab))\subseteq L(\bzeta^{(i)})$.
This in turn establishes that $\bzeta$ is a stable contraction of $\btheta$.  
\end{proof}

Next we see that, under a very mild condition, satisfied by primitive directive sequences,
all left or right proper directive sequences are stable. Note that the Prouhet-Thue-Morse substitution provides
an example of a primitive stable directive sequence that is neither left proper nor right proper (Example~\ref{eg:Prouhet-Thue-Morse-stable}).

\begin{proposition}
  \label{p:left-proper-implies-stable}
  Let $\bsigma$ be a directive sequence
  such that $A_n\subseteq L(X(\bsigma^{(n)}))$
  for every $n\in\nn$.
  If $\bsigma$ is left proper or right proper, then it is stable.
 \end{proposition}

 \begin{proof}
   By symmetry, it suffices to consider the case where $\bsigma$ is left proper.
   Take $a,b\in A_{n+1}$.
   Let $w\in\fac_2(\sigma_n(ab))$.
   We wish to show that $w\in L(\bsigma^{(n)})$.
   Denote by $c$ the first letter of $\sigma_n(b)$.
   Then $w$ is a factor of $\sigma_n(a)$, or of $\sigma_n(b)$, or
   is a suffix of $\sigma_n(a)c$.
   Note that, by the definition of $L(\bsigma^{(n)})$,
   both words $\sigma_n(a)$ and $\sigma_n(b)$ belong to $L(\bsigma^{(n)})$.
   Therefore, to show that $w\in L(\bsigma^{(n)})$, it suffices to consider the case where
   $w$ is a suffix of $\sigma_n(a)c$.
   Since $A_{n+1}\subseteq L(X(\bsigma^{(n+1)}))$, we may choose a letter $d\in A_{n+1}$ such that $ad\in L(\bsigma^{(n+1)})$.
   Moreover, by Lemma~\ref{l:changing-levels-shift}, we have $\sigma_n(ad)\in L(X(\bsigma^{(n)}))$.
   Since $\bsigma$ is left proper, the words $\sigma_n(d)$ and $\sigma_n(b)$ have the same first letter,
   and so $\sigma_n(a)c$ is a prefix of $\sigma_n(ad)$.
   Since we are assuming that $w$ is suffix of $\sigma_n(a)c$, it follows that $w\in L(\bsigma^{(n)})$.
   This concludes the proof that $\fac_2(\sigma_n(ab))\subseteq L(\bsigma^{(n)})$,
   thus showing that $\bsigma$ is stable.
 \end{proof}

The following theorem improves a similar result that
the second author obtained for the special case of primitive substitutive directive sequences~\cite[Lemma~3.12]{ACosta:2023}.

\begin{theorem}
  \label{t:contraction-stable}
  Let $\bsigma$ be a primitive directive sequence.
  Let \pv V be a pseudovariety of semigroups containing~\pv{\loc Sl}.
  The following statements are equivalent:
    \begin{enumerate}
    \item $\bsigma$ has a stable contraction; \label{i:contraction-stable-1}
    \item $\img_\pv{V}(\bsigma^{(k)})$ is a simple semigroup, for every $k\geq 0$;\label{i:contraction-stable-2}
    \item $\img_\pv{V}^\infty(\bsigma)$ is a simple semigroup.\label{i:contraction-stable-3}
    \end{enumerate}
\end{theorem}

\begin{proof}  
    $\ref{i:contraction-stable-1}
    \Rightarrow
    \ref{i:contraction-stable-2}$
    If $\bsigma$ has a stable contraction, then so does $\bsigma^{(k)}$
    for every $k\geq 0$. Therefore, it suffices to show that the inclusion $\img_{\pv V}(\bsigma)\subseteq J_{\pv V}(\bsigma)$ holds,
    which, by Corollary~\ref{c:Im-intersection-J},
    means that the semigroup $\img_{\pv V}(\bsigma)$ is simple.
    Moreover, since taking a contraction leaves the $\pv V$-image unchanged (Lemma~\ref{l:V-image-contraction}), we may as well suppose that
    $\bsigma$ is stable.
    
    Let $u\in \img_{\pv V}(\bsigma)$.
    By Lemma~\ref{l:V-image-as-a-set-of-cluster-points},
    we may take a strictly increasing sequence $(n_k)_{k\in\nn}$ of positive integers and a sequence $(u_k)_{k\in\nn}$ of words, with $u_k\in (A_{n_k})^+$,
   such that $\sigma_{0,n_k}(u_k)\to u$.
   If $|u_k|\to 1$, then $u\in\Lambda_{\pv V}(\bsigma)$, and so $u\in J_{\pv V}(\bsigma)$ by Proposition~\ref{p:properties-of-lambda}.
   Therefore, we may as well assume that $|u_k|\geq 2$ for every $k\in\nn$.
   Let $w$ be a finite-length factor of $u=\lim \sigma_{0,n_k}(u_k)$.
   Since $\pv V$ contains \pv{\loc Sl},
   by taking subsequences, we may further assume that
   $w$ is a factor of $\sigma_{0,n_k}(u_k)$ for every $k\in\nn$.
   Because $\bsigma$ is primitive, there is $k_0\in\nn$ such that, for every $k>k_0$ and every letter $a$ that is a factor of $u_k$,
   we have $|\sigma_{0,n_k}(a)|>|w|$.
   Hence, for every $k>k_0$, and because $w$ is a factor of $\sigma_{0,n_{k+1}}(u_{k+1})$,
   there is a factor $v_k$ of length two of the word $\sigma_{n_k,n_{k+1}}(u_{k+1})$
   such that $w$ is a factor of $\sigma_{0,n_k}(v_k)$.
   Since $\bsigma$ is stable, we have
   $\fac_2(\sigma_{n_k,n_{k+1}}(u_{k+1}))\subseteq L(\bsigma^{(n_k)})$, thus $v_k\in  L(\bsigma^{(n_k)})$.
   It follows that $\fac(\sigma_{0,n_k}(v_k))\subseteq L(\bsigma)$ (cf.~Lemma~\ref{l:contraction-does-not-change-the-language}), whence $w\in L(\bsigma)$. Since $w$ is an arbitrary finite-length factor of $u$, we conclude that $u\in J_{\pv V}(\bsigma)$
   by Proposition~\ref{p:mirage}.
   This establishes the inclusion $\img_{\pv V}(\bsigma)\subseteq J_{\pv V}(\bsigma)$.

   $\ref{i:contraction-stable-2}
   \Rightarrow
   \ref{i:contraction-stable-1}$
   Suppose that property \ref{i:contraction-stable-2} holds
   for $\bsigma$. Then the same property holds for $\bsigma^{(n)}$
   for every $n\in\nn$, as $(\bsigma^{(n)})^{(k)}=\bsigma^{(n+k)}$
   for every $n,k\in\nn$. Therefore, to
   establish the implication
   $\ref{i:contraction-stable-2}
   \Rightarrow
   \ref{i:contraction-stable-1}$,
   it suffices to establish
   the inclusion $\fac_2(\sigma_{0,n}(A_{n}^2))\subseteq L(\bsigma)$
   for some positive integer~$n$.
   
   Suppose, on the contrary, that we have
   $\fac_2(\sigma_{0,n}(A_{n}^2))\nsubseteq L(\bsigma)$ for every
   positive integer $n$. Then, for each $n\geq 1$, we may take letters
   $a_n,b_n\in A_n$ and a word $w_n\in A_0^2\setminus L(\bsigma)$ such
   that $\sigma_{0,n}(a_nb_n)\in A^*w_nA^*$. Let $(\alpha,\beta,w)$ be
   a cluster point in $(\Om {A_0}V)^3$ of the sequence
   $(\sigma_{0,n}(a_n),\sigma_{0,n}(b_n),w_n)_{n\geq 1}$. Note that
   $w$ is a factor of $\alpha\beta$. Moreover, since $A_0$ is a finite
   alphabet, one must have $w=w_m$ for infinitely many integers $m$.
   Therefore, $\alpha\beta$ has a finite-length factor (namely $w$)
   not in $L(\bsigma)$. This implies that $\alpha\beta\notin J_{\pv
     V}(\bsigma)$ by Proposition~\ref{p:mirage}. On the other hand, we
   have $\alpha,\beta\in J_{\pv V}(\bsigma)\cap \img_{\pv V}(\bsigma)$
   by Proposition~\ref{p:properties-of-lambda}. It follows that
   $\alpha\beta\in \img_{\pv V}(\bsigma)\setminus J_{\pv V}(\bsigma)$,
   which contradicts the assumption that $\img_{\pv V}(\bsigma)$ is
   simple. To avoid the contradiction, we indeed must have
   $\fac_2(\sigma_{0,n}(A_{n}^2))\subseteq L(\bsigma)$ for some
   positive integer $n$.

   $\ref{i:contraction-stable-3}
   \Rightarrow
   \ref{i:contraction-stable-2}$
   Immediately after defining the inverse limit
   $\img_\pv{V}^\infty(\bsigma)=\varprojlim
   \img_\pv{V}(\bsigma^{(k)})$, we observed that
   $\img_\pv{V}(\bsigma^{(k)})$ is a homomorphic image of
   $\img_\pv{V}^\infty(\bsigma)$, for every natural number $k$. This
   gives the implication, as the homomorphic image of a simple
   semigroup is also simple.

   $\ref{i:contraction-stable-2}
   \Rightarrow
   \ref{i:contraction-stable-3}$
   It is well known that properties that can be
   defined by so-called \emph{pseudoidentities} are preserved by
   taking products and subsemigroups, whence by inverse limits
   \cite{Almeida:2003cshort}.
   This is the case of simplicity for 
   profinite semigroups, which is defined by the pseudoidentity
   $(xy)^\omega x=x$. Alternatively, see~\cite[Corollary
   9.2]{Rhodes&Steinberg:2001}.
 \end{proof}

 A semigroup $S$ is said to be \emph{right simple} if the relation
 $\green{R}$ on $S$ is the universal relation; likewise, there is a
 dual notion of \emph{left simplicity} where $\green{R}$ is replaced
 by $\green{L}$ (cf.~\cite[Section~A.1]{Rhodes&Steinberg:2009qt}).
 Note that both of these properties are strictly stronger than
 simplicity.

\begin{theorem}
  \label{t:right-proper-in-an-R-class}
  Let $\bsigma$ be a primitive directive sequence.
  Let \pv V be a pseudovariety of semigroups containing~\pv{\loc Sl}.
  The following statements are equivalent:
  \begin{enumerate}
  \item $\bsigma$ has a left proper
    contraction; \label{i:right-proper-in-an-R-class-1}
  \item $\img_\pv{V}(\bsigma^{(k)})$ is a right simple semigroup for
    every $k\geq 0$;\label{i:right-proper-in-an-R-class-2}
  \item $\img_\pv{V}^\infty(\bsigma)$ is a right simple
    semigroup.\label{i:right-proper-in-an-R-class-3}
  \end{enumerate}    
\end{theorem}

Before establishing Theorem~\ref{t:right-proper-in-an-R-class}, we
state the following lemma used in its proof.

\begin{lemma}
  \label{l:place-of-finite-factor-in-ufv}
  Let $\pv V$ be a pseudovariety of finite semigroups that is
  contained in $\pv{\loc{Sl}}$. Consider a finite alphabet $A$. Let
  $u,v\in\Om AV$. If $x$ is a finite-length factor of $uv$, then
  either $x$ is a factor of $u$, or of $v$, or $x=sp$ for some suffix
  $s$ of $u$ and some prefix $p$ of $v$. In particular, if $x$ is a
  finite factor of $uwv$ and $w\in\Om AV\setminus A^+$, then $x$ is a
  factor of $uw$ or of $wv$.
\end{lemma}

This lemma, whose proof is an easy exercise, is subsumed into
Lemma~8.2 of the paper~\cite{Almeida&Volkov:2006}. The statement
in~\cite{Almeida&Volkov:2006} concerns the pseudovariety $\pv S$ of
all finite semigroups, but the proof given there holds for all
pseudovarieties containing $\pv{\loc{Sl}}$.

\begin{proof}[Proof of Theorem~\ref{t:right-proper-in-an-R-class}]
   $\ref{i:right-proper-in-an-R-class-2}
   \Rightarrow
   \ref{i:right-proper-in-an-R-class-3}$
   As in the end of the proof of Theorem~\ref{t:contraction-stable},
   it suffices to observe that right simplicity in  finite semigroups
   is defined by the pseudoidentity $x^\omega y=y$.

   $\ref{i:right-proper-in-an-R-class-3}
   \Rightarrow
   \ref{i:right-proper-in-an-R-class-2}$
   Every homomorphic image of a right simple semigroup is right simple.

   $\ref{i:right-proper-in-an-R-class-1}
   \Rightarrow
   \ref{i:right-proper-in-an-R-class-2}$
   Note that $\bsigma$ has a left proper contraction if and only if
   $\bsigma^{(k)}$ has a left proper contraction, for every [some] $k\geq 0$.
   Therefore, it suffices to show that $\img_\pv{V}(\bsigma)$
   is right simple.
   We may as well assume that $\sigma_n$ is left proper
   for every $n\in\nn$, thanks to Lemma~\ref{l:V-image-contraction}.
   
   For each $n\in\nn$, let $b_n\in A_n$ be such that
   $\sigma_n(A_{n+1})\subseteq b_nA_n^*$. By compactness of $\Om
   {A_0}V$, we may pick a strictly increasing sequence
   $(n_k)_{k\in\nn}$ of positive integers such that
   $(\sigma_{0,n_k}(b_{n_k}))_{k\in\nn}$ converges to some pseudoword
   $\beta$ of $\Om {A_{0}}V$. Note that $\beta\in\Lambda_{\pv
     V}(\bsigma)$. Hence, $\beta$ is a regular element of the
   semigroup $\img_{\pv V}(\bsigma)$ by
   Proposition~\ref{p:properties-of-lambda}, and so we may select an
   idempotent $e$ in the $\green{R}$-class of $\beta$ in $\img_{\pv
     V}(\bsigma)$. In particular, the equality $\beta =e\beta$ holds.

   We claim that
   \begin{equation}
     \label{eq:right-proper-in-an-R-class-1}
     \forall u\in \img_{\pv V}(\bsigma),\quad u=eu.
   \end{equation}
   Let $u\in \img_{\pv V}(\bsigma)$.
   By Lemma~\ref{l:V-image-as-a-set-of-cluster-points},
   there is a strictly increasing sequence $(m_k)_{k\in\nn}$
   of positive integers
   such that $u=\lim\sigma_{0,m_k}(u_k)$
   for some sequence $(u_k)_{k\in\nn}$ of words.
   By taking a subsequence of $(m_k)_{k\in\nn}$,
   we may as well assume that $m_k>n_k$
   for every $k\in\nn$.
   For each $k\in\nn$, let $s_k\in (A_{n_k})^*$ be a word such that
   $\sigma_{n_k,m_k}(u_k)=b_{n_k}s_k$.
   Further taking subsequences, we may assume that
   the sequence $(\sigma_{0,n_k}(s_{k}))_k$ converges to some pseudoword
   $s\in\Om {A_0}V$.
   We then have
   \begin{equation*}
     u=\lim\sigma_{0,m_k}(u_k)
     =\lim\sigma_{0,n_k}(b_{n_k})\sigma_{0,n_k}(s_{k})
     =\beta s=e\beta s=eu.
   \end{equation*}
   This establishes the claim that
   \eqref{eq:right-proper-in-an-R-class-1} holds.

   As we are assuming that $\bsigma$ is left proper,
   we get from Proposition~\ref{p:left-proper-implies-stable},
   that $\bsigma$ is stable. Hence, by Theorem~\ref{t:contraction-stable},
    $\img_{\pv V}(\bsigma)$ is a simple semigroup. Taking into
   account~\eqref{eq:right-proper-in-an-R-class-1}, it now follows
   from the stability of the profinite simple semigroup $\img_{\pv
     V}(\bsigma)$ that $u$ is in the $\green{R}$-class of $\img_{\pv
     V}(\bsigma)$ containing $e$. Since $u$ is an arbitrary element of
   $\img_{\pv V}(\bsigma)$, this establishes that $\img_{\pv
     V}(\bsigma)$ is a right simple semigroup.

   $\ref{i:right-proper-in-an-R-class-2}
   \Rightarrow
   \ref{i:right-proper-in-an-R-class-1}$
   Since $\bsigma$ has a left proper contraction if and only if every
   tail of $\bsigma$ has a left proper contraction, it suffices to show
   that there is a positive integer $n$ such that $\sigma_{0,n}$ is left proper. 
   Suppose that, on the contrary,
   for every positive integer~$n$ the homomorphism $\sigma_{0,n}$ is
   not left proper.
   Then, for each $n\geq 1$, there are letters $a_n,b_n\in A_0$,
   with $a_n\neq b_n$,
   and $c_n,d_n\in A_n$ 
   such that  the words $\sigma_{0,n}(c_n)$
   and $\sigma_{0,n}(d_n)$ respectively start with $a_n$ and $b_n$.
   Let $(\gamma,\delta,a,b)$ be a cluster
   point in $(\Om {A_0}V)^4$ of the sequence
   $(\sigma_{0,n}(c_n),\sigma_{0,n}(d_n),a_n,b_n)_{n\geq 1}$. 
   Since $A_o$~is finite, we have $(a,b)=(a_m,b_m)$ for infinitely
   many integers $m$, and so $a$ and $b$ 
   are distinct letters of $A_0$. Moreover, the pseudowords $\gamma$
   and $\delta$ start with $a$ and $b$, respectively.
   If two elements of $\Om {A_0}V$ are $\green{R}$-equivalent, then
   they have the same finite-length prefixes. Hence, $\gamma$ and
   $\delta$ are not 
   $\green{R}$-equivalent. But this contradicts
   our assumption that $\img_{\pv V}(\bsigma)$
   is right simple. Therefore, there is indeed a positive integer $n$
   such that $\sigma_{0,n}$ is left proper.
 \end{proof}

 We end this section with the analog of
 Theorem~\ref{t:right-proper-in-an-R-class} for proper directive
 sequences.

\begin{theorem}
  \label{t:proper-is-group}
  Let $\bsigma$ be a primitive directive sequence.
  Let \pv V be a pseudovariety of semigroups containing~\pv{\loc Sl}.
  The following statements are equivalent:
  \begin{enumerate}
  \item $\bsigma$ has a proper contraction; \label{i:proper-is-group-1}
  \item $\img_\pv{V}(\bsigma^{(k)})$ is a group for every $k\geq
    0$;\label{i:proper-is-group-2}
  \item $\img_\pv{V}^\infty(\bsigma)$ is a group.\label{i:proper-is-group-3}
  \end{enumerate}    
\end{theorem}

\begin{proof}
  $\ref{i:proper-is-group-1}
  \Rightarrow
  \ref{i:proper-is-group-2}$
  This implication follows from
  Theorem~\ref{t:right-proper-in-an-R-class} and its dual, because a
  semigroup is a group if and only if it is both left and right
  simple~\cite[Lemma~A.3.1]{Rhodes&Steinberg:2009qt}.

  $\ref{i:proper-is-group-2}
  \Rightarrow
  \ref{i:proper-is-group-3}$
  Any inverse limit of profinite groups is a profinite group.

  $\ref{i:proper-is-group-3} \Rightarrow \ref{i:proper-is-group-1}$ By
  Theorem~\ref{t:right-proper-in-an-R-class} and its dual, $\bsigma$
  has a left proper contraction and a right proper contraction, which
  means that it has a proper contraction: indeed, for any $n\in\nn$,
  provided $k\geq n$ is sufficiently large, there will be
  $r,s,r',s'\in\nn$ with $n\leq r<s\leq k$ and $n\leq r'<s'\leq k$
  such that $\sigma_{r,s}$ is left proper and $\sigma_{r',s'}$ is
  right proper, hence $\sigma_{n,k}$ is proper.
\end{proof}

\section{Profinite images of bounded directive sequences}
\label{sec:case-bound-direct}

Let us say that the directive sequence $\bsigma =
(\sigma_n)_{n\in\nn}$, with $\sigma_n\from A_{n+1}^+\to A_n^+$, is
\emph{bounded} when the set $\{ A_n : n\in\nn\}$ of its alphabets is
finite.

\begin{remark}
  \label{r:bounded-versus-finite-alphabet-rank}
  A directive sequence has finite alphabet rank if and only if it has
  some contraction that is, up to relabeling of its alphabets,
  bounded. Moreover, if $\bsigma'$ is a contraction of $\bsigma$, then
  the relabeled directive sequence $\bsigma''$ obtained from
  $\bsigma'$ may be chosen, by not relabeling $A_0$, such that
  $X(\bsigma)=X(\bsigma'')$ and $\img_{\pv V}(\bsigma)=\img_{\pv
    V}(\bsigma'')$ for every pseudovariety of semigroups $\pv V$
  containing $\pv N$.
\end{remark}

For technical reasons, related with the convenience of using
finite-vertex profinite categories, we mostly prefer to work directly
with bounded directive sequences, although they have the same
expressive power of directive sequences with finite alphabet rank, as
seen in the previous remark.

A way of thinking about the directive sequence $\bsigma$ is to
visualize it as a left-infinite path
\begin{equation}\label{eq:visualize-directive-sequence}
  A_0^+
  \xleftarrow{\sigma_0}A_1^+
  \xleftarrow{\sigma_1}A_2^+
  \xleftarrow{\sigma_2}A_3^+
  \xleftarrow{\sigma_3}\cdots
\end{equation}
over the graph $\Gamma(\bsigma)$ whose vertices are the free
semigroups $A_n^+$ and where the arrows from $A_k^+$ to $A_l^+$ are
the homomorphisms from $A_k^+$ to $A_l^+$. Note that $\bsigma$ being
bounded means that $\Gamma(\bsigma)$ has a finite number of vertices.

From hereon, let $\pv V$ be a pseudovariety of finite semigroups
containing the pseudovariety $\pv N$ of finite nilpotent semigroups.
Consider the following set of finitely generated profinite semigroups:
\begin{equation*}
  \mathcal{F}_{\pv V}({\bsigma}) = \{ \Om {A_n}V : n\in\nn\}.
\end{equation*}
Let $\catsig_{\pv V}(\bsigma)$ denote the category
$\catpro[\mathcal{F}_{\pv V}(\bsigma)]$, consisting of continuous
homomorphisms between elements of $\mathcal{F}_{\pv V}(\bsigma)$.
Closely associated to the left-infinite
path~\eqref{eq:visualize-directive-sequence} in $\Gamma(\bsigma)$, we
also have the following left-infinite path
\begin{equation}\label{eq:visualize-directive-sequence-profinite-level}
  \Om {A_0}V
  \xleftarrow{\prov V\sigma_0}\Om {A_1}V
  \xleftarrow{\prov V\sigma_1}\Om {A_2}V
  \xleftarrow{\prov V\sigma_2}\Om {A_3}V
  \xleftarrow{\prov V\sigma_3}\cdots,
\end{equation}
which is a path in the graph $\catsig_{\pv V}(\bsigma)$.

The set $\mathcal{F}_{\pv V}(\bsigma)$ is finite precisely
when~$\bsigma$ is bounded. Therefore, assuming that $\bsigma$ is
bounded, the category $\catsig_{\pv V}(\bsigma)$ is a finite-vertex
profinite category
(cf.~Proposition~\ref{p:hom-is-a-profinite-category}).

Let us say that a continuous homomorphism $\psi\from\Om AV\to\Om BV$
is \emph{primitive} when every element of $B$ is a factor of every
element of $\psi(A)$. In particular, if $\varphi\from A^+\to B^+$ is a
primitive substitution, then its extension $\prov V\varphi\from \Om
AV\to\Om BV$ is a primitive homomorphism.

\begin{lemma}
 \label{l:space-of-primitive-endomorphisms}
 Let $\bsigma$ be a bounded directive sequence. The set of primitive
 homomorphisms between elements of $\mathcal{F}_{\pv V}({\bsigma})$ is
 a closed subspace of $\catsig_{\pv V}(\bsigma)$.
\end{lemma}

\begin{proof}
  Let $(\varphi_i)_{i\in I}$ be a net of primitive homomorphisms
  between elements of $\mathcal{F}_{\pv V}({\bsigma})$ converging in
  $\catsig_{\pv V}(\bsigma)$ to a homomorphism $\varphi$ from $\Om AV$
  to $\Om BV$. Since the space of vertices in the category
  $\catsig_{\pv V}(\bsigma)$ is a finite discrete space, we may assume
  that $\varphi_i$ is always a homomorphism from $\Om AV$ to $\Om BV$.
  Let $a\in A$ and $b\in B$. As $\varphi_i$ is primitive, we have
  $\varphi_i(a)\leq_{\green{J}} b$ for every $i\in I$. Since
  $\leq_{\green{J}}$ is a closed relation in $\Om BV$ and we are
  dealing with the pointwise topology of $\catsig_{\pv V}(\bsigma)$,
  we conclude that $\varphi(a)\leq_{\green{J}} b$ for every $a\in A$
  and $b\in B$. This means that $\varphi$ is primitive, thereby
  concluding the proof.
\end{proof}

\begin{definition}
  \label{d:compressions}
  Let $\bsigma= (\sigma_n)_{n\in\nn}$ be a bounded directive sequence,
  and let $\pv V$ be a pseudovariety of semigroups containing $\pv N$.
  A\/ \emph{$\pv V$-compression of $\bsigma$} is a cluster point of
  the sequence $\bigl(\sigma_{0,n}^{\pv V}\bigr)_{n\in\nn}$, in the
  profinite category $\catsig_{\pv V}(\bsigma)$.
\end{definition}

A $\pv V$-compression $\xi$ of a bounded directive sequence $\bsigma$
must be a continuous homomorphism from $\Om {A_k}V$ to $\Om {A_0}V$,
for some $k\geq 0$. If $\bsigma$ is primitive, then $\xi$ is
primitive, by Lemma~\ref{l:space-of-primitive-endomorphisms}.

\begin{example}
  \label{e:omega-power-V-compression}
  Let $\sigma\from A^+\to A^+$ be a substitution. 
  Consider the constant directive sequence $\bsigma=(\sigma,\sigma,\ldots)$.
  Then $(\prov V\sigma)^\omega$ is a $\pv V$-compression
  of $\bsigma$.
\end{example}

The next theorem says, in particular, that when the directive sequence
$\bsigma$ is bounded primitive, the profinite semigroup $\img_{\pv
  V}(\bsigma)$ is generated by elements of $J_{\pv V}(\bsigma)$. A
similar result, concerning primitive directive sequences of
substitutions over a constant alphabet, appeared in earlier work by
the first author~\cite[Theorem~3.7]{Almeida:2003a}.

\begin{theorem}
  \label{t:letters-mapped-to-Js}
  Let $\xi\from\Om BV\to\Om{A_0}V$ be a $\pv V$-compression of a
  bounded directive sequence $\bsigma$. The equality
  $\img(\xi)=\img_{\pv V}(\bsigma)$ holds. Moreover, the inclusion
  $\xi(B)\subseteq \Lambda_{\pv V}(\bsigma)$ holds.
\end{theorem}

\begin{proof}
  We may take a subnet $\bigl(\sigma_{0,n_i}^{\pv V}\bigr)_{i\in I}$
  of $\bigl(\sigma_{0,n}^{\pv V}\bigr)_{n\in\nn}$ such that
  $\xi=\lim_{i\in I}\prov V\sigma_{0,n_i}$ in $\catsig_{\pv
    V}(\bsigma)$ and $A_{n_i}=B$ for all $i\in I$, as the profinite
  category $\catsig_{\pv V}(\bsigma)$ has a discrete vertex space.
  
  Let $u\in \Om BV$. Since we are dealing with the pointwise topology
  of $\catsig_{\pv V}(\bsigma)$, we have $\xi(u)=\lim_{i\in I}\prov
  V\sigma_{0,n_i}(u)$. This implies that $\xi(u)\in \img_{\pv
    V}(\bsigma)$ by Lemma~\ref{l:V-image-as-a-set-of-cluster-points},
  thus establishing the inclusion $\img(\xi)\subseteq\img_{\pv
    V}(\bsigma)$.

  Conversely, let $w\in \img_{\pv V}(\bsigma)$. Then, for each $i\in
  I$, there is $u_i\in\Om BV$ such that $w=\prov V\sigma_{n_i}(u_i)$.
  Let $u$ be a cluster point of the net $(u_i)_{i\in I}$. By
  continuity of the evaluation mapping $\hom(\Om BV,\Om {A_0}V)\times
  \Om BV\to \Om {A_0}V$, considered in
  Corollary~\ref{c:evaluation-mapping-is-continuous}, it follows that
  $w=\xi(u)$, thus establishing the inclusion $\img_{\pv
    V}(\bsigma)\subseteq \img(\xi)$.

  When $u\in B$, from the equality $\xi(u)=\lim_{i\in I}\prov
  V\sigma_{0,n_i}(u)$ we get $\xi(u)\in\Lambda_{\pv V}(\bsigma)$ by
  the definition of $\Lambda_{\pv V}(\bsigma)$.
\end{proof}
  
\begin{corollary}
  \label{c:letters-mapped-to-Js}
  Let $\bsigma$ be a primitive directive sequence with finite alphabet
  rank~$n$.
  \begin{enumerate}
  \item\label{item:letters-mapped-to-Js-1} There exists a $\pv
    V$-compression $\xi\from\Om BV\to\Om{A_0}V$ of $\bsigma$ with
    $|B|\le n$.
  \item\label{item:letters-mapped-to-Js-2} For every a $\pv
    V$-compression $\xi\from\Om BV\to\Om{A_0}V$ of $\bsigma$ with
    $|B|\le n$, the profinite semigroup $\img_{\pv V}(\bsigma)$ is
    generated by the finite subset $\xi(B)$ of $J_{\pv V}(\bsigma)$
    with at most $n$ elements.
  \end{enumerate}
\end{corollary}

\begin{proof}
  We may as well assume that $\bsigma$ is bounded with alphabet rank
  $n$ (cf.~Remark~\ref{r:bounded-versus-finite-alphabet-rank}).

  \ref{item:letters-mapped-to-Js-1}
  We may pick an alphabet $B$ such that $\card(B)=n$ and $B=A_k$ for
  infinitely many values of $k$. 
  Then, by compactness of the category $\catsig_{\pv V}(\bsigma)$,
  there exists a $\pv V$-compression $\xi\from\Om BV\to\Om{A_0}V$ of
  $\bsigma$.

  \ref{item:letters-mapped-to-Js-2} By
  Theorem~\ref{t:letters-mapped-to-Js}, the profinite semigroup
  $\img_{\pv V}(\bsigma)$ is generated by the subset $\xi(B)$ of
  $\Lambda_{\pv V}(\bsigma)$, which is contained in $J_{\pv
    V}(\bsigma)$ by
  Proposition~\ref{p:properties-of-lambda}\ref{item:properties-of-lambda-3}.
\end{proof}

The next example illustrates that past investigation, concerning the
endomorphism $(\prov S{\varphi})^\omega$ associated to a primitive
substitution
$\varphi$~\cite{Almeida:2005c,Almeida&ACosta:2013,Goulet-Ouellet:2022d,Goulet-Ouellet:2022c},
may be seen as a special case of the study of $\pv V$-images of
bounded directive sequences.

\begin{example}
  \label{e:V-image-omega-power}
  Let $\sigma\from A^+\to A^+$ be a substitution. Consider the
  constant directive sequence $\bsigma=(\sigma,\sigma,\ldots)$. Then
  the equality $\img_{\pv V}(\bsigma)=\img((\prov V\sigma)^\omega)$
  holds (cf.~Example~\ref{e:omega-power-V-compression}), and so
  $\img_{\pv V}(\bsigma)$ is generated by $\card(A)$ elements of
  $J_{\pv V}(\bsigma)$.
\end{example}
 
\section{Kernel endomorphisms for bounded directive sequences}
\label{sec:models}

When arguing about a $\pv V$-compression $\xi=\lim_{i\in I}\prov V\sigma_{0,n_i}$,
where the limit of the net is being taken in $\catsig_{\pv V}(\bsigma)$,
it will be convenient to keep
track of the path $\bigl(\prov V\sigma_0,\prov V\sigma_1,\ldots,\prov V\sigma_{n_i-1}\bigr)$ of the graph $\catsig_{\pv V}(\bsigma)$, which produces
the homomorphism $\prov V\sigma_{0,n_i}$ by multiplication of its edges.
Further abstracting, we are
led to Definition~\ref{d:model} below, which will play a key role in the proof of Theorem~\ref{t:recurrent-encoding-implies-saturating}. In what follows, for any mapping $\psi$ and element $x$ of the domain of $\psi$, we may denote
$\psi(x)$ by $\psi_x$.

\begin{definition}[Model of directive sequence]
  \label{d:model}
  Let $\bsigma$ be a bounded directive sequence.
  A \emph{\pv{V}-model} of $\bsigma$ is a triple
  $\boldsymbol{\psi} = (\Gamma, \psi, x)$ consisting of:
  \begin{enumerate}
    \item a finite-vertex graph $\Gamma$;
    \item a continuous category homomorphism
      $\psi\from\Om\Gamma{Cat}\to
      \catsig_{\pv V}(\bsigma)$;
    \item a prefix accessible pseudopath $x$ of $\Om\Gamma{Cat}$ such that $\psi_{x[n]} =
      \prov V\sigma_n$ for all $n\in\nn$.
  \end{enumerate}
\end{definition}

\begin{proposition}
  \label{p:compressions-versus-models}
  Let $\bsigma$ be a bounded directive sequence.
  Let $\xi$ be a morphism of the category $\catsig_{\pv V}(\bsigma)$.
  Then $\xi$ is a $\pv V$-compression of $\bsigma$ if and only if
  $\xi=\psi_x$ for some $\pv V$-model $(\Gamma, \psi, x)$.
\end{proposition}

\begin{proof}
  First consider a $\pv V$-model $(\Gamma, \psi, x)$. 
  For each $n\in\nn$, let $x_n$ be the prefix of length $n$ of $x$.
  Note that $\psi_{x_n}=\prov V\sigma_{0,n}$.
  Since $x$ is a prefix accessible pseudopath, there is a net $(x_{n_i})_{i\in I}$ converging in $\Om\Gamma{Cat}$ to $x$.
  Then $\psi_x=\lim \prov V\sigma_{0,n_i}$
  is a $\pv V$-compression of $\bsigma$.
  
  Next suppose that $\xi=\lim_{i\in I}\prov V\sigma_{0,n_i}$,
  where the limit of the net is being taken in $\catsig_{\pv V}(\bsigma)$.
  Since  $\bsigma$ is bounded, the graph $\Gamma=\catsig_{\pv V}(\bsigma)$ is finite.
  Therefore, the identity mapping on $\Gamma$ extends to a unique continuous
  homomorphism of categories
  $\psi\from\Om\Gamma{Cat}\to \catsig_{\pv V}(\bsigma)$.
  For each $n\in\nn$, consider the following path in the graph $\Gamma$:
  \begin{equation*}
    x_n=(\prov V\sigma_0,\prov V\sigma_1,\ldots,\prov V\sigma_{n-1}).
  \end{equation*}
  We then have $\psi_{x_n}=\prov V\sigma_{0,n}$.
  By compactness, we may consider a cluster point $x$ of the net $(x_{n_i})_{i\in I}$ in $\Om{\catsig_{\pv V}(\bsigma)}{Cat}$. Note that $x$ is a prefix accessible pseudopath of $\Om{\catsig_{\pv V}(\bsigma)}{Cat}$, whence $(\catsig_{\pv V}(\bsigma), \psi, x)$ is a $\pv V$-model of $\bsigma$. 
  Moreover, by continuity of
  $\psi$, we must have $\psi_x=\xi$.
\end{proof}

\begin{remark}
  \label{r:existence-V-model}
  Since every bounded directive has a $\pv{V}$ compression, it follows
  from \ref{p:compressions-versus-models} that every bounded directive
  sequence has a $\pv{V}$-model.
\end{remark}

\begin{corollary}
  \label{c:compressions-versus-models}
  Let $\bsigma$ be a bounded directive sequence.
  If $(\Gamma, \psi, x)$ is a $\pv V$-model of~$\bsigma$,
  then $\img_{\pv V}(\bsigma)=\img(\psi_x)$.
\end{corollary}

\begin{proof}
  This follows from combining Theorem~\ref{t:letters-mapped-to-Js} with Proposition~\ref{p:compressions-versus-models}.
\end{proof}

The following may be convenient to deal with tails of a directive sequence, as it often occurs. Recall that if $x$
is a pseudopath with prefix $u$ of finite length $k$, then $x^{(k)}$ is the unique pseudopath $w$ such that $x=uw$.

  \begin{lemma}
    \label{l:cutting-finite-prefixes-Pw-translation-for-models}
   Let $\bsigma$ be a bounded directive sequence.
    Let $k\in\nn$.
    If $(\Gamma,\psi,x)$ is a \pv V-model of $\bsigma$,
    then $(\Gamma,\psi,x^{(k)})$ is a \pv V-model of $\bsigma^{(k)}$.
  \end{lemma}

  \begin{proof}
    For every infinite-length pseudopath $x$, and every $n\in\nn$, the equality
    $(x^{(k)})[n]=x[k+n]$ holds. If moreover $x$ is a prefix accessible pseudopath, then $x^{(k)}$ is also a prefix accessible pseudopath~\cite[Proposition 6.10]{Almeida&ACosta&Goulet-Ouellet:2024a}.
  \end{proof}
  
The following notion allows us to give a new description of $\img^\infty_\pv{V}(\bsigma)$, found in Proposition~\ref{p:isomorphism-with-inverse-limit}. It is also  important in the companion papers~\cite{Almeida&ACosta&Goulet-Ouellet:2024c,Almeida&ACosta&Goulet-Ouellet:2024d}.

\begin{definition}[Kernel endomorphism of a directive sequence]
  \label{d:V-kernel-endomorphism}
  Let $\bsigma$ be a bounded directive sequence.
  A \emph{$\pv V$-kernel endomorphism for $\bsigma$} is an endomorphism
  of $\Om {A_{\alpha(y)}}V$ of the form $\psi_y$ for some $\pv V$-model $(\Gamma,\psi,x)$
  and some element $y$ of the kernel of the right stabilizer $\rstab_{\Om\Gamma{Cat}}(x)$.
\end{definition}

\begin{lemma}
  \label{l:V-kernel-endomorphism}
  Every $\pv V$-kernel endomorphism for $\bsigma$ is an idempotent continuous homomorphism.
  Moreover, if $\xi$ is a $\pv V$-compression of $\bsigma$, then
  $\xi=\xi\circ \zeta$ for some $\pv V$-kernel endomorphism $\zeta$ for $\bsigma$.
\end{lemma}
  
\begin{proof}
  If $(\Gamma,\psi,x)$ is \pv V-model of $\bsigma$, and $y$ is an element of the kernel of $\rstab(x)$,
  then $y$ is idempotent by Theorem~\ref{t:right-stabilizers-are-R-trivial-bands}, and so
  $\psi_y$ is an idempotent endomorphism.
  
  Moreover, if $\xi$ is a $\pv V$-compression of $\bsigma$, then $\xi=\psi_x$ for some \pv V-model $(\Gamma,\psi,x)$ of $\bsigma$,
  by Proposition~\ref{p:compressions-versus-models}.
  For $y$ in the kernel of $\rstab(x)$, set $\zeta=\psi_y$. Then we have $\xi=\psi_{xy}=\psi_x\circ\psi_y=\xi\circ\zeta$.
\end{proof}

\begin{proposition}\label{p:isomorphism-with-inverse-limit}
  Let \pv V be a pseudovariety of semigroups containing~\pv N and let $\bsigma$
  be a bounded primitive directive sequence.
  Suppose that $\zeta\from\Om BV\to\Om BV$ is a $\pv V$-kernel endomorphism for $\bsigma$.
  Then  the following hold:
  \begin{enumerate}
  \item $\zeta$ is primitive;\label{item:isomorphism-with-inverse-limit-1}
  \item the profinite semigroups $\img(\zeta)$ and $\img_{\pv V}^\infty(\bsigma)$ are isomorphic;\label{item:isomorphism-with-inverse-limit-3}
    \item the set $\zeta(B)$ is contained in a
    regular $\green{J}$-class of the semigroup $\img(\zeta)$.\label{item:isomorphism-with-inverse-limit-2}
  \end{enumerate}
\end{proposition}

\begin{proof}
  \ref{item:isomorphism-with-inverse-limit-1}
  By definition of $\pv V$-kernel endomorphism, there is a $\pv V$-model $(\Gamma,\psi,x)$
  of $\bsigma$ and some loop $y$ in the kernel of $\rstab(x)$
  such that $\zeta=\psi_y$.  
  By Theorem~\ref{t:characterization-of-L-minimal-right-stabilizers},
  there is a net $(x_i)_{i\in I}$ of finite-length prefixes of $x$ such that
  $x_i\to x$ and $x_i^{-1}x\to y$. Since the space of vertices of the category $\Om \Gamma{Cat}$ is discrete,
  we may as well assume that $x_i^{-1}x$ is a loop at $\alpha(y)=\omega(y)$.

  For every $i\in I$,
  the triple
  $(\Gamma,\psi,x_i^{-1}x)$ is a \pv V-model of $\bsigma^{(|x_i|)}$ by Lemma~\ref{l:cutting-finite-prefixes-Pw-translation-for-models},
  whence $\psi_{x_i^{-1}x}$ is a $\pv V$-compression of
  $\bsigma^{(|x_i|)}$ by Proposition~\ref{p:compressions-versus-models}.
  Therefore,  by Lemma~\ref{l:space-of-primitive-endomorphisms}, the continuous endomorphism $\psi_{x_i^{-1}x}$ is primitive for every $i\in I$, and so is $\zeta=\lim\psi_{x_i^{-1}x}$.
  This establishes item \ref{item:isomorphism-with-inverse-limit-1} in the statement.
  
  \ref{item:isomorphism-with-inverse-limit-3}
  Consider the infinite set $M=\{|x_i|:i\in I\}$.
  For each $n\in M$, let $\Psi_n$ denote the continuous endomorphism $\psi_{x_i^{-1}x}\from \Om BV\to \Om BV$
  when $i\in I$ is such that $x_i$ is the prefix of $x$ with length $n$.
  Since $\psi_{x_i^{-1}x}$ is a $\pv V$-compression of $\bsigma^{(|x_i|)}$,
  by Theorem~\ref{t:letters-mapped-to-Js} the equality
  $\img(\Psi_n)=\img_{\pv V}(\bsigma^{(n)})$ holds for every $n\in M$.
  
  As $\lim |x_i|=\infty$, the set $M$ is cofinal in $\nn$, and so
  the profinite semigroup
  $\img_{\pv V}^\infty(\bsigma)=\varprojlim_{n\in\nn} \img_{\pv V}(\bsigma^{(n)})$ is isomorphic to the inverse limit $\varprojlim_{n\in M}\img_{\pv V}(\bsigma^{(n)})$.

  Consider the mapping $\Psi\from \img(\zeta)\to \prod_{n\in M}\img_{\pv V}(\bsigma^{(n)})$
  defined by
  the formula
  $\Psi(u)=(\Psi_n(u))_{n\in M}$ for every $u\in \img(\zeta)$. Note that $\Psi$ is a continuous homomorphism, as all the mappings $\Psi_n$ are continuous homomorphisms. Hence, to prove
  item~\ref{item:isomorphism-with-inverse-limit-3} it suffices
  to show that $\img(\Psi)=\varprojlim_{n\in M}\img_{\pv V}(\bsigma^{(n)})$ and that $\Psi$ is injective. 
  
  Let $n\in M$, and take $i\in I$ such that $n=|x_i|$.
  Let $m\in M$ be such that $m>n$, and take $j\in I$ such that $|x_j|=m$.
  Then we have $x_j=x_iz$ for a path $z$ of length $m-n$, with
  $\psi_z=\prov V\sigma_{n,m}$.
  Since
  \begin{equation*}
    x_iz(x_j^{-1}x)=x_j(x_j^{-1}x)=x=x_i(x_i^{-1}x),
  \end{equation*}
  canceling the finite-length prefix $x_i$ we obtain
  $z(x_j^{-1}x)=x_i^{-1}x$ (cf.~Proposition~\ref{p:letter-super-cancelativity} and~Remark~\ref{r:extension-pseudowords-to-pseudopaths}). Therefore, for every $u\in\img(\zeta)$, we have
  \begin{equation*}
    \prov V\sigma_{n,m}(\Psi_m(u))=\psi_z\psi_{x_j^{-1}x}(u)=\psi_{x_i^{-1}x}(u)=\Psi_n(u).
  \end{equation*}
  This shows that the inclusion $\img(\Psi)\subseteq \varprojlim_{n\in M}\img_{\pv V}(\bsigma^{(n)})$ holds.

  We claim that $\Psi_n=\Psi_n\circ\zeta$ for every $n\in M$. Letting $i\in I$ be such that $|x_i|=n$, one has
  \begin{equation*}
    x_i(x_i^{-1}x)y=xy=x=x_i(x_i^{-1}x),
  \end{equation*}
  thus $(x_i^{-1}x)y=x_i^{-1}x$ by cancellation of the finite-length
  prefix $x_i$ (see Remark~\ref{r:extension-pseudowords-to-pseudopaths}). As $\Psi_n=\psi_{x_i^{-1}x}$ and $\zeta=\psi_y$, this
  establishes the claim that $\Psi_n=\Psi_n\circ\zeta$. Therefore, we
  have $\Psi_n(\img(\zeta))=\img(\Psi_n)=\img_{\pv V}(\bsigma^{(n)})$
  for every $n\in M$. This entails the equality
  $\img(\Psi)=\varprojlim_{n\in M}\img_{\pv V}(\bsigma^{(n)})$ by well
  known properties of the continuous mappings involving inverse
  systems of compact spaces~\cite[Theorem 3.2.14]{Engelking:1989}.
  
  It remains to show that $\Psi$ is injective. Let $u,v\in\img(\zeta)$
  be such that $\Psi(u)=\Psi(v)$. Then we have $\Psi_n(u)=\Psi_n(v)$
  for every $n\in M$. This is the same as saying that
  $\psi_{x_i^{-1}x}(u)=\psi_{x_i^{-1}x}(v)$ for every $i\in I$. Since
  we are endowing $\hom(\Om BV,\Om BV)$ with the pointwise topology,
  we get
  \begin{equation*}
    \zeta(u)=\psi_y(u)=\lim_{i\in I}\psi_{x_i^{-1}x}(u)=\lim_{i\in I}\psi_{x_i^{-1}x}(v)=\psi_y(v)=\zeta(v).
  \end{equation*}
  But $\zeta$ is idempotent (cf.~Lemma~\ref{l:V-kernel-endomorphism}), and so it restricts to the identity on $\img(\zeta)$.
  Hence we have $u=v$. This establishes the injectivity of $\Psi$ and finishes the proof of item~\ref{item:isomorphism-with-inverse-limit-3} of the proposition.

    \ref{item:isomorphism-with-inverse-limit-2}
  Note that, since $\zeta$ is primitive by~\ref{item:isomorphism-with-inverse-limit-1},
  every element of  $\img(\zeta)$ admits every element of $\zeta(B)$ as a factor.
  Hence, to prove item  \ref{item:isomorphism-with-inverse-limit-2}, it suffices to show that
  the semigroup has $\img(\zeta)$ has a unique maximal $\green{J}$-class, which is regular.
  Since, by Proposition~\ref{p:at-the-limit-properties-of-lambda},
  the semigroup $\img_{\pv V}^\infty(\bsigma)$ has that property, item~\ref{item:isomorphism-with-inverse-limit-2} follows immediately
  from item~\ref{item:isomorphism-with-inverse-limit-3}.
\end{proof}

In the setting of Proposition~\ref{p:isomorphism-with-inverse-limit},
let $J$ be the regular $\green{J}$-class of $\Om BV$ containing the
set $\zeta(B)$. If $\varphi\from B^+\to B^+$ is a primitive
substitution, then we know that the $\green{J}$-class $J_{\pv
  V}(\varphi)$ is $\le_{\green{J}}$-maximal among the regular
$\green{J}$-classes of $\Om BV$, whenever $\pv V$ contains $\pv {\loc
  Sl}$ (cf.~Proposition~\ref{p:mirage}). Hence, as $\zeta$ is a
primitive continuous endomorphism of $\Om BV$, it is natural to ask
whether the $\green{J}$-class $J$ is also $\le_{\green{J}}$-maximal
among the regular $\green{J}$-classes of $\Om BV$. The following
example shows that that may not be the case.

\begin{example}
  \label{eg:not-J-maximal}
  Consider the sequence of substitutions $\sigma_n$ over the alphabet
  $A=\{\ltr{a},\ltr{b}\}$ defined by
  \begin{equation*}
    \sigma_n\from  \ltr{a}\mapsto \ltr{ab^n},\ \ltr{b}\mapsto \ltr{a}.
  \end{equation*}
  Note that $\bsigma=(\sigma_n)_{n\in\nn}$ is a bounded primitive
  directive sequence. Let \pv V, $(\Gamma,\psi,x)$, $y$, $x_i$ and
  $y_i$ be as in the statement and proof of
  Proposition~\ref{p:isomorphism-with-inverse-limit}. Then,
  $\ltr{ab}^{|x_i|}$ is a prefix of~$\psi_{y_i}(\ltr{a})$, so that
  $\ltr{ab}^\omega$ is a prefix of $\psi_y(\ltr{a})$. Hence,
  $\ltr{b}^\omega$ is an idempotent which lies strictly
  $\le_{\green{J}}$-above $\psi_y(\ltr{a})$ provided \pv V
  contains~\pv{Sl}. Thus, for such \pv V, the $\green{J}$-class of
  $\psi_y(\ltr{a})$ is not $\le_{\green{J}}$-maximal among the regular
  $\green{J}$-classes of~$\Om AV$.
\end{example}

\section{Saturating directive sequences}
\label{sec:satur-direct-sequ}

We saw in Section~\ref{sec:proper-case} that when the primitive
directive sequence $\bsigma$ has a proper contraction, the $\pv
V$-image of $\bsigma$ is a closed subgroup of the free pro-$\pv V$
semigroup over the alphabet of $X(\bsigma)$. It is natural to ask for
necessary and sufficient conditions under which this subgroup is a
\emph{maximal} subgroup of that free pro-$\pv V$ semigroup. In this
section, we investigate that question in a more general framework,
assuming only that $\bsigma$ is primitive, not necessarily having a
proper contraction. In the process, we establish a strong link with
the notion of recognizable directive sequence.

This section is divided into three subsections. In the first one, we
lay the foundations for our framework by introducing the notion of
\emph{$\pv{V}$-saturating} directive sequence (\pv{V} a
pseudovariety). We give a straightforward proof that primitive
directive sequences consisting of pure encodings are
$\pv{S}$-saturating (Theorem~\ref{t:surjectivity-pure-encoding}), and
study the case where $\bsigma$ is recurrent and consists of encodings
that may not be pure
(Theorem~\ref{t:recurrent-encoding-implies-saturating}). In the second
subsection, we show that recognizability of $\bsigma$ is sufficient
for $\bsigma$ to be $\pv{S}$-saturating (Theorem~\ref{t:surjectivity})
Finally, in the last subsection we consider cases where
recognizability is a necessary condition for $\bsigma$ to be
$\pv{S}$-saturating
(Theorem~\ref{t:a-sort-of-converse-of-surjectivity-theorem}), leading
to new criteria for recognizable directive sequences
(Corollary~\ref{c:surjectivity-pure-encoding} and
Theorem~\ref{t:recurrent-encoding-made-implies-recognizable}).

\subsection{The notion of \texorpdfstring{$\pv S$}{S}-saturating directive sequence}
\label{sec:notion-pv-s}

In what follows, $\bsigma$ is a directive sequence $(\sigma_n)_{n\in\nn}$ with $\sigma_n$ a homomorphism from $A_{n+1}^+$ to $A_n^+$.
The following definition is the cornerstone upon which this entire
section is built.

\begin{definition}
  \label{d:V-saturating}
  Let $\bsigma$ be a primitive directive sequence and $\pv{V}$ be a
  pseudovariety containing~$\pv{N}$. We say that $\bsigma$ is
  \emph{\pv{V}-saturating} if $\img_\pv{V}(\bsigma)$ contains a maximal subgroup
  of $J_{\pv V}(\bsigma)$.
\end{definition}

\begin{remark}
  \label{r:V-saturating-proper}
  If $\bsigma$ is primitive and has a proper contraction, then
  $\bsigma$ is $\pv S$-saturating if and only if
  $\img_\pv{V}(\bsigma)$ is a maximal subgroup
  of $J_{\pv V}(\bsigma)$, by Theorem~\ref{t:proper-is-group}.
\end{remark}

In the next proposition we present several equivalent
alternatives for Definition~\ref{d:V-saturating}.

\begin{proposition}
  \label{p:saturation-equivalence}
  Let $\bsigma$ be a primitive directive sequence and $\pv{V}$
  be a pseudovariety containing~$\pv{N}$. The following conditions are
  equivalent:
  \begin{enumerate}
  \item $\bsigma$ is $\pv{V}$-saturating;
    \label{i:saturation-equivalence-1}
  \item $\img_\pv{V}(\bsigma)$ contains an $\green{H}$-class of
    $J_{\pv V}(\bsigma)$;
    \label{i:saturation-equivalence-2}
  \item $J_{\pv V}(\bsigma)\cap \img_\pv{V}(\bsigma)$ is a union of
    $\green{H}$-classes of $J_{\pv V}(\bsigma)$;
    \label{i:saturation-equivalence-3}
  \item if $p,q,r$ are elements of~$\Om{A_0}V$ such that
    the relations $p\green{R}q\green{L}r$ hold in $\Om {A_0}V$, and $p$ and $r$ belong to~$J_{\pv
      V}(\bsigma)\cap \img_\pv{V}(\bsigma)$, then so does $q$.
    \label{i:saturation-equivalence-4}
  \end{enumerate}
\end{proposition}

\begin{proof}
  Let $J=J_{\pv V}(\bsigma)\cap \img_\pv{V}(\bsigma)$ and recall that
  $J$ is a regular $\green{J}$-class of the semigroup $\img_{\pv V}(\bsigma)$, by
  Theorem~\ref{t:Im-intersection-J}.

  The implication
  $\ref{i:saturation-equivalence-1}
  \Rightarrow
  \ref{i:saturation-equivalence-2}$
  holds because every maximal
  subgroup of a semigroup is an $\green{H}$-class of that same semigroup.
  
  For
  $\ref{i:saturation-equivalence-2}
  \Rightarrow
  \ref{i:saturation-equivalence-3}$,
  suppose that $H$ is an
  $\green{H}$-class of~$J_{\pv V}(\bsigma)$ contained in $\img_{\pv V}(\bsigma)$. Take $h\in H$. Let $s\in J$.
  Since $H\subseteq J$, by Theorem~\ref{t:Im-intersection-J} there are
  $u,v\in\img_{\pv V}(\bsigma)$ such that $uhv=s$. By Green's Lemma (cf.~\cite[Lemma~A.3.1]{Rhodes&Steinberg:2009qt}), applied to
  $\Om{A_0}V$ and $\img_{\pv V}(\bsigma)$, we deduce that $uHv$ is
  the $\green{H}$-class of $s$ in both $\Om{A_0}V$ and~$\img_{\pv V}(\bsigma)$.
  As $s$ is an arbitrary element of $J$, we conclude that
  $J$ is a union of $\green{H}$-classes of~$J_{\pv V}(\bsigma)$.

  We proceed to show that
  $\ref{i:saturation-equivalence-3}
  \Rightarrow
  \ref{i:saturation-equivalence-4}$.
  For each Green's relation symbol $\mathcal K$, denote by $\mathcal K'$
  the corresponding Green's relation in $\img_{\pv V}(\bsigma)$, that is, $\mathcal K'=\mathcal K_{\img_{\pv V}(\bsigma)}$.
  Take $p,q,r$ as
  in~\ref{i:saturation-equivalence-4}.
  Since $J$ is a $\green{J}'$-class by Theorem~\ref{t:Im-intersection-J}, there is $t\in J$ such
  that $p\green{R}'t\green{L}'r$. Hence, $q$ lies in the same $\green{H}$-class of~$J_{\pv V}(\bsigma)$ as $t$.
  Since $t$ belongs
  to~$\img_{\pv V}(\bsigma)$, it follows from
  \ref{i:saturation-equivalence-3} that $q$
  also belongs to $\img_{\pv V}(\bsigma)$.

  For $\ref{i:saturation-equivalence-4} \Rightarrow
  \ref{i:saturation-equivalence-1}$, take an idempotent $e\in J$.
  In~\ref{i:saturation-equivalence-4} we may take $p=r=e$ and $q$ an
  arbitrary element in the maximal subgroup of $\Om {A_0}V$ containing
  $e$, and then conclude that $q\in J$.
\end{proof}

Before proceeding, it is worth noting the following simple observation.

\begin{proposition}
  \label{p:saturation-inheritance}
  If $\bsigma$ is $\pv{V}$-saturating and \pv{W} is a pseudovariety
  such that $\pv{\loc Sl}\subseteq\pv{W}\subseteq\pv{V}$, then
  $\bsigma$ is also $\pv{W}$-saturating.
\end{proposition}

\begin{proof}
  Suppose that $\img_\pv{V}(\bsigma)$ contains a maximal subgroup $H$
  of $J_{\pv V}(\bsigma)$. Then $\img_\pv{W}(\bsigma)$ contains
  $p_{\pv{V}, \pv{W}}(H)$ by
  Proposition~\ref{p:projecting-profinite-image}. Moreover, the set
  $p_{\pv{V}, \pv{W}}(H)$ is a maximal subgroup of $J_{\pv
    W}(\bsigma)$ by Corollary~\ref{c:projecting-the-J-class-of-X}.
  Hence, $\bsigma$ is $\pv W$-saturating.
\end{proof}  

Let $\bsigma = (\sigma_n)_{n\in\nn}$ be a directive sequence. We say
that $\bsigma$ is \emph{pure} if $\sigma_n$ is a pure encoding for all
$n\in\nn$.

\begin{theorem}
  \label{t:surjectivity-pure-encoding}
  Let $\bsigma$ be a primitive directive sequence. If $\bsigma$ is
  pure, then it is $\pv{S}$-saturating.
\end{theorem}

\begin{proof}
  There is a maximal subgroup $H$ of $J_{\pv S}(\bsigma)$ such that
  $H\cap \img_{\pv S}(\bsigma)\neq\emptyset$, by
  Theorem~\ref{t:Im-intersection-J}. Let $n\in\nn$. In particular, we
  have $H\cap \img(\prov S\sigma_{0,n})\neq\emptyset$. The
  homomorphism $\sigma_{0,n}$ is pure, as every composition of pure
  encodings remains pure. In other words, the set
  $C=\sigma_{0,n}(A_{n})$ is a pure code. Since $\img(\prov
  S\sigma_{0,n})=\clos S(C^+)$, it follows from
  Proposition~\ref{p:surjective-image-pure-code} that
  $H\subseteq\img(\prov S\sigma_{0,n})$. As $n$ is arbitrary, this
  shows that $H\subseteq \img_{\pv S}(\bsigma)$, thereby establishing
  that $\bsigma$ is $\pv S$-saturating.
\end{proof}

The next example illustrates Theorem~\ref{t:surjectivity-pure-encoding}.

\begin{example}
  \label{e:example-that-img-may-not-be-in-J-continued}
  Recall the primitive substitution over
  $A=\{\ltr{a},\ltr{b},\ltr{c}\}$ considered in
  Example~\ref{e:example-that-img-may-not-be-in-J}:
  \begin{equation*}
    \sigma\from
    \ltr a\mapsto \ltr{ac},\
    \ltr b\mapsto \ltr{bcb},\
    \ltr c\mapsto \ltr{ba},
  \end{equation*}
  Set $C=\sigma(A)$. No element of $C$ is a prefix or a suffix of
  another element of $C$, that is, $C$ is a \emph{bifix} code
  (cf.~\cite{Berstel&Perrin&Reutenauer:2010}). Using, for instance, a
  GAP package~\cite{Delgado&Morais:sgpviz}, one may check that the
  syntactic semigroup of $C^+$ is aperiodic. Hence, $\sigma$ is a pure
  encoding, and so the directive sequence
  $\bsigma=(\sigma,\sigma,\ldots)$ is \pv S-saturating by
  Theorem~\ref{t:surjectivity-pure-encoding}. Denote by $\pro\sigma$
  the unique continuous endomorphism $\prov S\sigma\from\Om AS\to\Om
  AS$ extending $\sigma$. Recall that $\img_{\pv
    S}(\bsigma)=\img(\pro\sigma^\omega)$
  (cf.~Example~\ref{e:V-image-omega-power}).

  From
  Proposition~\ref{p:saturation-equivalence}\ref{i:saturation-equivalence-3},
  it follows that $J_{\pv S}(\bsigma)\cap \img_\pv{S}(\bsigma)$ is a
  union of $\green{H}$-classes. Since $\pro\sigma^\omega$ is
  idempotent, these $\green{H}$-classes are determined by the first
  and last letters of the images of~$\pro\sigma^\omega$
  (Table~\ref{tb:omega-letters}).
  \begin{table}[ht]
    \begin{tabular}{ccc}
      \toprule
      $\ell$
      & first letter of $\pro\sigma^\omega(\ell)$
      & last letter of $\pro\sigma^\omega(\ell)$ \\
      \midrule
      $\ltr a$ & $\ltr a$ & $\ltr a$ \\
      $\ltr b$ & $\ltr b$ & $\ltr b$ \\
      $\ltr c$ & $\ltr b$ & $\ltr c$ \\
      \bottomrule
      \addlinespace
    \end{tabular}
    \caption{First and last letters of the images of~$\pro\sigma^\omega$}
    \label{tb:omega-letters}
  \end{table}
  We therefore conclude that $J_{\pv S}(\bsigma)\cap
  \img_\pv{S}(\bsigma)$ consists of six $\green{H}$-classes of $\Om
  AS$, of which $\pro\sigma^\omega(\ltr{a})$,
  $\pro\sigma^\omega(\ltr{ab})$, $\pro\sigma^\omega(\ltr{ac})$,
  $\pro\sigma^\omega(\ltr{ba})$, $\pro\sigma^\omega(\ltr{b})$,
  $\pro\sigma^\omega(\ltr{c})$ are representative elements. As seen in
  Example~\ref{e:example-that-img-may-not-be-in-J}, the pseudoword
  $\pro\sigma^\omega(\ltr{a}^2)$ is not in $J_{\pv S}(\bsigma)\cap
  \img_\pv{S}(\bsigma)$, and therefore the $\green{H}$-class of
  $\pro\sigma^\omega(\ltr{a})$ is not a group.
\end{example}

We say that a directive sequence $\bsigma=(\sigma_n)_{n\in\nn}$ is
\emph{recurrent} if, seen as a right-infinite word over the alphabet
$\{\sigma_n : n\in\nn\}$, it is a recurrent right-infinite word.

\begin{theorem}
  \label{t:recurrent-encoding-implies-saturating}
  Let $\bsigma$ be a bounded primitive directive sequence.
  Suppose moreover that $\bsigma$ is recurrent and encoding.
  If there is $k\in\nn$ such that $\bsigma^{(k)}$
  is \pv S-saturating, then $\bsigma$ is \pv S-saturating.
\end{theorem}

\begin{proof}
  Let $k$ be a positive integer such that $\bsigma^{(k)}$ is \pv
  S-saturating. By
  Corollary~\ref{c:chain-of-idempotents-in-the-inverse-limit}, we may
  take idempotent pseudowords $g\in J_{\pv S}(\bsigma)\cap\img_{\pv
    S}(\bsigma)$ and $h\in J_{\pv S}(\bsigma^{(k)})\cap\img_{\pv
    S}(\bsigma^{(k)})$ such that $g=\prov S\sigma_{0,k}(h)$. We want
  to show that the maximal subgroup of $\Om {A_0}S$ to which $g$
  belongs is contained in $\img_{\pv S}(\bsigma)$.
  
  Because $\bsigma$ is recurrent,
  Proposition~\ref{p:idempotents-in-Pw-Cat} yields an $\pv S$-model
  $(\Gamma,\psi,e)$ of $\bsigma$ where $e$ is an idempotent. Let~$z$ be the prefix of length $k$ of~$e$ and consider the idempotent
  $f=e^{(k)}z=z^{-1}ez$.
  We claim that $\psi_z$ and $\psi_{z^{-1}e}$ restrict to mutually inverse continuous isomorphisms between
  $\img_{\pv S}(\bsigma^{(k)})$ and $\img_{\pv S}(\bsigma)$.

  In order to prove the claim, we first note that $\img_{\pv S}(\bsigma)=\img(\psi_e)$ by Corollary~\ref{c:compressions-versus-models}. Since $(\Gamma,\psi,z^{-1}e)$ is an $\pv S$-model of
  $\bsigma^{(k)}$ by
  Lemma~\ref{l:cutting-finite-prefixes-Pw-translation-for-models}, we have
  \begin{equation*}
    \img_{\pv S}(\bsigma^{(k)})=\img(\psi_{z^{-1}e})=\img(\psi_f),
  \end{equation*}
  where the first equality holds by
  Corollary~\ref{c:compressions-versus-models}
  and the second because $z^{-1}e\green{R}f$.
  On the other hand, the equalities $zf\cdot z^{-1}e=e$, $zf=ez$ yield $zf\green{R} e$
  and so
  \begin{equation*}
   \img(\psi_{zf})=\img(\psi_e)=\img_{\pv S}(\bsigma).  
 \end{equation*}
 
  As $f$ is idempotent, the homomorphism $\psi_f$ restricts to the
  identity on $\img(\psi_f)$ which can be factored as in the following
  commutative diagram of restricted mappings, which, for simplicity,
  are indicated simply by adding a vertical bar:
  \begin{displaymath}
    \xymatrix@C=0mm{
      \img(\psi_f)
      \ar[rr]^{\psi_f|=\mathrm{id}}
      \ar[rd]_{\psi_z|}
      &
      &
      \img(\psi_f)
      \\
      &
      \img(\psi_{zf})
      \ar[ru]_{\psi_{z^{-1}e}|}.
      &
    }
  \end{displaymath}
 In view of the aforementioned equalities $\img(\psi_f)=\img_{\pv S}(\bsigma^{(k)})$
  and $\img(\psi_{zf})=\img_{\pv S}(\bsigma)$,
  it follows that $\psi_z$ restricts to a continuous isomorphism from
  $\img_{\pv S}(\bsigma^{(k)})$ onto $\img_{\pv S}(\bsigma)$, and that $\psi_{z^{-1}e}$
  restricts to a continuous isomorphism from $\img_{\pv S}(\bsigma)$ onto
  $\img_{\pv S}(\bsigma^{(k)})$. This proves the claim.

  In what follows, bear in mind that
  the equality
  \begin{equation*}
  \psi_{z^{-1}e}(g)=h
  \end{equation*}
  holds: indeed, one has $g=\psi_z(h)$
  as $\psi_z=\prov V\sigma_{0,k}$,
  and $h\in \img_{\pv S}(\bsigma^{(k)})= \img(\psi_f)$,
  whence $h=\psi_f(h)=\psi_{z^{-1}e}(\psi_z(h))=\psi_{z^{-1}e}(g)$.
  
  Since $\bsigma^{(k)}$ is \pv S-saturating, by
  Proposition~\ref{p:saturation-equivalence} we know that $\img_{\pv
    S}(\bsigma^{(k)})$ contains the maximal subgroup $H$ of~$J_{\pv
    S}(\bsigma^{(k)})$ to which the idempotent $h$ belongs. As $\psi_{z^{-1}e}$
  restricts to an isomorphism from $\img_{\pv S}(\bsigma)$ to
  $\img_{\pv S}(\bsigma^{(k)})$, the maximal subgroup $G$
  of~$\img_{\pv S}(\bsigma)$ containing the idempotent $g$ is such that
  $\psi_{z^{-1}e}(G)=H$. Let $K$ be the maximal subgroup
  of~$\Om{A_0}S$ containing $g$. Then, as $K\supseteq G$ and $H$ is a maximal subgroup of $\Om{A_k}S$, we must have
  \begin{equation*}
    \psi_{z^{-1}e}(G) = H = \psi_{z^{-1}e}(K).
  \end{equation*}
  Hence, to show that $\bsigma$ is \pv S-saturating, it suffices to
  show that the restriction of $\psi_{z^{-1}e}$ to $K$ is injective, as that implies that $K=G\subseteq \img_{\pv S}(\bsigma)$. The reader may wish to look at Figure~\ref{fig:torsion-free}
  while checking the proof.

  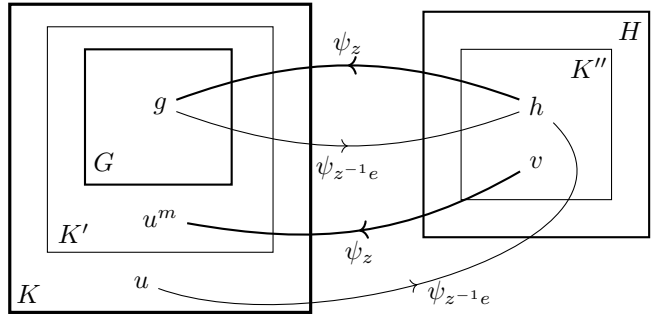
\begin{figure}[ht]
    \centering
    \begin{tikzpicture} [scale = 0.5,
      decoration={
        markings,
        mark=at position 0.5 with {\arrow{>}}
      }]
      
      \draw [very thick] (-6,-4) rectangle (2,4.2);
      \draw (-5.5,-3.5) node {$K$};
      \node (u) at (-2.5,-3.2) {$u$};
      \draw [thin] (-5,-2.4) rectangle (1,3.6);
      \draw (-4.3,-1.9) node {$K'$};
      \node (um) at (-2,-1.5) {$u^m$};
      \draw [thick] (-4,-0.6) rectangle (-0.1,3);
      \draw (-3.5,0) node {$G$};
      \node (g) at (-2,1.5) {$g$};
      
      \draw [thick] (5,-2) rectangle (11,4);
      \node (h) at (8,1.5) {$h$};
      \node(v) at (8,0) {$v$};
      \draw [thin] (6,-1) rectangle (10,3);
      \draw (9.4,2.5) node {$K''$};
      \draw (10.5,3.5) node {$H$};
      
      \draw [-,thick,postaction={decorate},out=160,in=20] (h) to
      node[midway,above] {$\psi_{z}$} (g);
      
      \draw [-,thick,postaction={decorate},out=-150,in=-10] (v) to
      node[midway,below] {$\psi_{z}$} (um);
      
      \draw [-,thin,postaction={decorate},out=-20,in=-160] (g) to
      node[below] {$\psi_{z^{-1}e}$} (h);
      
      \draw [-,thin,postaction={decorate},out=-20,in=-45] (u) to
      node[midway,below] {$\psi_{z^{-1}e}$} (h);
    \end{tikzpicture}
    \caption{Illustration of the proof of
      Theorem~\ref{t:recurrent-encoding-implies-saturating}}
    \label{fig:torsion-free}
  \end{figure}

  Let $u\in K$ be such that $\psi_{z^{-1}e}(u)=h$.
  Note that
  \begin{equation*}
    \lim u^{n!}=u^\omega=g\in \img(\psi_{z}).
  \end{equation*}
  As the set $\psi_z(A_k)^+$ is a recognizable language by Kleene's theorem~\cite[Theorem 3.2]{Lallement:1979}, and the equality $\img(\psi_z)=\clos S(\psi_z(A_k)^+)$ holds by continuity of $\psi_z$, we know that the set $\img(\psi_z)$
  is clopen by Theorem~\ref{t:V-recognizability}. Hence, there is a positive
  integer~$m$ such that $u^m\in \img(\psi_z)$. Let $K' = K\cap\img(\psi_z)$. Since $K'$ is a closed subgroup,
  the closed subsemigroup $\psi_z^{-1}(K')$ of $\Om {A_k}S$ contains a closed subgroup $K''$ such that $\psi_z(K'')=K'$~\cite[Proposition~3.1.1]{Rhodes&Steinberg:2009qt}. As $u^m\in K'$, we may take $v\in K''$ such that $u^m=\psi_z(v)$.

  We claim that $K''\subseteq H$. On one hand we have $\psi_z(h)=g\in K'=\psi_z(K'')$,
  and on the other hand, as $\bsigma$ is encoding, the homomorphism $\psi_z$ is injective by Theorem~\ref{t:sufficient-conditions-for-injectivity}.
  Therefore, we must have $h\in K''$, which establishes the claim $K''\subseteq H$ by maximality of~$H$.

  In particular, we have $v\in H$. On the other hand,  we also have
  \begin{equation*}
    h=h^m=\psi_{z^{-1}e}(u^m)=\psi_{z^{-1}e}(\psi_z(v))=\psi_f(v).
  \end{equation*}
  Since $H\subseteq\img_{\pv S}(\bsigma^{(k)})$ and $\psi_f$ restricts to
  the identity on $\img_{\pv S}(\bsigma^{(k)})$, it follows that
  $v=h$, thus $u^m=\psi_z(v)=\psi_z(h)=g$. But every closed subgroup of $\Om {A_0}S$ is torsion-free
  by~\cite[Theorem~1]{Rhodes&Steinberg:2008}, and so $u=g$. This
  proves that the restriction of $\psi_{z^{-1}e}$ to $K$ is injective,
  thereby establishing that $\bsigma$ is \pv S-saturating.
\end{proof}

The next proposition and the ensuing corollary, which are not necessary for the sequel, shed additional  light on Theorem~\ref{t:recurrent-encoding-implies-saturating}.

\begin{proposition}
  \label{p:occasionally-versus-eventually}
  Let $\bsigma=(\sigma_n)_{n\in\nn}$ be a primitive directive sequence.
  Let $m\in\nn$ be such that $\sigma_{0,m}^{\pv V}$ is injective.
  If $\bsigma$ is \pv V-saturating,
  then $\bsigma^{(m)}$ is \pv V-saturating.
\end{proposition}

\begin{proof}
  The intersection $J_{\pv V}(\bsigma^{(m)})\cap\img_{\pv V}(\bsigma^{(m)})$
  is a regular $\green{J}$-class of $\img_{\pv V}(\bsigma^{(m)})$, by Theorem~\ref{t:Im-intersection-J}, and so it contains a maximal subgroup $G$ of $\img_{\pv V}(\bsigma^{(m)})$.

  Since $\img_{\pv V}(\bsigma)=\sigma_{0,m}^{\pv V}(\img_{\pv V}(\bsigma^{(m)}))$ by Corollary~\ref{c:V-image-sent-to-image},
  and $\sigma_{0,m}^{\pv V}$ is injective,
  we know that $\sigma_{0,m}^{\pv V}$
  restricts to a continuous isomorphism
  $\img_{\pv V}(\bsigma^{(m)})\to \img_{\pv V}(\bsigma)$.
  Therefore, $\sigma_{0,m}^{\pv V}(G)$ is a maximal subgroup of $\img_{\pv V}(\bsigma)$. Moreover, the inclusion $\prov V\sigma_{0,m}(G)\subseteq J_{\pv V}(\bsigma)$ holds by Corollary~\ref{c:at-the-limit-properties-of-lambda}.
  Since we are assuming that $\bsigma$ is \pv V-saturating,
  the group $\sigma_{0,m}^{\pv V}(G)$ is in fact a maximal subgroup of $\Om {A_0}V$, by Proposition~\ref{p:saturation-equivalence}.

  Let $H$ be the maximal subgroup of $\Om {A_m}V$
  containing $G$.
  Since $\prov V\sigma_{0,m}(G)$ is a maximal subgroup of $\Om {A_0}V$,
  we necessarily have
  $\prov V\sigma_{0,m}(G)=\prov V\sigma_{0,m}(H)$, whence $G=H$ by injectivity of $\prov V\sigma_{0,m}$. This shows that the maximal subgroup
  $H$ of $\Om {A_m}V$ is contained in $\img_{\pv V}(\bsigma^{(m)})$, thus establishing that $\bsigma^{(m)}$ is $\pv V$-saturating.
\end{proof}

We say that a directive sequence $\bsigma=(\sigma_n)_{n\in\nn}$ is \emph{eventually \pv V-saturating} if there is $k\in\nn$ such that
$\bsigma^{(m)}$ is \pv V-saturating for every $m\geq k$.

\begin{corollary}
  \label{c:occasionally-versus-eventually}
  Let $\pv{H}$ be an extension-closed pseudovariety of groups and $\bsigma$ be a primitive directive sequence. Assume that $\sigma_n$ is an $\pvo{H}$-encoding for all $n\in\nn$.
  Then $\bsigma$ is eventually \pvo H-saturating if and only if
  there is $k\in\nn$ such that $\bsigma^{(k)}$ is \pvo H-saturating.
\end{corollary}

\begin{proof}
  Suppose that there is $k\in\nn$ such that $\bsigma^{(k)}$ is \pvo H-saturating.
   Let $m\in\nn$ be such that $k\leq m$.
   By assumption and Theorem~\ref{t:sufficient-conditions-for-injectivity}, each $\sigma_n^{\pvo H}$ is injective, hence so is $\sigma_{k,m}^{\pvo H} = \sigma_k^{\pvo H}\circ\cdots\circ\sigma_{m-1}^{\pvo H}$.
   Applying Proposition~\ref{p:occasionally-versus-eventually} to $\bsigma^{(k)}$,
   we deduce that $\bsigma^{(m)}$ is \pvo H-saturating.
   Since this holds for every $m\geq k$, we have established the ``if'' part of the corollary. The ``only if'' part is trivial.
\end{proof}

\subsection{Recognizable directive sequences}
\label{sec:case-recogn-direct}

In earlier work, the first two authors showed that every primitive aperiodic proper substitution $\sigma\from A^+\to A^+$ is such that $\img((\prov S\sigma)^\omega)$ is a maximal subgroup of $\Om AS$~(cf.~\cite[Lemma~6.3]{Almeida&ACosta:2013}, see also~\cite[Theorem~5.6]{Almeida&ACosta:2013}). An essential ingredient of the proof is Mossé's theorem stating that every primitive aperiodic substitution is recognizable. Therefore, the following theorem may be considered a generalization to the S-adic setting
of the result of the two first authors.

\begin{theorem}\label{t:surjectivity}
  Let $\bsigma$ be a primitive directive sequence. If
  $\bsigma$ is recognizable, then it is $\pv{S}$-saturating.
\end{theorem}

For the proof of this theorem we need the next couple of lemmas.  The first one is included in~\cite[Lemma 3.5]{Berthe&Steiner&Thuswaldner&Yassawi:2019} (also in \cite[Proposition 6.4.16]{Durand&Perrin:2022}).

\begin{lemma}
  \label{l:composition-of-recognizable-homomorphisms}
   Let $\bsigma=(\sigma_n)_{n\in\nn}$ be a primitive directive
 sequence.
  Let $n,m\in\nn$, with $n<m$. The substitution
  $\sigma_{n,m}$ is recognizable in $X(\bsigma^{(m)})$
  if and only if, for every integer $k$ such that $n\leq k\leq m$,
  the substitution $\sigma_k$ is recognizable in~$X(\bsigma^{(k)})$.
\end{lemma}

In~\cite[Proposition 4.4.17]{Almeida&ACosta&Kyriakoglou&Perrin:2020b}
one finds a proof of the following lemma.\footnote{In
  \cite[Proposition 4.4.17]{Almeida&ACosta&Kyriakoglou&Perrin:2020b}
  it is used the notation $\widehat{A^*}$ instead of $(\Om AS)^1$, and
  the assumption that $A$ is finite is implicit in the statement,
  since it is done globally in an early point of the chapter. Indeed,
  $\widehat{A^*}$ denotes there the free profinite \emph{monoid}
  generated by $A$, which is equal to $(\Om AS)^1$, see the last
  paragraph in~\cite[Section
  4.4]{Almeida&ACosta&Kyriakoglou&Perrin:2020b}. See
  also~\cite[Section 4.12]{Almeida&ACosta&Kyriakoglou&Perrin:2020b}
  for early references to this lemma.}

 \begin{lemma}
   \label{l:metric-open-mapping}
   Let $A$ be a finite alphabet. Let $u,v\in\Om AS$. If
   $(w_n)_{n\in\nn}$ is a sequence of elements of $\Om AS$ such that
   $\lim w_n=uv$, then there are sequences $(u_n)_{n\in\nn}$ and
   $(v_n)_{n\in\nn}$ of elements of $(\Om AS)^1$ respectively
   converging to $u$ and $v$ and such that $w_n=u_nv_n$ for every
   $n\in\nn$.
\end{lemma}

We may now proceed to establish Theorem~\ref{t:surjectivity}.

\begin{proof}[Proof of Theorem~\ref{t:surjectivity}]
  By Corollary~\ref{c:chain-of-idempotents-in-the-inverse-limit}, we
  may consider a sequence $(e_k)_{k\in\nn}$ of idempotents such that
  $e_k\in J_{\pv S}(\bsigma^{(k)})\cap \img_{\pv S}(\bsigma^{(k)})$
  and $e_k=\prov S\sigma_{k,l}(e_l)$ for every $k,l\in\nn$ such that
  $k\leq l$. Set $z_k=\pli(e_k)$ for each $k\in\nn$. Note that $z_k\in
  X(\bsigma^{(k)})$ by
  Proposition~\ref{p:parametrization-of-H-classes}. Since $\prov
  S\sigma_{0,k}(e_k)=e_0$, we have
    \begin{equation*}
        \sigma_{0,k}(z_k)=z_0.  
    \end{equation*}

  Denote by $H$ the maximal subgroup of $\Om {A_0}S$ containing $e_0$.
  We claim that $H\subseteq\img_{\pv S}(\bsigma)$, from which the
  result follows by Proposition~\ref{p:saturation-equivalence}.
  
  Fix $s\in H$. By Proposition~\ref{p:infinite-closure-J-class}, we
  may write $s$ as a limit
  \begin{equation*}
    s = \lim_{n\to\infty} t_n, \quad t_n\in L(\bsigma).
  \end{equation*}
  Since $s=e_0se_0$, it follows from Lemma~\ref{l:metric-open-mapping}
  that we may choose for every $n\in\nn$ a factorization $t_n =
  p_ns_nq_n$ in $(A_0)^*$ such that
  \begin{equation*}
    e_0=\lim_{n\to\infty}p_n=\lim_{n\to\infty}q_n,\quad s=\lim_{n\to\infty}s_n,
  \end{equation*}
  with the limits being taken in $(\Om {A_0}S)^1$.

  Take an arbitrary positive integer $k$. By
  Lemma~\ref{l:composition-of-recognizable-homomorphisms}, the
  composite $\sigma_{0,k}$ is recognizable in $X(\bsigma^{(k)})$.
  Since $X(\bsigma^{(k)})$ is minimal, it is generated by the
  sequence~$(z_k)_k$. Hence, $\sigma_{0,k}$ is recognizable in Mossé's
  sense for $z_k$ by Proposition~\ref{p:mosse-sense}; let $\ell_k$ be
  the corresponding constant of recognizability. Because
  $\pli(H)=\pli(e_0)=z_0$ and for every $x\in (A_0)^+$ the sets of the
  form $x(\Om {A_0}S)^1$ and $(\Om {A_0}S)^1x$ are clopen subsets of
  $\Om{A_0}S$, there is a positive integer $N_k$ such that the
  following relations hold whenever $n> N_k$:
  \begin{equation}\label{eq:pref-suff}
    z_0[-\ell_k,0)
    \geq_{\green{L}} p_n,\quad z_0[0,\ell_k)
    \geq_{\green{R}} s_n,\quad z_0[-\ell_k,0)
    \geq_{\green{L}} s_n,\quad z_0[0,\ell_k)
    \geq_{\green{R}}q_n.
  \end{equation}
  Take $n>N_k$. Since $p_ns_nq_n\in L(\bsigma)$ and the minimal shift
  space $X(\bsigma)$ is generated by $z_0$, there are integers
  $j_1<j_2<j_3<j_4$ such that
  \begin{equation*}
    z_0[j_1,j_2)=p_n,\quad z_0[j_2,j_3)=s_n,\quad z_0[j_3,j_4)=q_n.
  \end{equation*}
  Consider the set $C = C_{\sigma_{0,k}}(z_k)$ of
  $\sigma_{0,k}$-cutting points of $z_k$, and bear in mind the
  equality~$\sigma_{0,k}(z_k)=z_0$. It follows from
  \eqref{eq:pref-suff} that
  \begin{equation*}
    z_0[j_2-\ell_k,j_2)
    =z_0[-\ell_k,0),\quad z_0[j_2,j_2+\ell_k)
    =z_0[0,\ell_k),
  \end{equation*}
  that is, $z_0[j_2-\ell_k,j_2+\ell_k)=z_0[-\ell_k,\ell_k)$, and so,
  since $0\in C$, by recognizability we conclude that $j_2\in C$.
  Similarly, we conclude that $j_3\in C$. Hence,
  as $s_n=z_0[j_2,j_3)$,
  we have $s_n\in \sigma_{0,k}(A_{k}^+)$. Since
  $n$ is an arbitrary integer greater than $N$,
  it follows that $s\in \img(\sigma_{0,k}^{\pv S})$.
  As $k$ is arbitrary, this shows that $s\in\img_{\pv S}(\bsigma)$.
  This establishes the claim and completes the proof of the theorem.
\end{proof}
  
\subsection{Sufficient conditions for recognizability}

The purpose of this section is to give conditions under which
saturating directive sequences are recognizable. In other words, we
are proposing a partial converse to Theorem~\ref{t:surjectivity}.

\begin{theorem}\label{t:a-sort-of-converse-of-surjectivity-theorem}
  Let $\bsigma$ be an encoding directive sequence. Let \pv V be a
  pseudovariety of semigroups containing $\pv{\loc Sl}$. Assume that
  the language $\img(\sigma_{0,n})$ is $\pv V$-recognizable for
  every~$n\in\nn$. If $\bsigma$ is $\pv{V}$-saturating and eventually
  recognizable, then $\bsigma$ is recognizable.
\end{theorem}

\begin{proof}
  Using an argument of \emph{reductio ad absurdum}, let us suppose
  that the hypotheses in the statement holds but $\bsigma$ is not
  recognizable.

  Let $m\in\nn$ be such that $\bsigma^{(m)}$ is recognizable.
  Take $n\in\nn$ such that $n>m$.
  By the assumption that $\bsigma$ is not recognizable and by
  Lemma~\ref{l:composition-of-recognizable-homomorphisms}, the
  homomorphism $\sigma_{0,{n}}$ is not recognizable in
  $X(\bsigma^{(n)})$. Therefore, letting $A=A_0$, there is an
  element in $A^\zz$ with two distinct centered
  $\sigma_{0,{n}}$-representations in $X(\bsigma^{({n})})$,
  which means that there are $x_n,z_n\in X(\bsigma^{({n})})$ and
  $\ell_n\in\nn$ such that
  \begin{equation}\label{eq:distinct-centered-representations}
    \sigma_{0,n}(x_n)=\shift^{\ell_n}\sigma_{0,{n}}(\shift^{n}(z_n))
  \end{equation}
  with $0\leq \ell_n<|\sigma_{0,n}(z_n[n])|$ and
  $(0,x_n)\neq(\ell_n,\shift^{n}(z_n))$.

  Despite the statement mentioning the pseudovariety $\pv V$, for most
  of the proof we work with the pseudovariety $\pv S$ of all finite
  semigroups. By Proposition~\ref{p:parametrization-of-H-classes} we
  may take the unique idempotents $e_n$ and $f_n$ of $J_{\pv
    S}(\bsigma^{(n)})$ such that $\pli(e_n)=x_n$ and $\pli(f_n)=z_n$.
  Note that both $\prov S\sigma_{0,{n}}(e_n)$ and $\prov
  S\sigma_{0,{n}}(f_n)$ belong to $J_{\pv S}(\bsigma)$ and that the
  following equalities hold:
  \begin{equation}
    \label{eq:biinfinite-of-idemptotents-of-distinct-representations}
    \pli(\prov S\sigma_{0,{n}}(e_n))=\sigma_{0,{n}}(x_n)\qquad
    \text{and}
    \qquad
    \pli(\prov S\sigma_{0,{n}}(f_n))=\sigma_{0,{n}}(z_n).
  \end{equation}
  Denote by $r_n$ the prefix of length $\ell_n$ of
  $\sigma_{0,{n}}(z_n[n])$,
  cf.~Figure~\ref{f:a-sort-of-converse-of-surjectivity-theorem-1}.
  Let
  \begin{equation*}
    p_n'=z_n{[}0,n{)},\  p_n=\sigma_{0,n}(p_n')r_n.
  \end{equation*}
  Then $p_n$ is a prefix of $\sigma_{0,n}(z_n{[}0,n{]})$ and suffix of
  $\sigma_{0,n}(x_n{[}-n',0{)})$ for some $n'\in\nn$, as illustrated
  by Figure~\ref{f:a-sort-of-converse-of-surjectivity-theorem-1}.
  Hence, $p_n$ is a prefix of the idempotent $\sigma_{0,n}(f_n)$, and
  a suffix of the idempotent $\sigma_{0,n}(e_n)$, in view
  of~\eqref{eq:biinfinite-of-idemptotents-of-distinct-representations}.
  Moreover, we have
  \begin{align*}
    \shift^{|p_n|}(\pli(\prov S\sigma_{0,n}(f_n)))
    &=\shift^{|r_n|}\shift^{|\sigma_{0,n}(z_n{[}0,n{)})|}(\sigma_{0,{n}}(z_n))\\
    & =\shift^{\ell_n}\sigma_{0,{n}}(\shift^{n}(z_n))\\
    & =\sigma_{0,n}(x_n)\\
    &=\pli(\prov S\sigma_{0,{n}} (e_n)),
  \end{align*}
  with the second last equality holding
  by~\eqref{eq:distinct-centered-representations}. Therefore, the
  equality
  \begin{equation}\label{eq:conjugation-with-small-word}
    p_n\prov S\sigma_{0,{n}}(e_n)=\prov S\sigma_{0,{n}}(f_n)p_n
  \end{equation}
  holds by Proposition~\ref{p:shift-of-idempotents}.

  \begin{figure}[ht]
    \begin{tikzpicture}[every node/.style={font=\footnotesize}]
      \node[cutting,fill=black] (x0) at (0,0) {} ;
      \node[cutting,right = 2 of x0.center, anchor=center] (x1) {} ;
      \node[above left = .55 and 1.2 of x0.center,anchor=center] (xdotsl) {$\cdots$} ;
      \node[above right = .55 and 1.2 of x1.center,anchor=center] (xdotsr) {$\cdots$} ;
      \node[cutting, left = 1.6 of x0.center, anchor=center] (z0) {} ;
      \node[cutting, right = 2.3 of z0.center, anchor=center] (z1) {} ;
      \node[below left = .55 and 1.2 of z0.center,anchor=center] (zdotsl) {$\cdots$} ;
      \node[below right = .55 and 1.2 of z1.center,anchor=center] (zdotsr) {$\cdots$} ;
      \node[left = 2 of z0.center,anchor=center] (ldots) {$\cdots$} ;
      \node[right = 2 of x1.center,anchor=center] (rdots) {$\cdots$} ;
      \node[above = 1 of z0.center] (lengthl) {} ;
      \node[above = 1 of x0.center] (lengthr) {} ;
      \begin{scope}[on background layer]
        \draw[-,thick] (ldots) to (z0) ;
        \draw[-,thick] (z0) to node [below] {$r_n$} (x0) ;
        \draw[-,thick] (x0) to (rdots) ;
      \end{scope}
      \draw[-] (xdotsl.east) to [bend left] (x0) ;
      \draw[-] (x1) to [bend left] (xdotsr.west) ;
      \draw[-] (x0) to [bend left=60] node [above] {$\sigma_{0,{n}}(x_n[0])$} (x1) ;
      \draw[-] (zdotsl.east) to [bend right] (z0) ;
      \draw[-] (z1) to [bend right] (zdotsr) ;
      \draw[-] (z0) to [bend right=60] node [below] {$\sigma_{0,{n}}(z_n[n])$} (z1) ;
      \draw[|-|,thick] ([xshift=-.3] lengthl.center) to node [above] {length $\ell_n$} ([xshift=.3] lengthr.center) ;
      \draw[-,densely dotted] (z0) to (lengthl.center) ;
      \draw[-,densely dotted] (x0) to (lengthr.center) ;
    \end{tikzpicture}
    \caption{The bi-infinite word
      $\sigma_{0,{n}}(x_n)=\shift^{\ell_n}\sigma_{0,{n}}(\shift^{n}(z_n))$, where
      the black cutting point marks the boundary of its left and
      right infinite parts}
    \label{f:a-sort-of-converse-of-surjectivity-theorem-1}
  \end{figure}
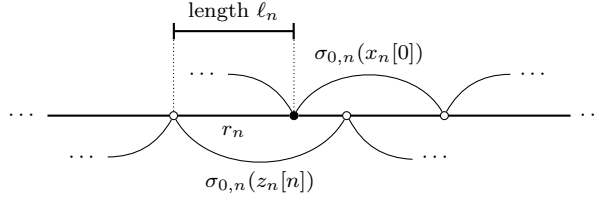
  
  Suppose that $\ell_n=0$, that is to say $p_n=\sigma_{0,n}(p_n')$.
  Then the equality~\eqref{eq:conjugation-with-small-word} becomes
  \begin{equation*}
    \prov S\sigma_{0,n}(p_n'e_n)=\prov S\sigma_{0,n}(f_np_n').
  \end{equation*}
  As by hypothesis $\bsigma$ is an encoding directive sequence, the
  homomorphism $\sigma_{0,n}$ is injective, and so $\sigma_{0,n}^{\pv
    S}$ is injective by
  Theorem~\ref{t:sufficient-conditions-for-injectivity}. It follows
  that $p_n'e_n=f_np_n'$, thus $x_n=\shift^{n}(z_n)$ by
  Proposition~\ref{p:shift-of-idempotents}. But this contradicts
  $(0,x_n)\neq (\ell_n,\shift^{n}(z_n))$. Therefore, $\ell_n$ must be
  positive.
  
  By compactness of $\Om AS$, the sequence $(p_n,\prov
  S\sigma_{0,{n}}(e_n),\prov S\sigma_{0,{n}}(f_n))_{n> m}$ has some
  subsequence converging in $(\Om AS)^3$ to a triple $(p,e,f)$. Bear
  in mind that the pseudowords $e,f$ are idempotents in $\img_{\pv
    S}(\bsigma)\cap J_{\pv{S}}(\bsigma)$: indeed, we have $e,f\in
  \img_{\pv S}(\bsigma)$ by
  Lemma~\ref{l:V-image-as-a-set-of-cluster-points}, and $\prov
  S\sigma_{0,{n}}(e_n),\prov S\sigma_{0,{n}}(f_n)\in
  J_{\pv{S}}(\bsigma)$ by
  Corollary~\ref{c:at-the-limit-properties-of-lambda}\ref{item:corollary-at-the-limit-properties-of-lambda-2},
  and thus taking limits $e,f\in J_{\pv{S}}(\bsigma)$. Note that
  $\lim|p_n|=\infty$, and so the pseudoword $p$ has infinite length.
  As $p_n\in L(\bsigma)$ for every $n>m$, it follows that $p\in J_{\pv
    S}(\bsigma)$. Since $p_{n}$ is a prefix of the idempotent $\prov
  S\sigma_{0,n}(f_{n})$ for each $n>m$, and the relation
  $\leq_{\green{R}}$ is closed in $\Om AS$, we know that $p$ is a
  prefix of~$f$. Similarly, $p$ is a suffix of~$e$. By~stability, we
  obtain $f\green{R}p\green{L}e$. It follows that $p_{\pv S,\pv
    V}(f)\green{R}p_{\pv S,\pv V}(p)\green{L}p_{\pv S,\pv V}(e)$. Note
  that the idempotents $p_{\pv S,\pv V}(f)$ and $p_{\pv S,\pv V}(e)$
  belong to $J_{\pv V}(\bsigma)\cap \img_{\pv V}(\bsigma)$ by
  Corollary~\ref{c:mirage} and
  Proposition~\ref{p:projecting-profinite-image}. Since $\bsigma$ is
  $\pv V$-saturating, we deduce that $p_{\pv S,\pv V}(p)\in\img_{\pv
    V}(\bsigma)$ by Proposition~\ref{p:saturation-equivalence}.
  
  Set $B=A_m$. Since $\img_{\pv V}(\bsigma)\subseteq \img_{\pv
    V}(\prov V\sigma_{0,m})$, we have $p_{\pv S,\pv V}(p)\in\clos
  V{(\sigma_{0,m}(B^+))}$. Because $\sigma_{0,m}(B^+)$ is assumed to
  be $\pv V$-recognizable, the set $\clos V{(\sigma_{0,m}(B^+))}$ is
  clopen by Theorem~\ref{t:V-recognizability}. Note also that, since
  the continuous mapping $p_{\pv S,\pv V}$ restricts to the identity
  on $A^+$, the pseudoword $p_{\pv S,\pv V}(p)$ is a cluster point of
  the sequence $(p_n)_n$ in the space $\Om AV$. Hence, there is
  $k\in\nn$ such that $k>m$ and $p_k\in \sigma_{0,m}(B^+)$. Take $q\in
  B^+$ such that $p_k=\sigma_{0,m}(q)$. The
  equality~\eqref{eq:conjugation-with-small-word} then entails
  \begin{equation*}
    \prov S\sigma_{0,m}\Bigl(q\cdot \prov S\sigma_{m,k}(e_{k})\Bigr)
    =
    \prov S\sigma_{0,m}\Bigl(\prov S\sigma_{m,k}(f_{k})\cdot q\Bigr).
  \end{equation*}
  But $\prov S\sigma_{0,m}$ is injective (by
  Theorem~\ref{t:sufficient-conditions-for-injectivity}), and so the
  equality $ q\cdot \prov S\sigma_{m,k}(e_{k}) = \prov
  S\sigma_{m,k}(f_{k})\cdot q$ holds. Since $\pli(\prov
  S\sigma_{m,k}(e_{k}))=\sigma_{m,k}(x_k)$ and $\pli(\prov
  S\sigma_{m,k}(f_{k}))=\sigma_{m,k}(z_k)$, we then deduce from
  Proposition~\ref{p:shift-of-idempotents} that
 \begin{equation*}
   \sigma_{m,k}(x_k)=\shift^{|q|}(\sigma_{m,k}(z_k))
 \end{equation*}
 and that $q$ is a nonempty prefix of a word of the form
 $\sigma_{m,k}(z_{k}{[}0,l{)})$, with $l>0$
 (see~Figure~\ref{f:a-sort-of-converse-of-surjectivity-theorem-2}).
 Hence, we may consider the integer
 \begin{equation*}
   l_0=\min\{l\in\nn :|\sigma_{m,k}(z_{k}{[}0,l{)})|\geq q\}.
 \end{equation*}
 Letting $d=|\sigma_{m,k}(z_{k}{[}0,l_0{)})|-q$ we see that
 $(d,\shift^{l_0-1}(z_{k}))$ is a centered
 $\sigma_{m,k}$-representation of $\sigma_{m,k}(x_k)$. Since
 $\bsigma^{(m)}$ is recognizable, we know that $\sigma_{m,k}$ is
 recognizable in $X(\bsigma^{(k)})$ by
 Lemma~\ref{l:composition-of-recognizable-homomorphisms}. Therefore,
 we must have $d=0$, thus $q=\sigma_{m,k}(z_{k}{[}0,l_0{)})$.
    
 \begin{figure}[ht]
   \begin{tikzpicture}[every node/.style={font=\footnotesize}]
     \node[cutting,fill=black] (x0) at (0,0) {} ;
     \node[above left = .55 and 1.2 of x0.center,anchor=center,label=above:{$\sigma_{m,k}(x_k[-1])$}] (xdotsl) {$\cdots$} ;
     \node[above right = .55 and 1.2 of x0.center,anchor=center,label=above:{$\sigma_{m,k}(x_k[0])$}] (xdotsr) {$\cdots$} ;
     \node[cutting, left = 3 of x0.center, anchor=center] (z0) {} ;
     \node[below left = .55 and 1.2 of z0.center,anchor=center,label=below:{$\sigma_{m,k}(z_{k}[-1])$}] (zdotsl) {$\cdots$} ;
     \node[below right = .55 and 1.2 of z0.center,anchor=center,label=below:{$\sigma_{m,k}(z_{k}[0])$}] (zdotsr) {$\cdots$} ;
     \node[left = 2.5 of z0.center,anchor=center] (ldots) {$\cdots$} ;
     \node[right = 2.5 of x0.center,anchor=center] (rdots) {$\cdots$} ;
     \begin{scope}[on background layer]
       \draw[-,thick] (ldots) to (rdots) ;
       \draw[-,densely dashed] ($(z0) + (0,1)$) to ($(z0) + (0,-1.5)$) ;
       \draw[-,densely dashed] ($(x0) + (0,1)$) to ($(x0) + (0,-1.5)$) ;
     \end{scope}
     \draw[-] (xdotsl.east) to [bend left] (x0) ;
     \draw[-] (x0) to [bend left] (xdotsr.west) ;
     \draw[-] (zdotsl.east) to [bend right] (z0) ;
     \draw[-] (z0) to [bend right] (zdotsr) ;
     \draw[thick, decorate, decoration = {calligraphic brace,mirror,raise=5pt,amplitude=7pt}] ($(z0) + (0,-1.5)$) to node [below=12pt] {$q$} ($(x0) + (0,-1.5)$) ;
   \end{tikzpicture}
   \caption{Location of $q$ in the infinite word $\pli(\prov S\sigma_{m,k}(e_k)) = \sigma_{m,k}(x_k)$}
   \label{f:a-sort-of-converse-of-surjectivity-theorem-2}
 \end{figure}
 
 We have therefore $\sigma_{0,k}(z_{k}{[}0,l_0{)})=\sigma_{0,m}(q)=p_k
 =\sigma_{0,k}(z_{k}{[}0,k{)})r_k$. In particular, we must have
 $l_0>k$, as $|r_k|=\ell_k\neq 0$. It follows that
 $\sigma_{0,k}(z_k[k])$ is a prefix of $r_k$. But this contradicts the
 fact that $|r_k|=\ell_k<|\sigma_{0,k}(z_k[k])|$ by choice of $r_k$
 and~$\ell_k$.

 This concludes the argument by \emph{reductio ad absurdum}, and
 therefore we proved that $\bsigma$ must be recognizable.
\end{proof}
  
In the special case where \pv{V} is of the form \pvo{H} for some
extension-closed pseudovariety of groups \pv{H}, the previous theorem
can be specialized as follows.

\begin{corollary}
  \label{c:an-equivalence-with-surjectivity}
  Let $\bsigma$ be an eventually recognizable primitive
  directive sequence. Let \pv H be an extension-closed pseudovariety
  of groups such that $\sigma_n$ is an \pvo{H}-encoding for every
  $n\in\nn$. Then the following conditions are equivalent:
    \begin{enumerate}
        \item $\bsigma$ is recognizable;
            \label{i:an-equivalence-with-surjectivity-1}
        \item $\bsigma$ is $\pv{S}$-saturating;
            \label{i:an-equivalence-with-surjectivity-2}
        \item $\bsigma$ is \pvo{H}-saturating.
            \label{i:an-equivalence-with-surjectivity-3}
    \end{enumerate}
\end{corollary}

The proof of the corollary requires the next lemma.

\begin{lemma}
  \label{l:image-under-encoding}
  Let $\pv{H}$ be an extension-closed pseudovariety of groups. Let
  $L\subseteq A^+$ be an $\pvo{H}$-recognizable language and
  $\sigma\from A^+\to B^+$ be an $\pvo{H}$-encoding. Then $\sigma(L)$
  is \pvo{H}-recognizable.
\end{lemma}

\begin{proof}
  Observe that $\sigma^{\pvo{H}}\from\Om A{\pvo{H}}\to\Om B{\pvo{H}}$
  is injective by
  Theorem~\ref{t:sufficient-conditions-for-injectivity}~\ref{item:sufficient-conditions-for-injectivity-2}.
  Moreover, its image is clopen by
  \cite[Corollary~5.7]{Almeida&Steinberg:2008}, which implies that it
  is an open mapping since it is a homeomorphism onto its image. By
  assumption, the closure of $L$ in $\Om A{\pvo{H}}$ is clopen. As the
  closure of $\sigma(L)$ in $\Om B{\pvo{H}}$ is equal to the image by
  $\sigma^{\pvo{H}}$ of the closure of $L$, the closure of $\sigma(L)$
  is clopen. Therefore, $\sigma(L)$ is $\pvo{H}$-recognizable by
  Theorem~\ref{t:V-recognizability}.
\end{proof}

\begin{remark}
  \label{r:alternative-proof-of-lemma-image-under-encoding-using-Pin-Sakarovitch}
  An alternative proof of this lemma is obtained by combining
  \cite[Proposition~4.3]{Pin&Sakarovitch:1985} with
  \cite[Theorem~3]{LeRest&LeRest:1980a}.
\end{remark}

\begin{proof}[Proof of
  Corollary~\ref{c:an-equivalence-with-surjectivity}]
  The implication
  $\ref{i:an-equivalence-with-surjectivity-1}
  \Rightarrow
  \ref{i:an-equivalence-with-surjectivity-2}$ is
  Theorem~\ref{t:surjectivity}, while the implication
  $\ref{i:an-equivalence-with-surjectivity-2}
  \Rightarrow
  \ref{i:an-equivalence-with-surjectivity-3}$ is given by
  Proposition~\ref{p:saturation-inheritance}. It remains to establish
  the implication
  $\ref{i:an-equivalence-with-surjectivity-3}
  \Rightarrow
  \ref{i:an-equivalence-with-surjectivity-1}$.

  Assume that \ref{i:an-equivalence-with-surjectivity-3} holds. By
  Lemma~\ref{l:image-under-encoding}, the composition of two
  $\pvo{H}$-encodings is again an $\pvo{H}$-encoding. Therefore, under
  our assumptions, the image of $\sigma_{0,n}$ is
  $\pvo{H}$-recognizable, for all $n\in\nn$. Since $\pv{\loc
    Sl}\subseteq\pv{A}\subseteq\pvo{H}$, we can apply
  Theorem~\ref{t:a-sort-of-converse-of-surjectivity-theorem} to
  conclude that \ref{i:an-equivalence-with-surjectivity-1} holds.
\end{proof}

The case $\pv{H} = \pv{I}$ in
Corollary~\ref{c:an-equivalence-with-surjectivity} is precisely the
pure case. Since $G_\pv{A}(\bsigma)$ is the trivial group, condition
\ref{i:an-equivalence-with-surjectivity-3} holds trivially in that
case. We deduce the following.

\begin{corollary}
  \label{c:surjectivity-pure-encoding}
  Let $\bsigma$ be an eventually recognizable primitive
  directive sequence. If $\bsigma$ is pure, then $\bsigma$ is
  recognizable.
\end{corollary}

The following result gives yet another sufficient condition for
recognizability.

\begin{theorem}
  \label{t:recurrent-encoding-made-implies-recognizable}
  Let $\bsigma$ be a bounded primitive directive sequence. If
  $\bsigma$ is eventually recognizable, recurrent, and encoding, then
  it is recognizable.
\end{theorem}

\begin{proof}
  By Theorem~\ref{t:surjectivity}, the assumption that $\bsigma$ is
  eventually recognizable entails that it is eventually \pv
  S-saturating. Theorem~\ref{t:recurrent-encoding-implies-saturating}
  then yields that $\bsigma$ is \pv S-saturating. Finally, as each
  $\sigma_n$ is an \pv S-encoding,
  Corollary~\ref{c:an-equivalence-with-surjectivity} shows that
  $\bsigma$ is recognizable.
\end{proof}

Although the statements of
Corollary~\ref{c:surjectivity-pure-encoding} and
Theorem~\ref{t:recurrent-encoding-made-implies-recognizable} concern
only symbolic dynamics, their proofs use the connection with profinite
semigroups in crucial ways. For instance, the proof of
Theorem~\ref{t:recurrent-encoding-made-implies-recognizable} relies
indirectly on the fact that closed subgroups of free profinite
semigroups are torsion-free~\cite[Theorem~1]{Rhodes&Steinberg:2008}
(needed in the proof of
Theorem~\ref{t:recurrent-encoding-implies-saturating}). This may
motivate a quest for proofs of purely dynamical and combinatorial
character for those results.

\section{The rank of
  \texorpdfstring{$\pv{V}$-Sch\"utzenberger}{V-Schutzenberger} groups}
\label{sec:rank}

A finitely generated profinite semigroup $S$ is said to have
\emph{rank} $k$, if $k$ is the smallest positive integer $n$ such that
$S$ is $n$-generated, as a profinite semigroup. In this section, we
investigate bounds on the rank of the Sch\"utzenberger groups of a
primitive directive sequence with finite alphabet rank. We provide
sharp bounds on the rank of Sch\"utzenberger groups as a function of
the alphabet rank under very general conditions.

The section is divided into three subsections. In the first one, we
gather the necessary results on semigroups; the second one presents
relationships between different notions of rank and depends on several
of our main results that appear in earlier sections; the third serves
to establish sharpness of some of the inequalities of the second one.

\subsection{Preliminaries on semigroup theory}
\label{sec:prelim-sgp-theory}

For the reader's convenience, we gather here the results from
semigroup theory that intervene in the remaining subsections. Most of
them are taken from the literature.

We start by recalling a very useful result due to Miller and Clifford
\cite[Theorem~3]{Miller&Clifford:1956}.

\begin{theorem}
  \label{t:Miller-Clifford}
  If $s\green{L}t\green{R}u$ in a semigroup $S$, then the
  $\green{H}$-class $H_t$ is a group if and only if
  $s\green{R}su\green{L}u$.
\end{theorem}

The next proposition is also useful in
Subsection~\ref{sec:rank-upper-bounds}. It is a
generalization of \cite[Lemma~2]{Hunter:1988}.

\begin{proposition}
  \label{p:top-principal-factor}
  Let $S$ be a compact semigroup which is generated by a 
  $\green{J}$-class $J$ with only finitely many $\green{H}$-classes.
  Then $J$ is clopen.
\end{proposition}

\begin{proof}
  One can easily show by induction on~$r$ that, given
  $t_1,\ldots,t_r\in J$, the product $t_1t_2\cdots t_r$ belongs to~$J$
  if and only if so do each of the products $t_it_{i+1}$
  ($i=1,\ldots,r-1$). Indeed, by induction, it suffices to consider
  the case case $r=3$, so suppose that $t_1,t_2,t_3,t_1t_2,t_2t_3\in
  J$. By stability, we must have $t_1t_2\green{L}t_2$. Since the
  relation $\green{L}$ is stable under multiplication on the right, we
  deduce that $t_1t_2t_3\green{L}t_2t_3$ so that, as $t_2t_3$ is
  assumed to belong to~$J$, so does $t_1t_2t_3$.

  Moreover, by stability and Theorem~\ref{t:Miller-Clifford}, whether
  the product $t_it_{i+1}$ belongs to~$J$ depends only on whether
  $L_{t_i}\cap R_{t_{i+1}}$ is a group.

  Hence, every product
  $s$ of elements of~$J$ that does not belong to~$J$ has a
  \emph{dropping factorization}, in the sense that it admits a
  factorization of the form $s=s'tus''$ with $s',s''\in S^1$, $t,u\in
  J$, and $tu\notin J$. Since $J$ is closed and contains only finitely
  many $\green{H}$-classes, which are themselves closed, by
  compactness the set of elements with dropping factorizations is
  closed and must coincide with $S\setminus J$, which completes the
  proof.
\end{proof}

If a semigroup $S$ has a zero element $0$, then $J_0=\{0\}$; if
$S\setminus\{0\}$ is the only other $\green{J}$-class and the
multiplication is not constant, then we say that $S$ is
\emph{0-simple}. As for completely simple semigroups, \emph{completely
  0-simple semigroups} are the stable 0-simple semigroups. By a
theorem of Rees, the structure of a completely 0-simple semigroup can
be described in terms of the structure of one of its maximal nonzero
subgroups $G$ and a matrix with entries in $G\cup\{0\}$ that basically
describes how the idempotents multiply.

More precisely, given sets $I$ and $\Lambda$, a group $G$, and a
matrix $P:\Lambda\times I\to G\cup\{0\}$ into the group $G$ with a
zero added, the set
\begin{displaymath}
  \mathcal{M}^0(G,I,\Lambda,P)=I\times G\times\Lambda\cup\{0\}
\end{displaymath}
is a semigroup under the operation defined by
\begin{equation}
  \label{eq:Rees-formula}
  (i,g,\lambda)\,(j,h,\mu)=
  (i,gP(\lambda,j)h,\mu) \text{ whenever } P(\lambda,j)\ne0,
\end{equation}
all other products being set to~$0$. If the matrix $P$ has no zero
entries, then the case of the formula (\ref{eq:Rees-formula}) always
holds and $\mathcal{M}(G,I,\Lambda,P)=I\times G\times\Lambda$ is a
subsemigroup. Both such semigroups are called \emph{Rees matrix
  semigroups}. Then, the matrix $P$ is said to be \emph{normalized
  along row $\lambda_0$ and column $i_0$} if
$P(\lambda_0,i)=P(\lambda,i_0)=e$, where $e$ is the idempotent of~$G$.
Each subset $\{i\}\times G\times\{\lambda\}$ is an $\green{H}$-class
and it is a group isomorphic to~$G$ if and only if $P(\lambda,i)\ne0$.
The following key structure theorem is due to Rees~\cite[Theorems~2.92 and~2.93]{Rees:1940}.

\begin{theorem}
  \label{t:Rees}
  The following hold for arbitrary semigroups.
  \begin{enumerate}
  \item\label{item:Ress-1} A semigroup is completely 0-simple
    semigroup if and only if it is isomorphic to a Rees matrix
    semigroup $\mathcal{M}^0(G,I,\Lambda,P)$ where the matrix $P$ has
    at least one nonzero entry in each row and in each column.
  \item\label{item:Rees-2} A semigroup is completely simple if and
    only if it is isomorphic to a Rees matrix semigroup
    $\mathcal{M}(G,I,\Lambda,P)$. Moreover, in such a representation,
    the matrix $P$ may be normalized along any chosen row and column.
  \end{enumerate}
\end{theorem}

In this section, we are interested in Rees matrix semigroups
$\mathcal{M}(G,I,\Lambda,P)$ where $I=\Lambda=\{1,\ldots,n\}$ for a
positive integer, so that $P$ is a square matrix with entries in $G$,
simply denoted $\mathcal{M}(G,n,n,P)$. In this case, it is common to
denote the matrix $P$ as $(p_{\lambda,i})_{\lambda,i=1,\ldots,n}$,
where $p_{\lambda,i}=P(\lambda,i)$. The matrix $P$ is usually taken to
be normalized along the first row and column. Sometimes, we will use
additive notation when the group $G$ is Abelian, denoting the
idempotent by $0$. This should not lead to confusion as we will only
need to do so when the matrix has all its entries in $G$.

For a pseudovariety of groups $\pv H$, let $\pv{CS(H)}$ be the
pseudovariety $\pv{CS}\cap\pvo H$. We are particularly interested in
the case of $\pv{CS}(\pv{Ab}_p)$, where $\pv{Ab}_p$ is the
pseudovariety of all finite elementary Abelian $p$-groups for a prime
$p$. Let
\begin{displaymath}
  K_p = \mathcal{M}\Bigl(\mathbb{Z}/p\mathbb{Z},n,n,
    \left(
      \begin{smallmatrix}
        0 & 0\\
        0 & 1
      \end{smallmatrix}
    \right)
  \Bigr)
\end{displaymath}
Note that the product of idempotents $(1,0,2)(2,0,1)=(1,1,1)$ is not
an idempotent, whence it is a generator of its (group)
$\green{H}$-class. Semigroups in which the product of two idempotents
is always idempotent are said to be \emph{orthodox}. The orthodox
finite simple semigroups are also known in the literature as
\emph{rectangular groups}. They are direct products of groups and
aperiodic simple semigroups, where the latter are also known as
\emph{rectangular bands}; this decomposition extends to
compact semigroups (see, for instance, \cite{Howie:1995} and \cite[Appendix A]{Rhodes&Steinberg:2009qt}). In particular, every maximal subgroup $G$ of a
(compact) orthodox simple semigroup $S$ is a (continuous) homomorphic image of $S$.

The following result is a simple application of the main theorem
in~\cite{Rasin:1979}.

\begin{theorem}
  \label{t:Rasin}
  Let $\pv V$ be a pseudovariety of simple semigroups. Then $\pv V$ contains a non-orthodox element if and only if it contains $K_p$ for some prime~$p$.
  The pseudovariety $\pv{CS}(\pv{Ab}_p)$ is the smallest containing
  $K_p$.
\end{theorem}

The Rees matrix representation theorem (Theorem~\ref{t:Rees}) extends to
finitely generated profinite (0-)simple semigroups $S$, where the
group $G$ involved is profinite, the cardinality of each of the sets
$I$ and $\Lambda$ is at most the rank of~$S$, and the topology is that
of the product $I\times G\times\Lambda$, where the first and third
factors are discrete, and, in the 0-simple case, the zero is an isolated point~\cite[Remark A.4.19]{Rhodes&Steinberg:2009qt}. The structure of $\Om n{CS(H)}$ is described in
the following theorem whose proof in the case of $\pv H=\pv G$ is
given in~\cite[Theorem~3.3]{Almeida:1991d} and extends to the case of an arbitrary
pseudovariety of groups $\pv H$.

\begin{theorem}
  \label{t:free-CS(H)}
  Let $\pv H$ be a pseudovariety of groups. Then the profinite
  semigroup $\mathcal{M}(\Om{n+(n-1)^2}H,n,n,P)$ is freely generated
  by the elements $(i,g_i,i)$ ($i=1,\ldots,n$) in the pseudovariety
  $\pv{CS(H)}$, where $P$ is normalized along the first row and column
  and the remaining entries are given by $P(i,j)=g_{i,j}$, and where the elements
  $g_i$ ($i=1,\ldots,n$) and $g_{i,j}$ ($i,j=2,\ldots,n$) are free
  generators of $\Om{n+(n-1)^2}H$.
\end{theorem}

We are also interested in the following application of
Theorem~\ref{t:free-CS(H)}, which carries over to the profinite
context a result of Ru\v{s}kuc~\cite[Corollary~5.3]{Ruskuc:1994}. An
alternative proof can be obtained directly from Ru\v{s}kuc's result
using inverse limits.

\begin{corollary}
  \label{c:rank-max-subgroup-0-simple}
  Let $S$ be an $n$-generated 0-simple profinite semigroup. Then its
  maximal subgroups have rank at most $n^2-n+1$. If, moreover, $S$ is
  orthodox, then its maximal subgroups have rank at most $n$.
\end{corollary}

\begin{proof}
  Let $m=n^2-n+1=n+(n-1)^2$ and let $\mathcal{M}^0(G,I,\Lambda,P)$ be
  a Rees matrix representation of~$S$. As observed in the proof
  of~\cite[Corollary~5.3]{Ruskuc:1994}, changing to the idempotent
  of~$G$ all zero entries of $P$ we obtain a semigroup
  $S'=\mathcal{M}(G,I,\Lambda,P')$ which is still $n$-generated. As
  $S'$ is a pro-$\pv{CS}$ semigroup, it is a continuous homomorphic
  image of $\Om n{CS}$. From Theorem~\ref{t:free-CS(H)}, it follows
  that $G$ is a continuous homomorphic image of the group $\Om mG$
  (see~\cite[Lemma 3.7]{ACosta&Steinberg:2011}, the straightforward extension for compact semigroups
  of a classical result about finite semigroups~\cite[Lemma~4.6.10]{Rhodes&Steinberg:2009qt}). Hence $G$ is
  $m$-generated.  If $S$ is orthodox, then $S'$ is also orthodox, which yields that $G$ is a continuous homomorphic image of $S'$.
  Therefore, if $S$ is orthodox, then the rank of $G$ cannot exceed $n$.
\end{proof}

Given an ideal $I$ of a semigroup $S$, the corresponding \emph{Rees
  quotient} is the semigroup $S/I$ which is obtained from $S$ by
identifying all elements of~$I$. Indeed, this defines a congruence on
$S$ in which the congruence class $I$ becomes a zero, all other classes
being singleton sets. In case $I$ is empty, $S/I$ is equal to $S$.

If $S$ is a profinite semigroup and $I$ is a clopen ideal then $S/I$
is a compact zero-dimensional semigroup, whence it is also profinite
by Numakura's theorem \cite[Theorem~1]{Numakura:1957} and the natural quotient
mapping, sending each element to its congruence class, is a continuous
homomorphism.

\subsection{Upper bounds on the rank of \texorpdfstring{\pv
    V}{V}-Sch\"utzenberger groups}
\label{sec:rank-upper-bounds}

We derive from results in earlier sections upper bounds on the rank of
Sch\"utzenberger groups. The starting point is the following theorem
where both Corollary~\ref{c:letters-mapped-to-Js} and semigroup theory
play key roles.

\begin{theorem}
  \label{t:maximal-subgroup-in-image}
  Let $\bsigma$ be a primitive directive sequence with finite alphabet
  rank~$n$ and let $\pv V$ be a pseudovariety of semigroups containing
  $\pv{LSl}$.
  \begin{enumerate}
  \item The maximal subgroups of $J_{\pv V}(\bsigma)\cap\img_{\pv
      V}(\bsigma)$ are profinite groups of rank at most
    $n^2-n+1$.\label{item:maximal-subgroup-in-image-1}
  \item If $\bsigma$ has a left proper contraction, then $\img_{\pv
      V}(\bsigma)$ is a right simple profinite semigroup whose maximal
    subgroups have rank at
    most~$n$.\label{item:maximal-subgroup-in-image-2}
  \item If $\pv V\cap\pv{CS}$ consists of orthodox semigroups, then
    the maximal subgroups of $J_{\pv V}(\bsigma)\cap\img_{\pv
      V}(\bsigma)$ are profinite groups of rank at
    most~$n$.\label{item:maximal-subgroup-in-image-3}
  \end{enumerate}
\end{theorem}

\begin{proof}
  In all cases, we use the fact that, by
  Corollary~\ref{c:letters-mapped-to-Js}, the profinite semigroup
  $\img_\pv{V}(\bsigma)$ is generated by a subset of the regular
  $\green{J}$-class $J=J_{\pv V}(\bsigma)\cap\img_{\pv V}(\bsigma)$
  with at most $n$~elements. 
  
  \ref{item:maximal-subgroup-in-image-1} Note that $I=\img_{\pv
    V}(\bsigma)\setminus J$ is an ideal, which may be empty. By
  stability, $J$ contains only finitely many $\green{L}$-classes and
  $\green{R}$-classes, as each of them must contain one of the
  generators. By Proposition~\ref{p:top-principal-factor}, $J$ is a
  clopen set in $\img_{\pv V}(\bsigma)$. Its complement $I$ is,
  therefore, a clopen ideal. As observed at the end of
  Subsection~\ref{sec:prelim-sgp-theory}, the Rees quotient
  $S=\img_{\pv V}(\bsigma)/I$ is a profinite semigroup. Since $J$ is a
  regular $\green{J}$-class, $S$ is $n$-generated and either simple or
  0-simple. The natural quotient mapping $\img_{\pv V}(\bsigma)\to S$
  is a continuous homomorphism which is injective on $J$. Hence, it
  suffices to show that the nonzero maximal subgroups $G$ of~$S$ are
  $(n^2-n+1)$-generated and this follows from
  Corollary~\ref{c:rank-max-subgroup-0-simple}.

  \ref{item:maximal-subgroup-in-image-2} If $\bsigma$ has a
  left proper contraction, then by
  Theorem~\ref{t:right-proper-in-an-R-class} the profinite semigroup $\img_{\pv V}(\bsigma)$ is right simple, thus orthodox.
  In particular, each maximal subgroup of $\img_{\pv V}(\bsigma)$
  is a continuous homomorphic image of $\img_{\pv V}(\bsigma)$,
  whence its rank is at most~$n$.

  \ref{item:maximal-subgroup-in-image-3} Proceeding as in the proof of
  \ref{item:maximal-subgroup-in-image-1}, it suffices to invoke the
  special case of Corollary~\ref{c:rank-max-subgroup-0-simple} at the
  end of the argument.
\end{proof}

Adding the saturating hypothesis, we obtain the following corollary.

\begin{corollary}
  \label{c:upper-bound-saturating}
  Let $\bsigma$ be a primitive directive sequence with finite alphabet
  rank~$n$. Let \pv{V} be a pseudovariety of semigroups containing
  \pv{\loc Sl}. Then,
  the following properties hold if $\bsigma$ is $\pv{V}$-saturating:
  \begin{enumerate}
  \item\label{item:upper-bound-saturating-1} The rank of the
    Sch\"utzenberger group $G_\pv{V}(\bsigma)$ is at most $n^2-n+1$.
  \item\label{item:upper-bound-saturating-2} If $\bsigma$ has a left
    proper contraction, then the rank of $G_\pv{V}(\bsigma)$ is at
    most $n$.
  \item\label{item:upper-bound-saturating-3} If $\pv V\cap\pv{CS}$
    consists of orthodox semigroups, then the rank of
    $G_\pv{V}(\bsigma)$ is at most~$n$.
  \end{enumerate}
  In particular, these properties hold if $\bsigma$ is recognizable.
\end{corollary}

\begin{proof}
  If $\bsigma$ is recognizable, then it is $\pv V$-saturating by Theorem~\ref{t:surjectivity} and Proposition~\ref{p:saturation-inheritance}.
  Hence, indeed, it suffices to assume that $\bsigma$ is $\pv V$-saturating.
  In that case, by definition of saturation, the profinite group $G_\pv{V}(\bsigma)$ is (isomorphic to) a maximal subgroup of $J_{\pv V}(\bsigma)\cap\img_{\pv V}(\bsigma)$. Therefore, it suffices to invoke Theorem~\ref{t:maximal-subgroup-in-image}.
 \end{proof}

We next show how Corollary~\ref{c:upper-bound-saturating} applies to
the important class of minimal shift spaces of \emph{finite
  topological rank}. A minimal shift space is said to be of finite
topological rank when it can be represented by a Bratteli-Vershik
diagram with a uniformly bounded number of vertices per level; and if
the least such bound among all such representations is $n$, then it is
said to have topological rank $n$; see~\cite[Chapter
6]{Durand&Perrin:2022} for details. A minimal shift space $X$ has
topological rank at most~$n$ if and only if it is topologically
conjugate to $X(\bsigma)$ for some proper, recognizable, primitive
directive sequence $\bsigma$ of alphabet rank at most~$n$ (this result
is from~\cite{Donoso&Durand&Maass&Petite:2021}, as attributed
in~\cite[Theorem~1.1]{Espinoza:2022}).

\begin{theorem}
  \label{t:finite-top-rank-implies-finitely-generated-Schutz}
  Let $X$ be a minimal shift space of finite topological rank $n$.
  Then, for every pseudovariety \pv{V} containing \pv{\loc Sl}, the
  Sch\"utzenberger group $G_\pv{V}(X)$ is a profinite group of rank at
  most~$n$.
\end{theorem}

\begin{proof}
  By Corollary~\ref{c:projecting-the-J-class-of-X}, it suffices to
  establish the result for the case $\pv{V}=\pv{S}$.

  By~\cite[Proposition~4.6]{Donoso&Durand&Maass&Petite:2021}, there
  exists a proper, recognizable, primitive directive sequence
  $\bsigma$ with alphabet rank at most~$n$ and such that $X$~is
  topologically conjugate to $X(\bsigma)$. Since the
  $\pv{S}$-Sch\"utzenberger group of a minimal shift space is a
  topological conjugacy invariant by
  Theorem~\ref{t:invariance-under-flow-equivalence}, we have
  $G_\pv{S}(X)\isom G_\pv{S}(\bsigma)$. The result now follows
  immediately from
  Corollary~\ref{c:upper-bound-saturating}\ref{item:upper-bound-saturating-2}.
\end{proof}

In general, the determination of the topological rank is a hard problem, motivating the investigation of upper and lower bounds for it;
the rank of the dimension group of the shift space is a lower bound to which intensive study is dedicated~\cite{Durand&Perrin:2022}.

The next example illustrates one motivation for determining the rank of the Sch\"utzenberger group of a minimal shift space.

\begin{example}
  \label{eg:ptm-comparing-ranks}
      Consider the stable primitive directive sequence $\btau$ defined by the Prouhet-Thue-Morse substitution, cf.~Example~\ref{eg:Prouhet-Thue-Morse-stable}.
      The topological rank of $X(\btau)$ is~$3$, and its dimension group has rank $2$ (cf.~\cite[Table C.1]{Durand&Perrin:2022}). The profinite
      group $G_\pv{S}(\btau)$ has rank $3$, and it is not free~\cite{Almeida&ACosta:2013}.
  \end{example}

We do not know if the converse of
Theorem~\ref{t:finite-top-rank-implies-finitely-generated-Schutz} holds:

\begin{problem}
  \label{pb:finitely-generated-Schutz-vs-finite-top-rank}
  Let $X$ be a minimal shift space.
  \begin{enumerate}
  \item Suppose that $G_{\pv S}(X)$ is finitely generated. Does $X$
    necessarily have finite topological rank?
  \item Is it true that, if $\pv H$ is a nontrivial pseudovariety of
    groups such that $G_{\pvo H}(X)$ is finitely generated, then
    $G_{\pv S}(X)$ is finitely generated?
  \end{enumerate}
\end{problem}

\subsection{Lower bounds on the rank of \texorpdfstring{\pv
    V}{V}-Sch\"utzenberger groups}
\label{sec:rank-lower-bounds}

The purpose of this subsection is to determine under what conditions
the bounds on the rank of the Sch\"utzenberger group of an arbitrary
recognizable primitive directive sequence with finite alphabet rank are optimal. We
concentrate on the case of constant primitive directive sequences,
which provides enough diversity to reach our goal.

Let $\varphi$ be a primitive endomorphism of $A_n^+$ for an $n$-letter
ordered alphabet $A_n=\{a_1,\ldots,a_n\}$. We also consider the set
$A_n^k=\{w_1,\ldots,w_{n^k}\}$ of $k$-letter words over $A_n$ ordered
by the lexicographic order. The \emph{$k$-frequency matrix}
of~$\varphi$, denoted $F_k(\varphi)$ is the $n\times n^k$ integer
matrix whose $i,j$-entry is the number $|\varphi(a_i)|_{w_j}$ of times
that $w_j$ occurs as a factor of the word $\varphi(a_i)$. In
particular, $F_1(\varphi)$ is the usual incidence matrix of~$\varphi$.
The matrix $F_2(\varphi)$ plays a key role in the sequel.

We also consider the matrix $T_2(\varphi)$ which accounts for extra
factors of length~2 when we concatenate words of the form
$\varphi(a_i)$. Specifically, $T_2(\varphi)$ is the $n^2\times n^2$
integer matrix whose $i,j$-entry is 1 if, when $w_i=a_ra_s$, then
$w_j=a_ta_u$, where $a_t$ is the last letter of~$\varphi(a_r)$ and
$a_u$ is the first letter of $\varphi(a_s)$; all other entries are
zero. Note that $T_2(\varphi)$ is a $0,1$ row monomial matrix, in the
sense that it has exactly one entry 1 in each row.
Since there are exactly $n^2$ nonzero entries in $T_2(\varphi)$,
$T_2(\varphi)$ is a permutation matrix if and only if has no zero
columns, that is, every letter from $A_n$ appears as the last letter
of $\varphi$ of some letter, and the first letter of $\varphi$ of some
letter; these conditions define properties of~$\varphi$ which are
called, respectively, \emph{right permutative} and \emph{left
  permutative}
in~\cite[Subsection~3.1]{Berthe&Steiner&Thuswaldner&Yassawi:2019}.

Finally, we consider the $m \times m$ integer matrix defined in block
form by
\begin{displaymath}
  \mathbb{I}_2(\varphi)=
  \begin{pmatrix}
    F_1(\varphi) & F_2(\varphi) \\
    \mathbf{0} & T_2(\varphi)
  \end{pmatrix}
\end{displaymath}
where $m=n+n^2$ and $\mathbf{0}$ is the $n^2\times n$ zero matrix.

\begin{example}
  \label{eg:sigma}
  Given $n\ge 2$, we let $\sigma_n$ be the endomorphism of~$A_n$
  defined by $\sigma_n(a_1)=a_1a_2\ldots a_n$ and
  $\sigma_n(a_i)=a_ia_i\ldots a_na_1\ldots a_{i-1}$ for
  $i=2,\ldots,n$. The matrix $F_1(\sigma_n)$ has 2 in the main
  diagonal entries except the first and 1 elsewhere. For instance,
  here is the matrix $\mathbb{I}_2(\sigma_3)$, highlighting both its
  block structure and that of its submatrix $T_2(\sigma_3)$:
  \begin{displaymath}
    \setcounter{MaxMatrixCols}{12}
    \mathbb{I}_2(\sigma_3) =
    \begin{pNiceArray}{ccc|ccccccccc}[margin]
      \Block[fill=[gray]{0.9},rounded-corners]{3-3}{}
      1 & 1 & 1 & 0 & 1 & 0 & 0 & 0 & 1 & 0 & 0 & 0 \\
      1 & 2 & 1 & 0 & 0 & 0 & 0 & 1 & 1 & 1 & 0 & 0 \\
      1 & 1 & 2 & 0 & 1 & 0 & 0 & 0 & 0 & 1 & 0 & 1 \\
      \Hline
      0 & 0 & 0 & 0 & 0 & 0 & 0 & 0 & 0 &
      \Block[fill=[gray]{0.9},rounded-corners]{3-3}{}
      1 & 0 & 0 \\
      0 & 0 & 0 & 0 & 0 & 0 & 0 & 0 & 0 & 0 & 1 & 0 \\
      0 & 0 & 0 & 0 & 0 & 0 & 0 & 0 & 0 & 0 & 0 & 1 \\
      0 & 0 & 0 &
      \Block[fill=[gray]{0.9},rounded-corners]{3-3}{}
      1 & 0 & 0 & 0 & 0 & 0 & 0 & 0 & 0 \\
      0 & 0 & 0 & 0 & 1 & 0 & 0 & 0 & 0 & 0 & 0 & 0 \\
      0 & 0 & 0 & 0 & 0 & 1 & 0 & 0 & 0 & 0 & 0 & 0 \\
      0 & 0 & 0 & 0 & 0 & 0 &
      \Block[fill=[gray]{0.9},rounded-corners]{3-3}{}
      1 & 0 & 0 & 0 & 0 & 0 \\
      0 & 0 & 0 & 0 & 0 & 0 & 0 & 1 & 0 & 0 & 0 & 0 \\
      0 & 0 & 0 & 0 & 0 & 0 & 0 & 0 & 1 & 0 & 0 & 0 \\
    \end{pNiceArray}
  \end{displaymath}
  By performing row reductions, it is easy to see that $F_1(\sigma_n)$
  has determinant~1. On the other hand, the matrix $T_2(\sigma_n)$ is
  a permutation matrix which is the Kronecker product of the
  permutation matrix of the cycle $(n\,1\,2\,\ldots\,n-1)$ by the
  identity matrix in dimension~$n$. Its determinant is $\pm1$. Hence,
  the integer matrix $\mathbb{I}_2(\sigma_n)$ is invertible.
\end{example}

In general, since the determinant of the matrix
$\mathbb{I}_2(\varphi)$ is the product of the determinants of the
matrices $F_1(\varphi)$ and $T_2(\varphi)$, the integer matrix
$\mathbb{I}_2(\varphi)$ is invertible if and only if so is
$F_1(\varphi)$ and $T_2(\varphi)$ is a permutation matrix. The latter
condition means that every letter appears as the first letter
(respectively last letter) of some word $\varphi(a_i)$; note that this
condition is stronger than the constant directive sequence
$(\varphi,\varphi,\ldots)$ being stable.

As already alluded above, the purpose of the block $T_2(\varphi)$ is
to aid in counting the occurrences of two-letter factors. More
precisely, we have the following result.

\begin{lemma}
  \label{l:I2-homomorphism}
  The mapping $\mathbb{I}_2:\en(A_n^+)\to M_m(\mathbb{Z})$ is an
  anti-homomorphism from the monoid of endomorphisms of the free
  semigroup $A_n^+$ to the multiplicative semigroup of $m\times m$
  integer matrices.
\end{lemma}

\begin{proof}
  Let $\varphi$ and $\psi$ be two endomorphisms of $A_n^+$. It is well
  known that $F_1$ is an anti-homomorphism. Given
  $i\in\{1,\ldots,n^2\}$, let $w_i=a_ra_s$ and $w_j=a_ta_u$ where the
  1 entry in row $i$ is in column $j$. Then $a_t$ is the last letter
  of $\varphi(a_r)$ and $a_u$ is the first letter of $\varphi(a_s)$,
  so that the last letter $a_x$ of $\psi(\varphi(a_r))$ is the last
  letter of $\psi(a_t)$ and the first letter $a_y$ of
  $\psi(\varphi(a_s))$ is the first letter of $\psi(a_u)$. Now, when
  we take the inner product of the $i$th row of $T_2(\varphi)$ by the
  $k$th column of $T_2(\psi)$, the result is 1 if and only if the
  $k$th column where the 1 in the $j$th row is located, that is,
  $a_xa_y=w_k$. Hence, $T_2(\varphi)T_2(\psi)=T_2(\psi\circ\varphi)$
  and $T_2$ is also an anti-homomorphism.

  It remains to verify that
  \begin{equation}
    \label{eq:I2-homomorphism-1}
    F_2(\psi\circ\varphi)=F_1(\varphi)F_2(\psi)+F_2(\varphi)T_2(\psi).
  \end{equation}
  Let $i\in\{1,\ldots,n\}$ and $j\in\{1,\ldots,n^2\}$. Let
  $\varphi(a_i)=a_{i_1}a_{i_2}\ldots a_{i_z}$. Since $\psi$ is a
  homomorphism the occurrences of $w_j$ in $\psi(\varphi(a_i))$ are
  those that occur entirely within one of the factors $\psi(a_{i_t})$
  plus those that are made from the last letter of some
  $\psi(a_{i_t})$ followed by the first letter of $\psi(a_{i_{t+1}})$.
  Thus, we have the formula
  \begin{equation}
    \label{eq:I2-homomorphism-2}
    \bigl(F_2(\psi\circ\varphi)\bigr)_{i,j}
    =\left|\psi\bigl(\varphi(a_i)\bigr)\right|_{w_j}
    =\sum_{t=1}^z|\psi(a_{i_t})|_{w_j}
    +\sum_{t=1}^{z-1}(T_2(\psi))_{k(i_t,i_{t+1}),j}
  \end{equation}
  where $a_ra_s=w_{k(r,s)}$. To conclude the verification of the
  formula \eqref{eq:I2-homomorphism-1}, we just observe that the
  $(i,j)$-entry of the matrix on its right side is exactly the sum on
  the right side of \eqref{eq:I2-homomorphism-2}.
\end{proof}

We proceed to consider the special case of the pseudovariety
$\pv{CS}(\pv{Ab}_p)$ for a prime~$p$. For a positive integer $n$ and
$m=n+n^2$, the elementary Abelian $p$-group
$G=(\mathbb{Z}/p\mathbb{Z})^m$ is free in $\pv{Ab}_p$ with free
generating set the canonical basis of the group viewed as a vector
space over the field $\mathbb{Z}/p\mathbb{Z}$. For $1\le i\le n$, we
denote the $i$th basis vector by $v_i$ and, additionally, for $1\le
j\le n$, we let $v_{i,j}$ be the $ni+j$th basis vector.

\begin{proposition}
  \label{p:suficient-for-CSAbp-invertible}
  Let $\varphi$ be an endomorphism of the free semigroup $A_n^+$.
  Then the induced homomorphism $\varphi^{\pv{CS}(\pv{Ab}_p)}$ is
  an automorphism of the semigroup
  $\Om{A_n}{CS}(\pv{Ab}_p)$ if and only if the matrix
  $\mathbb{I}_2(\varphi)$ is invertible mod the prime $p$.
\end{proposition}

\begin{proof}
  Let us assume first that $\mathbb{I}_2(\varphi)$ is invertible mod
  the prime $p$. Let $m=n+n^2$ and consider the field
  $F=\mathbb{Z}/p\mathbb{Z}$ and the vector space $V=F^m$, below
  sometimes viewed as an additive group. The latter assumption holds
  when we define the Rees matrix semigroup $S =
  \mathcal{M}\left(V,n,n,P\right)$, where $P_{i,j}=v_{i,j}$. Let $T$
  be the subsemigroup generated by the elements $s_i=(i,v_i,i)$
  ($i=1,\ldots,n$).

  We compute the product $s_{i_1}s_{i_2}\cdots s_{i_r}$: 
  \begin{equation}
    \label{eq:product-formula}
    \left\{
      \begin{aligned}
        s_{i_1}s_{i_2}\cdots s_{i_r}
        &= (i_1,v,i_r)\\
        v
        &= \sum_{i=1}^n\lambda_iv_i+\sum_{i,j=1}^n\lambda_{i,j}v_{i,j}\\
        \lambda_i
        &= |a_{i_1}a_{i_2}\cdots a_{i_r}|_{a_i}
          \quad (1\le i\le n)\\
        \lambda_{i,j}
        &= |a_{i_1}a_{i_2}\cdots a_{i_r}|_{a_ia_j}
          \quad (1\le i,j \le n).
      \end{aligned}
    \right.
  \end{equation}
  Note that the sum of row $i$ of the matrix $(\lambda_{i,j})_{i,j}$
  is $\lambda_i-1$ if $i=i_r$ and $\lambda_i$ otherwise; similarly,
  the sum of column $i$ of the matrix $(\lambda_{i,j})_{i,j}$ is
  $\lambda_i-1$ if $i=i_1$ and $\lambda_i$ otherwise. For fixed $i_1$
  and $i_r$, for all products lying in the same $\green{H}$-class,
  this gives $2n-1$ independent linear equations; not $2n$ because if
  one knows the row sums and all but one column sums in a matrix, then
  the remaining column sum is completely determined. Thus, the row and
  column sum equations define an affine subspace $V_{i_1,i_r}$ of~$V$
  of dimension~$m-(2n-1)=n^2-n+1$. Hence, the maximal
  subgroups of~$T$ have cardinality $p^{n^2-n+1}$ and
  $|T|=n^2p^{n^2-n+1}$.

  Consider the homomorphism $h:\Om{A_n}{CS}(\pv{Ab}_p)\to T$ sending
  the $i$th free generator to $s_i$. As $h$ is onto and its domain
  also has cardinality $n^2p^{n^2-n+1}$ by Theorem~\ref{t:free-CS(H)},
  we conclude that $h$ is an isomorphism, and so $T$ is freely
  generated by the elements $s_1,\ldots,s_n$ in the pseudovariety
  $\pv{CS}(\pv{Ab}_p)$. Hence, $\varphi$ induces an endomorphism
  $\varphi'=\varphi^{\pv{CS}(\pv{Ab}_p)}$ of~$T$.

  To establish that $\varphi'$ is an automorphism, it remains to show
  that $(\varphi')^\omega=(\varphi^\omega)'$ fixes each
  generator~$s_i$. The product formula (\ref{eq:product-formula})
  shows that the coordinates of $\lambda_i,\lambda_{i,j}$ of
  $\varphi'(s_i)$ are given by row $i$ of the matrix
  $\mathbb{I}_2(\varphi)$ (mod~$p$). Since this matrix is assumed to
  be invertible mod~$p$, it follows from Lemma~\ref{l:I2-homomorphism}
  that $\mathbb{I}_2(\varphi^\omega)$ is the identity matrix mod~$p$.
  This allows us to conclude that indeed $(\varphi')^\omega(s_i)=s_i$
  for $i=1,\ldots,n$.

  Conversely, suppose that $(\varphi^\omega)'$ is the identity mapping
  of the finite semigroup $S=\Om{A_n}{CS}(\pv{Ab}_p)$. In particular,
  $\varphi'$ must be a surjective transformation of~$S$, thus it must
  contain elements in each $\green{L}$-class and in each
  $\green{R}$-class. Now, since $S$ is finite and $(\varphi^\omega)'$
  is the identity mapping, every element of $S$ is of the form $p_{\pv
    S,\pv{CS}(\pv{Ab}_p)}\bigl(\varphi^\omega(w)\bigr)$ for some word
  $w\in A_n^+$, and its $\green{R}$-class is completely determined by $p_{\pv
    S,\pv{CS}(\pv{Ab}_p)}\bigl(\varphi^\omega(a)\bigr)=(\varphi^\omega)'(a)$,
  where $a$ is the first letter of~$w$. Since the number of $\green{R}$-classes of $S$ and
  the number of letters are both equal to $n$, we conclude that
  $\varphi$ is left permutative. Similarly, $\varphi$ is right
  permutative. Hence, as we observed before, $T_2(\varphi)$ is a
  permutation matrix.

  Next, we claim that $F_1(\varphi)$ is also an invertible matrix.
  Indeed, the assumption that $(\varphi^\omega)'$ is the identity on
  $S$, implies that $(\varphi^{\pv{Ab}_p})^\omega$ is the identity on
  the group $\Om{A_n}{Ab}_p$, which is isomorphic with the additive
  group of the above vector space $F^n$ under the homomorphism that
  sends the generator $a_i$ to the $i$th canonical basis vector. Hence,
  the matrix of the linear transformation of $F^n$ corresponding to
  $\varphi^{\pv{Ab}_p}$, which is precisely $F_1(\varphi)$, must be
  invertible. Combining with the already established fact that
  $T_2(\varphi)$ is invertible, we conclude that the matrix
  $\mathbb{I}_2(\varphi)$ is invertible, which completes the proof.
\end{proof}

We are now ready to establish the main result of this subsection.

\begin{theorem}
  \label{t:lower-bound}
  Let $\bsigma$ be a primitive directive sequence with alphabet rank $n$ and
  let $\pv V$ be a pseudovariety of semigroups containing $\pv{LSl}$.
  \begin{enumerate}
  \item\label{item:lower-bound-1} If $\pv V$ contains some non
    orthodox simple semigroup, then the upper bound $n^2-n+1$ for the
    rank of $G_{\pv V}(\bsigma)$ given by
    Theorem~\ref{t:maximal-subgroup-in-image}\ref{item:maximal-subgroup-in-image-1}
    is optimal.
  \item\label{item:lower-bound-2} If $\pv V$ contains only orthodox
    simple semigroups and contains some nontrivial group, then the
    upper bound $n$ for the rank of $G_{\pv V}(\bsigma)$ given by
    Theorem~\ref{t:maximal-subgroup-in-image}\ref{item:maximal-subgroup-in-image-3}
    is optimal.
  \end{enumerate}
\end{theorem}

\begin{proof}
  \ref{item:lower-bound-1} Let $\sigma_n$ be the substitution over
  $A_n$ of Example~\ref{eg:sigma}, whose matrix
  $\mathbb{I}_2(\sigma_n)$ is invertible. A substitution over a finite
  alphabet is called \emph{unimodular} when its $1$-frequency matrix
  has determinant $\pm 1$. The third author showed that if $\varphi$
  is a unimodular primitive substitution, then $X(\varphi)$ is
  aperiodic~\cite[Proposition 5.6]{Goulet-Ouellet:2022d}, and so in
  that case the constant directive sequence $(\varphi,\varphi,\ldots)$
  is recognizable by Mossé's theorem~\cite[Theorem
  2.4.34]{Durand&Perrin:2022}. Hence, the constant directive sequence
  $\bsigma_n=(\sigma_n,\sigma_n,\ldots)$ is $\pv{S}$-saturating by
  Theorem~\ref{t:surjectivity}, and so also $\pv{V}$-saturating by
  Proposition~\ref{p:saturation-inheritance}.

  By Example~\ref{e:omega-power-V-compression} and
  Corollary~\ref{c:letters-mapped-to-Js}, $\sigma_n^\omega(A_n)$ is
  contained in $J_{\pv V}(\bsigma_n)=J_{\pv V}(\sigma_n)$ and
  generates a closed subsemigroup $S$ which contains maximal subgroups
  of~$J_{\pv V}(\sigma_n)$. As $\sigma_n$ is stable, $S$ is a simple
  semigroup by Theorem~\ref{t:contraction-stable}.

  By Theorem~\ref{t:Rasin}, $\pv V$ contains the pseudovariety
  $\pv{CS}(\pv{Ab}_p)$ for some prime~$p$. By
  Proposition~\ref{p:suficient-for-CSAbp-invertible}, the natural
  projection $p_{\pv V,\pv{CS}(\pv{Ab}_p)}$ maps each
  $\sigma_n^\omega(a_i)$ to $\iota_{\pv{CS}(\pv{Ab}_p)}(a_i)$. In
  particular, each maximal subgroup of $S$ is mapped onto a maximal
  subgroup of~$\Om{A_n}{\pv{CS}(\pv{Ab}_p)}$. By
  Theorem~\ref{t:free-CS(H)}, the rank of the latter is $n^2-n+1$ and,
  therefore, so is the rank of the former.

  \ref{item:lower-bound-2} The proof can be achieved with the same
  kind of arguments using instead the unimodular proper primitive substitution $\sigma'_n$
  defined by $\sigma'_n(a_1)=a_1a_2$ and $\sigma_n'(a_i)=a_1a_ia_2$
  ($i=2,\ldots,n$). To complete the
  proof, it suffices to consider the projection $p_{\pv V,\pv{Ab}_p}$,
  where $p$ is a prime such that $\mathbb{Z}/p\mathbb{Z}$ belongs to $\pv V$. The
  restriction of $p_{\pv V,\pv{Ab}_p}$ to the maximal subgroup
  $\img_{\pv{V}}\bigl((\sigma'_n)^\omega\bigr)$ of $J_{\pv
    V}(\sigma'_n)$ is onto, the image being the rank~$n$ elementary
  Abelian $p$-group $\Om n{Ab}_p$. Hence, the group $G_{\pv
    V}(\sigma'_n)$ has rank~$n$.
\end{proof}

For the pseudovariety $\pv V=\pv S$, the following example gives an
alternative method to the proof of sharpness of the bound $n$ obtained
in Theorem~\ref{t:maximal-subgroup-in-image} for left proper directive
sequences.

\begin{example}
  If $X$ is a \emph{dendric} shift space (also called \emph{tree}
  shift) over an alphabet of size $n$, then $G_{\pv S}(X)$ is a free
  profinite group of rank $n$~\cite{Almeida&ACosta:2016b}. It is known
  that $X=X(\bsigma)$ for a primitive directive sequence of alphabet
  rank $n$ that is proper and recognizable, see the discussion
  in~\cite[Example 6.9]{Berthe&Steiner&Thuswaldner&Yassawi:2019}.
\end{example}

\section*{Acknowledgments}

The first author acknowledges partial support by CMUP (Centro de
Matemática da Universidade do Porto), member of LASI (Intelligent
Systems Laboratory), which is funded through Portuguese funds through
FCT (Fundação para a Ciência e a Tecnologia, I. P.) under the project
UIDB/00144/2025 (\url{https://doi.org/10.54499/UID/00144/2025}).

The second author was financially supported by the Fundação para a Ciência e a Tecnologia (Portuguese Foundation for Science and Technology) under the scope of the projects UID/00324/2025 (\url{https://doi.org/10.54499/UID/00324/2025}) (Centre for Mathematics of the University of Coimbra).

The third author was supported by the Czech Technical University
Global Postdoc Fellowship program.

\inputencoding{latin1}
\bibliographystyle{amsplain}
\bibliography{../../biblio/sgpabb,../../biblio/ref-sgps}

\end{document}